\documentclass{amsart}

\usepackage{amsmath}
\usepackage{amsfonts}
\usepackage{amssymb}
\usepackage{amsthm}
\usepackage{comment}
\usepackage{epsfig}
\usepackage{psfrag}
\usepackage{mathrsfs}
\usepackage{amscd}
\usepackage[all]{xy}
\usepackage{rotating}
\usepackage{lscape}
\usepackage{amsbsy}
\usepackage{verbatim}
\usepackage{moreverb}
\usepackage{color}
\usepackage{bbm}
\usepackage{eucal}
\usepackage{stmaryrd}
\usepackage{cleveref}

\usepackage{pdfpages}

\usepackage{float}

\usepackage{tikz-cd}
\usetikzlibrary{patterns,shapes.geometric,arrows,decorations.markings}
\usepackage{tikz-3dplot}

\usepackage{caption}
\usepackage{subcaption}

\colorlet{lightgray}{black!15}

\tikzset{->-/.style={decoration={
  markings,
  mark=at position .5 with {\arrow{>}}},postaction={decorate}}}
\tikzset{midarrow/.style={decoration={
    markings,
    mark=at position {#1} with {\arrow{>}}},postaction={decorate}}}

\newtheorem{theorem}{Theorem}[subsection]
\newtheorem{prop}[theorem]{Proposition}
\newtheorem{lemma}[theorem]{Lemma}
\newtheorem{cor}[theorem]{Corollary}

\newtheorem{fact}[theorem]{Fact}

\theoremstyle{definition}
\newtheorem{definition}[theorem]{Definition}

\newtheorem{observation}[theorem]{Observation}

\newtheorem{terminology}[theorem]{Terminology}
\newtheorem{remark}[theorem]{Remark}
\newtheorem{example}[theorem]{Example}

\newtheorem{notation}[theorem]{Notation}
\newtheorem{l.notation}[theorem]{Local Notation}

\newtheorem{convention}[theorem]{Convention}

\theoremstyle{remark}

\definecolor{orange}{rgb}{.95,0.5,0}
\definecolor{light-gray}{gray}{0.75}
\definecolor{brown}{cmyk}{0, 0.8, 1, 0.6}
\definecolor{plum}{rgb}{.5,0,1}

\DeclareMathOperator{\pr}{\mathsf{pr}}

\DeclareMathOperator{\ev}{\mathsf{ev}}

\DeclareMathOperator{\TwAr}{\sf TwAr}

\DeclareMathOperator{\sSet}{\sf sSet}

\DeclareMathOperator{\Aut}{\sf Aut}
\DeclareMathOperator{\colim}{{\sf colim}}

\DeclareMathOperator{\limit}{{\sf lim}}

\DeclareMathOperator{\hocolim}{\sf hocolim}
\DeclareMathOperator{\holim}{\sf holim}
\DeclareMathOperator{\Hom}{\sf Hom}

\DeclareMathOperator{\Fun}{{\sf Fun}}

\DeclareMathOperator{\Map}{\sf Map}

\DeclareMathOperator{\cMap}{\sf cMap}

\DeclareMathOperator{\Cat}{{\sf Cat}}

\DeclareMathOperator{\Ar}{{\sf Ar}}

\DeclareMathOperator{\Diff}{{\sf Diff}}
\DeclareMathOperator{\BDiff}{{\sf BDiff}}

\DeclareMathOperator{\op}{\mathsf{op}}

\DeclareMathOperator{\cBun}{{\sf c}\cB\mathsf{un}}

\DeclareMathOperator{\sd}{\mathsf{sd}}

\DeclareMathOperator{\Top}{\mathsf{Top}}

\DeclareMathOperator{\Emb}{\mathsf{Emb}}

\DeclareMathOperator{\spaces}{\mathsf{Spaces}}
\DeclareMathOperator{\Spaces}{\spaces}

\DeclareMathOperator{\sm}{\mathsf{sm}}

\DeclareMathOperator{\fr}{\sf fr}

\DeclareMathOperator{\Bord}{\sf Bord}
\DeclareMathOperator{\dBord}{\delta\sf Bord}
\DeclareMathOperator{\tBord}{\sf Bord}

\DeclareMathOperator{\Th}{\mathsf{Th}}

\DeclareMathOperator{\BO}{\sf BO}

\DeclareMathOperator{\BSO}{\sf BSO}

\DeclareMathOperator{\oo}{\infty}

\newcommand{\lag}{\langle}
\newcommand{\rag}{\rangle}

\newcommand{\w}{\widetilde}

\newcommand{\ov}{\overline}

\newcommand{\ra}{\rightarrow}
\newcommand{\la}{\leftarrow}
\newcommand{\xra}{\xrightarrow}
\newcommand{\xla}{\xleftarrow}

\def\cB{\mathcal B}
\def\cE{\mathcal E}
\def\cI{\mathcal I}
\def\cP{\mathcal P}

\def\cX{\mathcal X}

\def\BB{\mathbb B}\def\DD{\mathbb D}

\def\RR{\mathbb R}\def\SS{\mathbb S}

\def\sB{\mathsf B}\def\sC{\mathsf C}\def\sD{\mathsf D}
\def\sH{\mathsf H}

\def\sN{\mathsf N}
\def\sT{\mathsf T}

\def\bDelta{\mathbf\Delta}

\DeclareMathOperator{\Obj}{\mathsf{obj}}

\DeclareMathOperator{\PShv}{\mathsf{PShv}}

\DeclareMathOperator{\id}{\sf id}

\DeclareMathOperator{\Mor}{\sf mor}

\DeclareMathOperator{\Sets}{\sf Sets}

\DeclareMathOperator{\Un}{\sf Un}

\newcommand{\bit}[1]{\textbf{\textit{#1}}}

\DeclareMathOperator{\lacts}{\curvearrowright}

\DeclareMathOperator{\Sub}{\sf Sub}

\DeclareMathOperator{\dSub}{\delta{\sf Sub}}
\DeclareMathOperator{\tSub}{{\sf Sub}}

\begin{document}

\title{A parametrized Pontryagin--Thom theorem}

\author{David Ayala \& John Francis}

\address{Department of Mathematics\\Montana State University\\Bozeman, MT 59717}
\email{david.ayala@montana.edu}
\address{Department of Mathematics\\Northwestern University\\Evanston, IL 60208}
\email{jnkf@northwestern.edu}
\thanks{ 
DA was partially supported by the National Science Foundation under award 1945639; DA was sponsored in part by the Air Force Research Laboratory under Agreement Number FA8750-24-1-1019. The U.S. Government is authorized to reproduce and distribute reprints for government purposes notwithstanding any copyright notation thereon. Any opinions, findings, and conclusions or recommendations expressed in this material are those of the authors and do not necessarily reflect the view of the funder.
JF was supported by the National Science Foundation under award 1812057. 
This material is based upon work supported by the National Science Foundation under award 1928930 while the authors participated in a program supported by the Mathematical Sciences Research Institute during the Summer
of 2022 in partnership with the Universidad Nacional Autonoma de Mexico.
}

\begin{abstract}
We prove a space-level enhancement of the Pontryagin--Thom theorem, identifying the space of maps from a manifold to a Thom space with a moduli space of submanifolds.
\end{abstract}



\maketitle

\tableofcontents

\section{Introduction}

Let $M$ be an $(n+k)$-dimensional manifold with boundary.
Suppose $M$ is finitary, which is to say it may be identified as a compact manifold with corners remove some faces.
Let $\xi = (E \to B)$ be a rank-$k$ vector bundle, with associated Thom space $\Th(\xi)$. The Pontryagin--Thom theorem identifies the set of homotopy classes of pointed maps from the one-point compactification $M^+$ to $\Th(\xi)$ with a set of embedded cobordism classes of codimension-$k$ submanifolds of $M$. 

This paper answers the following question. 
\begin{itemize}
    \item[]
    {\bf Question.}
    {\it Is there a space-level refinement of the Pontryagin--Thom theorem, in which the space of pointed maps $\Map_\ast(M^+,\Th(\xi))$ is homotopy equivalent to a moduli space of submanifolds of $M$?}
\end{itemize}

This moduli space of submanifolds involves the following topological data: A \bit{codimension-$\xi$ submanifold} of $M$ is a codimension-$k$ submanifold $W \subset M$ with boundary $\partial W = W \cap \partial M$ such that $W$ is transverse to the boundary $\partial M$, which is furthermore equipped with a map $g\colon W \ra B$ and an isomorphism between vector bundles over $W$
\[
\alpha\colon \nu_{W\subset M}
\xra{~\cong~}
g^\ast \xi
\]
from the normal bundle of $W\subset M$ to the pullback of $\xi$ along $g$.
For $(W_0,g_0,\alpha_0)$ and $(W_1,g_1,\alpha_1)$ compact codimension-$\xi$ submanifolds of $M$, an \bit{embedded cobordism} from the first to the second is a compact codimension-$\xi$ submanifold $(W,g,\alpha)$ of $[0,1]\times M$ such that, for each $i=0,1$, there are identities $W_i = W \cap ( \{i\} \times M)$ and $g_i = g_{|W_i}$ and $\alpha_i = \alpha_{|W_i}$.

Our answer to the main question involves, for each $p\geq 0$, the space $\Bord^\xi_p(M)$ of compact codimension-$\xi$ submanifolds of $\Delta^p\times M$,
organized as a simplicial topological space
\[
\Bord^\xi_\bullet(M)
\]
in a way introduced by Quinn~\cite{quinn.thesis}.
In particular, $\Bord^\xi_0(M)$ is a space of compact codimension-$\xi$ submanifolds of $M$, and $\Bord^\xi_1(M)$ is a space of codimension-$\xi$ embedded cobordisms between such. 
Forgetting the topology on the topological spaces $\Bord^\xi_p(M)$ defines a simplicial set $\dBord^\xi_\bullet(M)$.

The following is the main result of this paper. 
For a $X$ a pointed space, let $\cMap(M,X)$ denote the subspace of $\Map(M,X)$ consisting of the {\it compactly-supported} maps (i.e., maps $f:M\ra X$ that carry the complement of a compact subspace to the basepoint of $X$).

\begin{theorem}[Parametrized Pontryagin–Thom]\label{thm.main}
Let $\xi$ be a vector bundle of rank $k$.
Let $M$ be a smooth manifold with boundary.
There are equivalences between spaces,
\[
\bigl|\dBord_\bullet^{\xi}(M)\bigr|
~\simeq~
\bigl|\Bord_\bullet^{\xi}(M)\bigr|
~\simeq~
{\sf cMap}\bigl(M, \Th(\xi)\bigr)
~,
\]
which are covariantly functorial with respect to open embeddings in $M$.
Furthermore, the simplicial set $\dBord^\xi_\bullet(M)$ is a Kan complex, and the simplicial space $\Bord^\xi_\bullet(M)$ satisfies the Kan condition.

\end{theorem}

This result is closely related to a number of previous works, which we discuss in \S\ref{sec.precedingworks}. The principal new idea we introduce is a simplicial space of transverse maps, so as to interpolate between the mapping space and the moduli space of submanifolds:
\[
\xymatrix{
&\ar[dl]_-{\bf \S\ref{sec.transv}} \cMap^\pitchfork_\bullet\bigl(M, \Th(\xi)\bigr)\ar[dr]^-{\bf \S\ref{sec.PT}}\\
\cMap_\bullet\bigl(M, \Th(\xi)\bigr)&& \Bord_\bullet^\xi(M)~.
}
\]
In {\bf \S\ref{sec.transv}}, we establish a parametrized transversality theorem (Theorem~\ref{t80}), which states that the inclusion of transverse maps into all maps induces an equivalence on geometric realizations:
    \[
    \bigl|\cMap^\pitchfork_\bullet\bigl(M, \Th(\xi)\bigr)\bigr| 
    ~\simeq~ 
    \cMap\bigl(M, \Th(\xi)\bigr)~.
    \]
    The proof relies on extension properties of smooth transverse maps (Lemmas~\ref{t85} and~\ref{lemma.sm.sing}), similar to the Whitney extension theorem~\cite{whitney}.

 In {\bf \S\ref{sec.PT}}, we prove that the assignment sending a map $f\colon M\ra \Th(\xi)$ to the preimage $f^{-1}(B)\subset M$ of the zero-section $B\subset \Th(\xi)$ defines a levelwise equivalence between simplicial spaces,
    \[
    (-)^{-1}(B)\colon \cMap^\pitchfork_\bullet\bigl(M, \Th(\xi)\bigr)
    \longrightarrow
    \Bord_\bullet^\xi(M)
    ~,
    \]
    where the mapping space is endowed with the compact-open $\sC^\infty$ topology, and the space of submanifolds is topologized as a quotient of embedding spaces $\Emb(W,M)$ indexed by diffeomorphism classes of $W$. In particular the assignment $(-)^{-1}(B)$ is {\it continuous} with respect to these topologies, which is the most technical result in this paper (Lemma~\ref{t11}). The techniques involved are similar to Palais's proof of the parametrized isotopy extension theorem~\cite{palais}. Once this continuity is established, then the proof of levelwise equivalence is a straightforward implementation of Pontryagin--Thom theory in compactly parameter families, organized by simplicial homotopy theory:
    parametrized collapse maps (Theorem~\ref{t10}) define a homotopy inverse to the zero-section preimage map.

In order to implement the above arguments of \S\ref{sec.PT}, the preceding {\bf \S\ref{sec.Bord.xi}} constructs the simplicial space $\Bord^\xi_\bullet(M)$. The key technical result is Lemma~\ref{t50} which coherently matches normal bundles along the simplicial structure maps of $\Bord^k_\bullet(C)$; its proof occupies \S\ref{sec.normal}. The presentation of $\Bord^k_\bullet(M)$ as a simplicial topological space allows for the construction of the map in Lemma~\ref{t27} and the map~(\ref{e47}).

In {\bf \S\ref{sec.Kan.condition}}, we construct a simplicial set $\dBord_\bullet^\xi(M)$, which we prove is a Kan complex. We construct a natural map $\dBord_\bullet^\xi(M)\ra \Bord_\bullet^\xi(M)$ and prove that it is a trivial Kan fibration (Corollary~\ref{cor.dBord.Bord.kan}), which is a slight weakening of the condition that a simplicial topological space $\tBord_\bullet^\xi(M)$ is Reedy fibrant (see Remark~\ref{rem.Reedy}). This implies that their geometric realizations are equivalent, and so $\bigl|\dBord_\bullet^\xi(M)\bigr|$ also models the mapping space $\cMap(M,\Th(\xi))$.

\begin{remark}\label{rem.GH}
Although we have defined $\Bord_\bullet^\xi(M)$ as a simplicial space, we think of it as an $(\infty,\infty)$-category, in which the $r$-morphisms are $(n+r)$-dimensional submanifolds of $\Delta^r\times M$. Consequently, this putative $(\oo,\oo)$-category has noninvertible $r$-morphisms for all $r> 0$.
In the work~\cite{gepner.heine}, $(\infty,\infty)$-categories are presented as presheaves on a category $\w{\bDelta}$ that satisfy certain descent conditions.
The category $\w{\bDelta}$ is the \bit{oriental category} -- it is gaunt, locally finite, and it admits a monomorphism $\bDelta \hookrightarrow \w{\bDelta}$ from the simplex category that is an equivalence between sets of objects.
We expect the simplicial space $\Bord_\bullet^\xi(M) \in \PShv(\bDelta)$ naturally extends as a presheaf on $\w{\bDelta}$ and thereby presents an $(\infty,\infty)$-category.\footnote{Our expectation is founded on observing a natural extension of the functor $\bDelta^{\op} \xra{\Delta^\bullet} \cBun$, in the sense of~\cite{striation}, which selects the skeletally stratified simplices, to $\w{\bDelta}^{\op} \to \cBun$.}
The work~\cite{gepner.heine} does not impose univalence in their notion of of an $(\infty,\infty)$-category, and we likewise expect that the existence of non-trivial embedded h-cobordisms to mean that the $(\infty,\infty)$-category $\Bord^\xi_\bullet(M)$ will not be univalent complete. 
Fulfilling these expectations would be an interesting project, which we pose for an interested reader.
Premised on these expectation, we further expect the $(\infty,\infty)$-category presented by $\Bord_\bullet^\xi(M)$ to be a sequential colimit of the finite tangle $(\oo,r)$-categories defined in a standard way (e.g., in terms of $\boldsymbol{\Theta}_r$), at least in the case that $B \xra{\xi}\BO(k)$ is base-changed from a stable vector bundle $\w{B} \xra{\w{\xi}} \BO$ along $\BO(k) \to \BO$:
\[
\colim\Bigl(
\Bord^{\zeta_0}_0(\RR^k)
\to 
\Bord^{\zeta_1}_1(\RR^k)
\to 
\dots
\to 
\Bord^{\zeta_r}_r(\RR^k)
\to\cdots
\Bigr)
~\simeq~
\Bord^\xi_\bullet(\RR^k)
~,
\]
where each $\zeta_r$ is the base-change of $\w{B} \xra{\w{\xi}}\BO$ along $\BO(r) \to \BO \xra{-1} \BO$.

\end{remark}

\begin{remark}
    For $\xi$ a smooth vector bundle over an infinite-dimensional manifold, we expect the simplicial space $\Bord^\xi_\bullet(M)$ to have a refinement as a simplicial $\sC^\infty$-sheaf.  (See Remark~\ref{r.smoothness}.)
    Such a lift of $\Bord^\xi_\bullet(M)$ might be interesting in its own right; it might also simplify several arguments in this work, notably around the repeated use of smooth approximation.
\end{remark}

\subsection{Classical Pontryagin--Thom from parametrized Pontryagin--Thom}

Applying $\pi_0$ to the equivalence of Theorem~\ref{thm.main} recovers the following form of the classical Pontryagin--Thom theorem (e.g., see Proposition~4.1 of~\cite{brs}): 

\begin{theorem}[Pontryagin--Thom]
\label{t5}
    Let $M$ be a finitary smooth manifold with boundary, and let $\xi$ be a rank-$k$ vector bundle over $B$.
    There is a bijection between the set of embedded cobordism classes of compact codimension-$\xi$ compact submanifolds of $M$ with boundary and the set of pointed homotopy classes of pointed maps to the Thom space of $\xi$:
    \[
    \Bigl\{
    W \underset{\rm codim{\text -}\xi}\subset M
    \Bigl\}_{/\rm Emb~Cob}
    ~\cong~
    [M^+,\Th(\xi)]
    ~.
    \]
    
\end{theorem}
\begin{proof}[Proof of Theorem~\ref{t5} from Theorem~\ref{thm.main}]
Recall that, for $X_\bullet$ a simplicial space, there is a canonical bijection between sets
\[
\pi_0 |X_\bullet| 
~\cong~
(\pi_0 X_0)_{/\sim}
~,
\]
where $\sim$ is the equivalence relation generated by declaring the faces of each edge to be equivalent: $[d_0 e] \sim [d_1 e]$ for each $e\in X_1$.
In the case of the simplicial space $\Bord^\xi_\bullet(M)$, we identify the space of 0-simplices as a space of compact codimenion-$\xi$ submanifolds of $M$ with boundary.
So $\pi_0 \Bord^\xi_0(M)$ is the set of compact codimension-$\xi$ submanifolds of $M$ with boundary, up to isotopy between such.  
Similarly, $\Bord^\xi_1(M)$ is a space of compact codimension-$\xi$ submanifolds of $[0,1] \times M$ with corners.
For $i=0,1$, the $i^{th}$ face map $\Bord^\xi_1(M) \xra{d_i} \Bord^\xi_0(M)$ is given by $[0,1]\times M \supset W \mapsto W \cap (\{i\} \times M) \subset M$.
Consequently, we identify the set of path-components
\[
\pi_0 |\Bord^\xi_\bullet(M)|
~\cong~
\left\{
W \subset M
\right\}_{/\rm Emb~Cob}
\]
as the set of embedded cobordism classes of compact codimension-$\xi$ submanifolds of $M$. 
The result follows. 
\end{proof}

The stable Pontryagin--Thom theorem has been particularly central in algebraic topology.
Let $\boldsymbol{\xi}$ denote a {\it stable} vector bundle, by which we mean the following data:
\begin{itemize}

    \item a sequence of maps $B_0 \xra{i_0} B_1 \xra{i_1} B_2 \xra{i_2} \cdots\ra B= \varinjlim B_k$~;
    
    \item a  rank-$k$ vector bundle $\xi_k = (E_k \to B_k)$, for each $k \geq 0$;
    
    \item an isomorphism between vector bundles $\xi_k \oplus \epsilon^1_{B_k} \cong i_k^\ast \xi_{k+1}$ over $B_k$, for each $k \geq 0$.

\end{itemize}
These data determine, for each $k \geq 0$, a pointed map
\[
\Sigma \Th(\xi_k)
\simeq
\Th(\xi_k \oplus \epsilon^1_{B_k})
\longrightarrow
\Th(\xi_{k+1})
~.
\]
The \bit{Thom spectrum} of $\boldsymbol{\xi}$ is the spectrum $\Th(\boldsymbol{\xi})$ associated to this sequence of pointed maps.
The stabilization of the classical Pontryagin--Thom theorem can be formulated in the following terms: a stably $\boldsymbol{\xi}$-framed compact $n$-manifold with boundary over $M$ consists of:
\begin{itemize}
\item a compact $n$-manifold with boundary $W^n$;
\item a map $W^n\xra{g} B$;
\item a smooth map $W^n\xra{f} M$ transverse to $\partial M$ and for which $\partial W= f^{-1}(\partial M)$; and
\item an isomorphism $f^\ast \tau_M\cong \tau_W \oplus g^\ast \boldsymbol{\xi}$ between stable vector bundles.
\end{itemize}
Taking the direct limit of the Pontryagin--Thom theorem for $M\times\RR^r$, as $r\mapsto\infty$, implies the following (e.g., see Theorem 5.1 of~\cite{brs}).

\begin{theorem}[Stable Pontryagin--Thom]\label{theorem.stable.PT}
\label{t5'}
    Let $n,k \geq 0$.
    Let $M$ be a finitary smooth $(n+k)$-manifold with boundary, and let $\boldsymbol{\xi}$ be as above.
    There is a bijection between the set of stably $\boldsymbol{\xi}$-framed cobordism classes of stably $\boldsymbol{\xi}$-framed compact $n$-manifolds with boundary over $M$ and the $k^{th}$ reduced cohomology of $M^+$ with coefficients in the Thom spectrum $\Th(\boldsymbol{\xi})$:
    \[
    \Bigl\{
    W^n \to M
    \Bigr\}_{/\rm Cob}
    ~\cong~
    \widetilde\sH^{k}\bigl(M^+, \Th(\boldsymbol{\xi})\bigr)
    ~.
    \]

\end{theorem}

Stabilizing Theorem~\ref{thm.main} gives a parametrized version of the stable Pontryagin--Thom theorem.
For each choice of $k \geq 0$ and $M$, the data $\boldsymbol{\xi}$ determine a sequence of morphisms between simplicial spaces
\[
\Bord^{\xi_k}_\bullet(M \times \RR^0)
\to 
\Bord^{\xi_{k+1}}_\bullet(M \times \RR^1)
\to 
\Bord^{\xi_{k+2}}_\bullet(M \times \RR^2)
\to 
\cdots
~,
\]
as well as a sequence of morphisms between simplicial sets $\dBord^{\xi_k}_\bullet(M \times \RR^0)
\to 
\dBord^{\xi_{k+1}}_\bullet(M \times \RR^1)
\to 
\cdots
$,
which are implemented by the inclusions
\[
\Delta^\bullet \times M \times \RR^r = \Delta^\bullet \times M \times \RR^r \times \{0\} 
~\hookrightarrow~ 
\Delta^\bullet \times M \times \RR^{r} \times \RR^1 = \Delta^\bullet \times M \times \RR^{r+1}
~.
\]
Denote the colimit simplicial space
\[
\Bord^{\Sigma^k \boldsymbol{\xi}}_\bullet(M)
~:=~
\underset{r \geq 0} \colim
\Bord^{\xi_{k+r}}_\bullet(M \times \RR^r)
~,
\]
and the colimit simplicial set $\dBord^{\Sigma^k \boldsymbol{\xi}}_\bullet(M) := \underset{r \geq 0} \colim
\Bord^{\xi_{k+r}}_\bullet(M \times \RR^r)$.
Evidently, for each $r \geq 0$, framed embeddings among finite disjoint unions of $\RR^r$ defines an $\cE_r$-algebra structure on the simplicial space $\Bord^{\xi_{k+r}}_\bullet(M \times \RR^r)$.
As detailed in~\S4.1 of~\cite{bord}, these structures determine the structure of an $\cE_\infty$-algebra on the simplicial space $\Bord^{\Sigma^k \boldsymbol{\xi}}_\bullet(M)$.
Denote the $k^{th}$ space of the Thom spectrum
\[
\Omega^{\infty} \Sigma^k \Th(\boldsymbol{\xi}) := \underset{r\geq 0}\colim \ \Omega^r\Th(\xi_{k+r})~,
\]
which is an $\cE_\infty$-algebra in $\Spaces$.

\begin{cor}[Parametrized stable Pontryagin--Thom]\label{cor.param.stable.PT}
\label{t13}
Let $M$ be a finitary manifold with boundary, and let $\boldsymbol{\xi}$ be a stable vector bundle.
For each $k \geq 0$, there are equivalences between spaces,
    \[
    \bigl| \dBord^{\Sigma^k \boldsymbol{\xi}}_\bullet(M) \bigr|
    ~\simeq~
    \bigl| \Bord^{\Sigma^k \boldsymbol{\xi}}_\bullet(M) \bigr|
    ~\simeq~
    \Map_\ast\bigl(
    M^+ , \Omega^{\infty} \Sigma^k \Th(\boldsymbol{\xi}) \bigr)
    ~,
    \]
    the second of which is one between $\cE_\infty$-algebras.
\end{cor}

The stable Pontryagin--Thom Theorem~\ref{theorem.stable.PT} follows from Corollary~\ref{cor.param.stable.PT} by applying path-components.

\begin{remark}
    Corollary~\ref{t13} gives a Kan complex of stably-framed iterated cobordisms whose geometric realization is  $\Omega^\infty \SS$, the zeroth space of the sphere spectrum.
    That is, in the case that $\boldsymbol{\xi}=\fr$ codifies stable normal framings (i.e., $B_k = \ast$ for all $k\geq 0$), Corollary~\ref{t13} gives a Kan complex $\dBord^{\fr}_\bullet(M)$ of iterated stably-framed cobordisms whose geometric realization  $|\dBord^{\fr}_\bullet(M) | \simeq \Omega^\infty \SS$ is equivalent with the zeroth space of the sphere spectrum.
\end{remark}

\begin{proof}[Proof of Corollary~\ref{t13} from Theorem~\ref{thm.main}]
For each $r\geq 0$, the morphism between simplicial sets $\dBord^{\xi_{k+r}}_\bullet(M \times \RR^r) \to \dBord^{\xi_{k+r+1}}_\bullet(M \times \RR^{r+1})$ is evidently a monomorphism.
Consequently, the canonical map between spaces $\underset{r\geq 0} \colim |\dBord^{\xi_{k+r}}_\bullet(M)| \to | \dBord^{\boldsymbol{\xi}}_\bullet(M)|$ is an equivalence.  
Meanwhile, the canonical map between spaces $\underset{r\geq 0} \colim |\Bord^{\xi_{k+r}}_\bullet(M)| \to | \Bord^{\boldsymbol{\xi}}_\bullet(M)|$ is an equivalence because the functor $\Fun(\bDelta^{\op},\Spaces) \xra{|-|} \Spaces$ commutes with colimits.
In this way, the first equivalence follows from the first equivalence of Theorem~\ref{thm.main}, by taking sequential colimits.

The second equivalence is a composition of the following sequence of routine equivalences between spaces:
\begin{eqnarray*}
    \bigl| \Bord^{\Sigma^k \boldsymbol{\xi}}_\bullet(M) \bigr|
&
    \simeq
&
    \underset{r \geq 0} \colim
\bigl|\Bord^{\xi_{k+r}}_\bullet(M \times \RR^r)\bigr|
\\
&
  \simeq
&
    \underset{r \geq 0} \colim
    \cMap\bigl(M \times \RR^r, \Th(\xi_{k+r})\bigr)
\\
&
\simeq
&
    \underset{r \geq 0} \colim
    \Map_\ast\bigl((M\times\RR^r)^+, \Th(\xi_{k+r})\bigr)
\\
&
\simeq
&
    \underset{r \geq 0} \colim
    \Map_\ast\bigl(\Sigma^r M^+, \Th(\xi_{k+r})\bigr)
\\
&
    \simeq
&
    \underset{r \geq 0} \colim
    \Map_\ast\bigl(M^+, \Omega^r\Th(\xi_{k+r})\bigr)
\\
&
       \simeq
&
    \Map_\ast\Bigl(M^+, \underset{r \geq 0} \colim \ \Omega^r\Th(\xi_{k+r})\Bigr)
\\
&
    \simeq
&
    \Map_\ast\bigl(M^+, \Omega^{\infty} \Sigma^k \Th(\boldsymbol{\xi})\bigr)~.
\end{eqnarray*}
The first equivalence follows by intercommutation of colimits: in particular, geometric realizations commute with direct limits. The second equivalence is Theorem~\ref{thm.main} applied to $M\times\RR^r$ and $\xi_{k+r}$.
For the third equivalence, since $M\times\RR^r$ is locally compact, there is a canonical map 
\begin{equation}\label{x30}
\cMap(M\times\RR^r,\Th(\xi)) \longrightarrow \Map_\ast((M\times\RR^r)^+,\Th(\xi))
~.
\end{equation}
We postpone to the next paragraph a proof that~(\ref{x30}) is an equivalence. The fourth equivalence follows since the one-point compactification of the product is equivalent to the smash product of the one-point compactifications: $(M\times\RR^r)^+\simeq M^+\wedge (\RR^r)^+\simeq \Sigma^r M^+$.
The fifth equivalence is the suspension-loop adjunction. The penultimate equivalence again uses that $M$ is finitary, so that $M^+$ is a compact object in the $\infty$-category $\Spaces$. The last equivalence is definitional.

We now prove~(\ref{x30}) is an equivalence.
Consider the poset ${\sf Cpt}(M)$ of compact subspaces of $M$ ordered by inclusion.
Recognize the subspace 
\[
\Map(M,\Th(\xi))
~\supset~
\cMap(M,\Th(\xi))
~\simeq~
\underset{K \in {\sf Cpt}(M)}\colim  \Map_K(M,\Th(\xi)) 
\]
as a colimit, where $\Map_K(M,\Th(\xi)) \subset \cMap(M,\Th(\xi))$ is the subspace of consisting of those smooth compactly-supported maps $f\colon M \ra \Th(\xi)$ such that ${\sf Supp}(f) := \ov{f^{-1}(E)} \subset K$.
Consequently, the map~(\ref{x30}) may be implemented as precomposition by the ${\sf Cpt}(M)$-indexed system of collapse-maps 
\[
M^+
\longrightarrow
M \underset{M\smallsetminus C} \amalg \ast
~.
\]
Next, using that $M$ is finitary, choose a compact manifold $\ov{M}$ with corners such that $M$ is identified as $\ov{M}$ remove some faces.  
Denote $\partial M := \ov{M} \smallsetminus M$.
Consider the full subposet $\sC \subset {\sf Open}(\ov{M})$ consisting of the collar-neighborhoods of $\partial M \subset \ov{M}$.
This collection $\sC$ of neighborhoods is a basis for the topology about $\partial M \subset \ov{M}$. 
In particular, $\sC$ is cofiltered, and the functor between posets $\sC^{\op} \xra{C\mapsto M \smallsetminus C} {\sf Cpt}(M)$ is final, as a functor between $\infty$-categories. 
In particular, we have a canonical identification between spaces:
\[
\underset{C \in \sC^{\op}}\colim \Map_{M \smallsetminus C}(M,\Th(\xi)) 
\xra{~\simeq~}
\underset{K \in {\sf Cpt}(M)}\colim \Map_K(M,\Th(\xi)) 
~\simeq~
\cMap(M,\Th(\xi))
~.
\]
Now, for $C \in \sC$ a collar-neighborhood of $\partial M \subset \ov{M}$, a choice of diffeomorphism $C \cong \partial M \times [0,1)$ reveals that the collapse-map $M^+ \to M \underset{M\smallsetminus C} \amalg \ast$ implements an equivalence between spaces:
\[
\Map_{M \smallsetminus C}(M,\Th(\xi))
\xra{~\simeq~}
\Map_\ast(M^+ , \Th(\xi))
~.
\]
Using that $\sC^{\op}$ is filtered, this implies the canonical map
\[
\Map_\ast(M^+,\Th(\xi))
\xra{~\simeq~}
\underset{C^\circ \in \sC^{\op}}\colim \Map_{M\smallsetminus C}(M,\Th(\xi))
\]
is an equivalence.
We conclude that the map~(\ref{x30}) is an equivalence, as desired.

For each $r \geq 0$, both $\Bord^{\xi_r}_\bullet(N)$ and $\Map_\ast\left( N^+ , \Th(\xi_r) \right)$ are functorial among open embeddings in the argument $N$, carrying disjoint unions to products and isotopy equivalences to equivalences.  
Therefore, both $\Bord^{\xi_r}_\bullet(M \times \RR^r)$ and $\Map_\ast\left( (M \times \RR^r)^+ , \Th(\xi_r) \right) \simeq \Map_\ast\left( M^+, \Omega^r \Th(\xi_r) \right)$ have canonical structures of $\cE_r$-algebras.  
The naturality of Theorem~\ref{thm.main} therefore implies its equivalence between spaces $\left|\Bord^{\xi_r}_\bullet(M \times \RR^r)\right| \simeq \Map_\ast\left( M^+, \Omega^r \Th(\xi_r) \right)$ lifts as one between $\cE_r$-algebras.
Consequently, the second equivalence is an equivalence of $\cE_\infty$-algebras.

\end{proof}

\subsection{Comparison to preceding works}\label{sec.precedingworks}
Our parametrized Pontryagin--Thom theorem is distinct from the enhancements of Pontryagin--Thom of Galatius--Madsen--Tillmann--Weiss~\cite{GMTW}, Ayala~\cite{phd}, Randal-Williams~\cite{oscar}, and Schommer-Pries~\cite{chrisSP} in the following two ways.
First, while both we and they recover the set of cobordism classes of $n$-manifolds as $\pi_0$ of a classifying space, these spaces are substantially different. 
We expect our space is the classifying space of an $(\oo,\oo)$-category (see Remark~\ref{rem.GH}), in which the $r$-morphisms are $(n+r)$-dimensional manifolds with corners.
For \cite{GMTW},~\cite{phd},~\cite{oscar},~\cite{chrisSP}, the space is a classifying space of an $(\oo,m)$-category for necessarily finite $m$: its objects are $n$-manifolds, and its $k$-morphisms are $(k-1)$-parameter families of diffeomorphisms for $k\geq m$.
Second, those works employ Graeme Segal's technique of \emph{scanning maps} (as in~\cite{segal},~\cite{mcduff1},~\cite{mcduff2}), and consequently are compatible with \emph{tangential} structures on the manifold $M$. Our method uses the more classical Pontryagin--Thom collapse maps, and hence is compatible with \emph{normal} structures. Stably these are the same (hence both our theorems imply the stable Pontryagin--Thom theorem), but unstably they are distinct and thus yield different unstable parametrized versions of the Pontryagin--Thom theorem.

The essential idea of building a simplicial set (or at least semi-simplicial) whose $p$-simplices are codimension-$k$ submanifolds with corners in $\Delta^p$ dates at least to Quinn's thesis: See \S1.4~\cite{quinn.thesis}, as well as a discussion of this history by Laures--McClure in \S1.1~\cite{laures.mcclure}. In particular, Quinn~\cite{quinn} constructed a semi-simplicial set (i.e., a $\Delta$-set) similar to the stabilization of $\dBord_\bullet(M)$ in the argument $M$, in the case in which $\xi$ is the universal bundle over $\BSO$. Using the Pontryagin--Thom theorem, Quinn deduces that the homotopy groups of the realization of this semi-simplicial set are the bordism homology groups of $M$, a result closely related to Theorem~5.1 of Buoncristiano--Rourke--Sanderson~\cite{brs}.

In the setting of topological manifolds, Laures--McClure~\cite{laures.mcclure} and Randal-Williams~\cite{oscar2} both consider bordism-type models for the Thom spectrum of $\Top$. In Appendix~B~\cite{laures.mcclure}, prove a topological version of Quinn's theorem, that ${\sf MTop}$ is the realization of a Quinn-type semi-simplicial set. Randal-Williams constructs a semi-simplicial set ${\sf Mock}(d,n)$ which is a topological analogue of our $\dBord_\bullet^{\gamma_{n-d}}(\RR^n)$, where $\gamma_{n-d}$ is the universal bundle over $\BSO(n-d)$. 
In \S4.3~\cite{oscar2}, he outlines how, in the large $n$ limit, topological transversality would prove an equivalence between the realization $|{\sf Mock}(d)|$ and the space $\Omega^{\oo+d}{\sf MSTop} = \Omega^{\infty+d} \Th(\boldsymbol{\gamma})$.

Both Quinn and Randal-Williams include a {\it collaring} condition in the definition of their semi-simplicial sets of manifolds, which is convenient for verifying the Kan condition. In block automorphism spaces of manifolds, the inclusion of this collaring condition is standard (see~\cite{blr}, \S1.3 of~\cite{krannich}, Definition~2.2.1 of~\cite{HLLR}) because it makes the Kan condition straightforward to establish. The downside is that it complicates the definition of degeneracies, resulting in only semi-simplicial sets, rather than simplicial sets. 
However, this collaring condition is not expected to affect the homotopy type: See Theorem~2.60~\cite{gimm}. Like in~\cite{gimm}, the present work does not include a collaring condition in our definition of the spaces $\Bord_\bullet^\xi(M)$, which allows us to construct simplicial, rather than merely semi-simplicial, objects. We nevertheless verify the Kan condition for the simplicial space $\Bord_\bullet^\xi(M)$ as well as the simplicial set $\dBord_\bullet^\xi(M)$.

Similarly, in the study of smooth singular complexes, it is also common to include a collaring condition. See, for example,~\cite{kihara} and~\cite{CW}. See~Remark~1.7 and Warning~2.24 of~\cite{oh.hiro} for a discussion of diffeologies on simplices and the effect on the Kan condition for smooth singular complexes. 
Similar to our definition of $\Bord^\xi_\bullet(M)$ and $\dBord^\xi_\bullet(M)$, we do not impose a collaring condition in the definition of the simplicial smooth mapping spaces $\Map^{\sm}_\bullet(M,\Th(\xi))$ and $\Map^{\pitchfork}_\bullet(M,\Th(\xi))$, but still verify the Kan condition.

\subsection*{Acknowledgments}
We thank our advisors, Ralph L. Cohen and Mike Hopkins, for teaching us Pontryagin--Thom theory.
DA thanks the Wednesday Seminar at Montana State University for building community around the subject.  
JF thanks Josh Cynamon for coworking during the writing of this paper.
We thank Søren Galatius, Sam Gunningham, Manuel Krannich, Oscar Randal-Williams, and Hiro Lee Tanaka for helpful comments on this paper.

\subsection*{Terminology}
    \begin{itemize}
    
        \item The term \bit{manifold}, and \bit{manifold with boundary}, refers to a smooth manifold, and a smooth manifold with boundary, respectively. 
        The term \bit{manifold with corners} refers to a smooth manifold with corners in the sense of Definition 2.1 of~\cite{joyce}.
        For $M$ an $n$-manifold with corners, its underlying topological space is canonically stratified by the poset $[n]=\{0<\dots<n\}$, such that, for $i\in [n]$, the (locally closed) $i$-stratum is the locus of those $x\in M$ for which there exists an open embedding $\RR_{\geq 0}^i \times \RR^{n-i} \xra{\varphi} M$ such that $\varphi(0)=x$.
        For $i\in [n]$, an $i$-\bit{face} is a connected component $F \subset M$ of the $i$-stratum.

                \item A \bit{space} is an object in the $\oo$-category of spaces, $\Spaces$. A \bit{topological space} is an object of the ordinary category of topological spaces, $\Top$. There is a functor
                \[
                \Top \ra \Spaces
                \]
                which associates to a topological space its underlying space. We will use the same notation for both a topological space and its underlying space. (E.g., $\Sub^\xi(M)$ is defined as a space in Definition~\ref{d5} and as a topological space through Proposition~\ref{t88}.)

                \item For $\chi = (E \to B)$ a vector bundle, we often denote its total space as $E(\chi) := E$ and its base as $B(\chi) := B$.

    \end{itemize}

\section{Parametrized transversality}\label{sec.transv}

Throughout this section, we fix the following.
$M$ is a manifold with boundary.
$\xi = (E \to B)$ is a smooth rank-$k$ vector bundle.
$C$ is a manifold with corners. 
$X$ is a pointed topological space; the complement $X\smallsetminus \ast$ of the base point has the structure of a smooth manifold; and $X_0\subset X\smallsetminus\ast$ is a properly embedded smooth submanifold whose closure in $X$ does not contain the base point.

\subsection{Spaces of compactly-supported maps}

\begin{definition}
    A \bit{compactly-supported map} from $C$ to $X$ is a map $f\colon C \ra X$ whose support ${\sf Supp}(f) := \ov{f^{-1}(X \smallsetminus \ast)} \subset C$ is compact. 
    The topological space
    \[
    \cMap(C,X)
    \]
    is the set of compactly-supported maps from $C$ to $X$, endowed with the compact-open topology.
    \end{definition}

    \begin{definition}\label{d40}
    A \bit{smooth compactly-supported map} from $C$ to $X$ is a compactly-supported map $f\colon C \ra X$ such that the restriction $f_|\colon f^{-1}(X\smallsetminus \ast) \ra X\smallsetminus \ast$ is smooth.
    The topological space of smooth compactly-supported maps from $C$ to $X$
     \[
    \cMap^{\sm}(C,X)
    \]
    is the set of smooth compactly-supported maps endowed with the topology generated by the collection of subsets
    \[
    N(f, K,K',(U,\varphi), (V,\psi), V', \epsilon)
    ~\subset~
    \cMap^{\sf sm}(C,X)
    ~,
    \]
    given as follows.  
    \begin{itemize}
        \item[] Let $\{(U, \varphi)\}$ be a member of the atlas for $C$, and let $K \subset K' \subset U$ be concentric compact subspaces contained in $U$.
        Let $\{(V,\psi)\}$ be a member of the atlas for $X\smallsetminus \ast$, and let $V \subset V' \subset X$ be an open subset containing $V$.
        Let $f\in\cMap^{\sm}(C,X)$ such that $f(K')\subset V'$ and $f(K)\subset V$. 
        Let $\epsilon>0$.
        The subset $N(f, K, K', (U,\varphi), (V,\psi), V', \epsilon)$ consists of those elements $g\in \cMap^{\sf sm}(C,X)$ such that
    \begin{itemize}
        \item $g(K')$ is contained in $V'$;
        \item $g(K)$ is contained in $V$;
        \item $|\!|D^{(r)}_x(\psi f\varphi^{-1})-D^{(r)}_x(\psi g\varphi^{-1}) |\!|<\epsilon$ for all $x\in K$.
    \end{itemize}
    \end{itemize}
\end{definition}

\begin{remark}
    Given a choice of Riemannian metric on $X\smallsetminus\ast$, a sequence $f_n$ in $\cMap^{\sf sm}(C,X)$ converges to $f$ if and only if: it converges uniformly on compact subspaces; for each $x\in M$ such that $f(x)\neq \ast$, the derivatives $D^{(r)}f_n(x)$ converge uniformly to $D^{(r)}f(x)$.
\end{remark}
\begin{definition}\label{d41}
 A \bit{transverse compactly-supported map} from $C$ to $X$ is a smooth compactly-supported map $f\colon C \ra X$ such that for each face $F \subset C$, the smooth map $f_|\colon F \cap f^{-1}(X \smallsetminus \ast ) \ra X \smallsetminus \ast$ is transverse to $X_0$.
    The topological subspace
     \[
    \cMap^{\pitchfork}(C,X)
    ~\subset~
    \cMap^{\sf sm}(C,X)
    \]
    consists of the transverse compactly-supported maps.
\end{definition}

\subsection{Simplicial spaces of compactly-supported maps}

We recall some standard cosimplicial objects.
Let $p\geq 0$.
The \bit{topological $p$-simplex} and the \bit{extended $p$-simplex} are the respective subspaces:
\[
\Delta^p
~:=~
\left\{
\{0,\dots,p\} \xra{t} [0,1]
\mid
\underset{i\in \{0,\dots,p\}} \sum t_i = 1
\right\}
~\subset~
[0,1]^{\{0,\dots,p\}}
~,
\]
and
\[
\Delta^p_e
~:=~
\left\{
\{0,\dots,p\} \xra{t} \RR
\mid
\underset{i\in \{0,\dots,p\}} \sum t_i = 1
\right\}
~\subset~
\RR^{\{0,\dots,p\}}
~.
\]
Each morphism $[p]\xra{\sigma} [q]$ in $\bDelta$ determines maps
\[
\Delta^p \xra{~\sigma_\ast}~ \Delta^q
\qquad\text{ and }\qquad
\Delta^p_e \xra{~\sigma_\ast~} \Delta^q_e
\]
both given by
\[
\sigma_\ast(t)
\colon
j
\longmapsto
\underset{i \in \sigma^{-1}(j)} \sum t_i
~.
\]

\begin{observation}\label{t44}
\begin{enumerate}
\item[]
\item For each morphism $[p] \xra{\sigma} [q]$ in $\bDelta$, the map $\Delta^p_e \xra{\sigma_\ast} \Delta^q_e$ is an affine map between affine subspaces of Euclidean spaces.
In particular, $\sigma_\ast$ factors as a submersion followed by a proper embedding.

\item For each $p \geq 0$, the topological $p$-simplex inherits the structure of a $p$-manifold with corners, from the standard structure of a compact manifold with corners on finite-fold products of the closed interval.
The evident inclusion $\Delta^p \subset \Delta^p_e$ witnesses a collar extension.

\item These data organize as a cosimplicial smooth manifold $\Delta^\bullet_e$ and as a cosimplicial smooth manifold with corners $\Delta^\bullet$.  

\item 
For each morphism $[p] \xra{\sigma} [q]$ in $\bDelta$, the map $\Delta^p \xra{\sigma_\ast} \Delta^q$ is proper and it restricts to each face of of $\Delta^p$ as a submersion onto the face containing its image.

\end{enumerate}
\end{observation}

Observation~\ref{t44} implies the following.
\begin{observation}\label{t44'}
Let $[p]\xra{\sigma} [q]$ be a morphism in $\bDelta$. 
\begin{enumerate}
    \item The map $\Delta^p \times M \xra{\sigma_\ast \times M} \Delta^q \times M$ is a proper smooth map between manifolds with corners.

    \item Let $\Delta^q \times M \xra{f} X$ be a smooth map. 
 If $f$ is transverse to $X_0 \subset X$, then the composite map $f \circ (\sigma_\ast \times M) \colon \Delta^p \times M \xra{\sigma_\ast \times M} \Delta^q \times M \xra{f} X$ is transverse to $X_0$. 
    
\end{enumerate}

\end{observation}

Observation~\ref{t44'} allows us to define the simplicial structure maps in the following.

\begin{definition}\label{def.transv.sing}
We denote simplicial topological spaces
\begin{eqnarray*}
{\cMap}_\bullet(M,X)
&
~:=~
&
\cMap(\Delta^\bullet \times M,X)
~,
\\
{\cMap}_\bullet^{\sm}(M,X)
&
~:=~
&
\cMap^{\sf sm}(\Delta^\bullet \times M , X)
~,
\\
{\cMap}_\bullet^\pitchfork(M,X)
&
~:=~
&
\cMap^\pitchfork(\Delta^\bullet \times M , X)
~.
\end{eqnarray*}

\end{definition}

\begin{observation}\label{t44''}
There are canonical injective morphisms between simplicial topological spaces
\begin{equation}\label{e80}
\cMap^\pitchfork_\bullet(M,X)
\longrightarrow
\cMap^{\sf sm}_\bullet(M,X)
\longrightarrow
\cMap_\bullet(M,X)
\longleftarrow
\cMap(M,X)
~,
\end{equation}
in which the domain of the last leftward morphism is the constant simplicial topological space, and it is induced by the unique natural transformation $\Delta^\bullet \to \ast$, which is by proper maps.

\end{observation}

\subsection{Smooth extension and approximation results}

In this technical subsection, we prove a few results about extending smooth maps respecting properness and transversality.

\begin{definition}\label{d20}
Let $C$ be a manifold with corners.
A closed union of faces, $C_0 \subset C$, is \bit{properly acyclic} if it satisfies the following condition.
\begin{itemize}
    \item[] Let $C_0 \subset V \subset C$ be an open neighborhood of $C_0$ whose complement $C \smallsetminus V \subset C$ is compact.
    There exists a smooth proper map $C \xra{e} C$ with the following properties.
    \begin{enumerate}
        \item It is a diffeomorphism onto its image, which is open.
        
        \item There is containment of its image $e(C) \subset V$.
        
        \item It extends the inclusion $C_0 \hookrightarrow C$.
    \end{enumerate}
\end{itemize}
\end{definition}

We will be particularly concerned with the following example.

\begin{example}
    Let $0\leq i \leq p$. 
    The closed union of faces $\Lambda^p_i \times M \subset \Delta^p \times M$ is properly acyclic.
    
\end{example}

\begin{lemma}\label{lemma.sm.sing}    
Consider a solid diagram
\[\begin{tikzcd}
	{C_0} & T \\
	C
	\arrow["f", from=1-1, to=1-2]
	\arrow[hook', from=1-1, to=2-1]
	\arrow["{\w{f}}"', dashed, from=2-1, to=1-2]
\end{tikzcd}\]
in which $T$ is a manifold, $C$ is a manifold with corners, $C_0 \subset C$ is a closed union of faces, and $f$ is a map whose restriction to each face in $C_0$ is smooth.
\begin{enumerate}
    \item 
    If $T$ is a vector space, then there is a smooth extension $\w{f}$.

    \item 
    If $f$ is proper and $T$ is a vector space whose dimension is greater than that of $C$, then there is a proper smooth extension $\w{f}$.

\item 
If $C_0 \subset C$ is properly acyclic,
then there is a smooth extension $\w{f}$.

\end{enumerate}

\end{lemma}

\begin{proof}
We first prove statement~(1). 
Using that $C_0 \subset C$ is a closed union of faces, for each $x\in C_0$ choose the following data:
\begin{itemize}
    \item Non-negative integers $\ell \geq 0$ and $d \geq 0$.
    
    \item A down-closed subset $I \subset \cP(\{1,\dots,d\})$ of the power set.

    \item A smooth open embedding
\[
\RR^\ell \times \RR_{\geq 0}^d
\xra{~\varphi_x~}
C
\]
such that $\varphi_x(0) = x$ and $\varphi_x^{-1}(C_0) = \RR^\ell \times B(I)$. 
Here, 
\[
B(I)
~:=~ 
\Bigl\{ 
(r_1,\dots,r_d)
\big| \ \{i \mid r_i \neq 0\} \in I
\Bigr\}
=
\bigcup_{S\in I}\RR_{\geq 0}^S
~\subset~
\RR_{\geq 0}^d
\]
is the union of the coordinate planes, where $\RR^S$ is identified with the subspace of $\RR^d$ whose $i^{th}$ coordinate is zero for each $i\in \{1,\ldots,n\}\smallsetminus S$.
\end{itemize}  
Define the extension
\[\begin{tikzcd}
	{\RR^\ell \times B(I)} & {C_0} & T \\
	{\RR^\ell \times \RR_{\geq 0}^d}
	\arrow[swap, "{{\varphi_x}_{|}}"', from=1-1, to=1-2]
	\arrow[hook', from=1-1, to=2-1]
	\arrow["f", from=1-2, to=1-3]
	\arrow["{\w{\iota \circ f \circ \varphi_x}_|}"', dashed, from=2-1, to=1-3]
\end{tikzcd}\]
by the expression
\[
\w{f \circ {\varphi_x}_|}
~:=~
\sum_{S \in I} (-1)^{d-|S|+1} f \circ {\varphi_x}_| \circ (\id_{\RR^\ell} \times {\sf proj}_S)
\]
where, for $S\in I$, the map $\RR_{\geq 0}^d \xra{{\sf proj}_S} \RR_{\geq 0}^S$ is orthogonal projection onto the subspace non-negatively spanned by the $S$-coordinates. 
Note that this expression is defined in terms of the vector space structure on $T$.
It is easily checked that the map $\w{f\circ {\varphi_x}_|}$ is smooth, and that  signs cancel to ensure that it extends $f \circ {\varphi_x}_|$. 
Denote the open neighborhood $x\in U_x \subset C$ that is the image of $\varphi_x$.
Denote the composite map 
\[
\w{f}_x
\colon 
U_x
\xra{~\varphi_x^{-1}~}
\RR^\ell \times \RR_{\geq 0}^d
\xra{~\w{f \circ {\varphi_x}_|}~}
T
~.
\]
This map $\w{f}_x$ is a composition of smooth maps, and it is therefore smooth. 
Furthermore, the maps $\w{f}_x$ and $f$ restrict identically along $U_x \cap C_0$.

Let $x\in C \smallsetminus C_0$.
Using that $C_0 \subset C$ is a closed subspace, for each $x\in C \smallsetminus C_0$ choose an open neighborhood $x\in U_x \subset C$ such that $U_x \cap C_0 = \emptyset$.  
Choose a smooth map $U_x \xra{\w{f}_x} T$.

Next, choose a smooth partition of unity $(\rho_x)_{x\in C}$ subordinate to the open cover $(U_x)_{x\in C}$.  
The map
\[
\w{f}
\colon 
C
\longrightarrow
T
~,\qquad
\w{f}
~:=~
\sum_{x\in C} \rho_x \cdot  \w{f}_x
~,
\]
is defined, and smooth, and it extends $f$.  
This completes the proof of statement~(1).

We now prove statement~(2).
Once and for all, choose a linear inner product on $T$.
Using statement~(1), choose a smooth extension $C \xra{\w{f}} T$.
By Sard's Theorem, the assumed dimension bound implies $\w{f}$ is not surjective.
So choose $t_0 \in T$ that is not contained in the image of $f$.
Using that $f$ is assumed proper, there exists $r>0$ such that the open ball $\sB_{2r}(t_0) \subset T$ is disjoint from the image of $f$.
Choose a diffeomorphism $\alpha \colon T \smallsetminus \{t_0\} \cong T \smallsetminus \ov{\sB}_r(0)$ from $T$ remove $t_0$ to $T$ remove the closed disk of radius $r$ about $t_0$, such that $\alpha(t) = t$ for each $t \in T \smallsetminus \sB_{2r}(t_0)$.  
Observe that the resulting map $C \xra{\alpha \circ \w{f} } T$ is well-defined, smooth, and proper, and that it extends $C_0 \xra{f} T$.
Upon replacing $\w{f}$ with $\alpha \circ \w{f}$, in what follows we assume the image of $\w{f}$ is disjoint from $\sB_r(t_0)$.

Choose a proper continuous filler of the diagram
\[\begin{tikzcd}
	{C_0} & {\RR_{\geq r}} \\
	C
	\arrow["{\lVert f - t_0\rVert}", from=1-1, to=1-2]
	\arrow[hook', from=1-1, to=2-1]
	\arrow["d"', from=2-1, to=1-2]
\end{tikzcd}
~.
\]
Choose an open neighborhood $C_0 \subset U \subset C$ that admits a retraction $U \xra{\rho} C_0$, which is a deformation retraction.
Homotope $d$ as needed to assume the restriction $d_{|U} \colon U \xra{\rho} C_0 \xra{\lVert f - t_0 \rVert} \RR_{\geq r}$ factors through $C_0$ via the retraction.
Choose a smooth map $C \xra{\sigma} \RR_{\geq 1}$ such that $\sigma(c) = c$ for all $c\in C_0$ and $\sigma \geq \frac{d}{\lVert \w{f}-t_0\rVert}$ -- the assumption on $d_{|U}$ ensures such a $\sigma$ exists.
Consider the map
\[
C \xra{~\w{f}_\sigma} T
~,\qquad
c
\longmapsto
\sigma(c) \Bigl( \w{f}(c) - t_0 \Bigr) + t_0
~,
\]
which is evidently smooth.
For each $c\in C$, observe the inequality
\[
\lVert \w{f}_\sigma(c) - t_0 \rVert
~=~
\sigma \lVert \w{f}(c) - t_0 \rVert
~\geq~
\frac{d(c)}{\lVert \w{f}(c)-t_0 \rVert} \lVert \w{f}(c) - t_0 \rVert
~=~
d(c)
~.
\]
Because $d$ is proper, this inequality ensures $\w{f}_\sigma$ is also proper.
This completes the proof of statement~(2).

We now prove statement~(3).
Choose a proper smooth embedding $T \xra{\iota} \RR^N$ for some $N \geq 0$.
Choose a tubular neighborhood $\iota(T) \subset \w{T} \subset \RR^N$, and a smooth retraction $\w{T} \xra{\pi} T$.
Using statement~(1), choose a smooth extension $\w{\iota \circ f}$ of $\iota \circ f$.
Because $\w{\iota \circ f}$ extends $\iota \circ f$, we have containment 
\[
C_0 
~\subset~ 
\w{\iota \circ f}^{-1}(\w{T}) 
~\subset~ 
C
\]
in the preimage of the tubular neighborhood $\w{T}\subset \RR^N$ of $\iota(T)$.
Using that $C_0 \subset C$ is properly acyclic, choose a smooth proper embedding $C \xra{e} C$ extending the inclusion $C_0 \hookrightarrow C$ whose image is contained in $\w{\iota \circ f}^{-1}(\w{T})$. 
Consider the composite map
\[
\w{f}
\colon 
C
\xra{~e~}
\w{\iota \circ f}^{-1}(\w{T})
\xra{~\w{\iota \circ f}~}
\w{T}
\xra{~\pi~}
T
~.
\]
This map is a composition of smooth maps, and therefore it is smooth.
Furthermore, this map $\w{f}$ extends $f$ because $\w{\iota \circ f}$ extends $\iota \circ f$, and because $\pi$ is a retraction.

\end{proof}

\begin{lemma}\label{t85}
Consider a solid diagram
\[\begin{tikzcd}
	{C_0} & \Th(\xi) \\
	C
	\arrow["f", from=1-1, to=1-2]
	\arrow[hook', from=1-1, to=2-1]
	\arrow["{\w{f}}"', dashed, from=2-1, to=1-2]
\end{tikzcd}\]
in which $\xi = (E \to B)$ is a smooth rank-$k$ vector bundle, $C$ is a manifold with corners, $C_0 \subset C$ is a properly acyclic closed union of faces, and $f$ is a compactly-supported map that restricts to the closure of each face in $C_0$ as a smooth compactly-supported map.
\begin{enumerate}
    \item There is a smooth compactly-supported extension $\w{f}$.

\item If the restriction of $f$ along each face $F$ in $C_0$ is transverse to the zero-section, $f_{|F} \pitchfork B$, then there exists a smooth compactly-supported extension $\w{f}$ that is also transverse to the zero-section, $\w{f} \pitchfork B$.
\end{enumerate}

\end{lemma}

\begin{proof}
These statements readily reduces to the case in which $C$ is connected. 
We now reduce these statements to the case in which $B$ is connected.
Write $B = \underset{\alpha} \coprod B_\alpha$ as a coproduct of its connected components.
The Thom space $\Th(\xi) \cong \underset{\alpha} \bigvee \Th(\xi_{|B_\alpha})$ 
is canonically homeomorphic to the wedge sum of the Thom spaces of the base-change of $\xi$ over each connected component of $B$.
So the map $f$ is a $\pi_0(B)$-indexed sequence $(f_\alpha)_{\alpha \in \pi_0(B)}$ of compactly-supported maps with mutually disjoint supports: ${\sf Supp}(f_\alpha) \cap {\sf Supp}(f_\beta) = \emptyset$ for $\alpha \neq \beta$. 
Note the identity ${\sf Supp}(f) = \underset{\alpha\in \pi_0(B)} \bigcup {\sf Supp}(f_\alpha)$ between subspaces of $C$.
The assumption that $f$ is compactly-supported implies the set $\{\alpha \in \pi_0(B) \mid {\sf Supp}(f_\alpha) \neq \emptyset\}$ is finite.  
Using that $C$ is normal, for each $\alpha\in \pi_0(B)$ such that ${\sf Supp}(f_\alpha) \neq \emptyset$, choose an open neighborhood ${\sf Supp}(f_\alpha) \subset U_\alpha \subset C$ with compact closure such that $U_\alpha \cap U_\beta = \emptyset$ for $\alpha \neq \beta$.
Let 
\[\begin{tikzcd}
	U_\alpha \cap C_0 & \Th(\xi_{|B_\alpha}) \\
	U_\alpha
	\arrow["f", from=1-1, to=1-2]
	\arrow[hook', from=1-1, to=2-1]
	\arrow["{\w{f}_\alpha}"', dashed, from=2-1, to=1-2]
\end{tikzcd}\]
be a smooth extension as in statement~(1) in the case that $B$ is connected.
The map
\[
C \xra{~\w{f}~} \underset{\alpha \in \pi_0(B)} \bigvee \Th(\xi_{|B_\alpha}) \cong \Th(\xi)
~,\qquad
c
\longmapsto 
\begin{cases}
    \w{f}_\alpha(c)
    &
    ,~\text{ if } c\in U_\alpha
    \\
    +
    &
    ,~
    \text{ if } c \notin \underset{\alpha} \bigcup U_\alpha
\end{cases}
~,
\]
is a smooth extension as in statement~(1).
Clearly, if $\w{f}_\alpha$ is transverse to $B_\alpha$ for each $\alpha$, then $\w{f}$ is transverse to $B$.
This completes the reduction to the case in which $B$ is connected.

In what follows, we assume both $C$ and $B$ are connected.  
In particular, both $C$ and $B$ have constant dimension.  
Using this, choose $N> {\sf dim}(C)/2$ together with a smooth embedding $\xi \hookrightarrow \epsilon^N_{\RR^N}$ into a trivial rank-$N$ vector bundle over $\RR^N$ such that the smooth embedding between bases $B \hookrightarrow \RR^N$ has compact closure.  
This embedding uniquely extends as a pointed embedding $\Th(\xi) \xra{\iota} \Th(\epsilon^N_{\RR^N}) \hookrightarrow (\RR^{2N})^+$ to the one-point compactification of $\RR^{2N}$, such that restriction $E \xra{\iota_|} \RR^{2N}$ is smooth.

Because $f$ is continuous and compactly-supported, the preimage $f^{-1}(E) = (\iota \circ f)^{-1}(\RR^{2N}) \subset C_0$ is an open subset whose closure is compact.  
It follows that the smooth map $f^{-1}(E) \xra{f_|} E \xra{\iota_|} \RR^{2N}$ is proper.
Using that $C$ is locally compact, choose an open subset $\w{f}^{-1}(E) \subset C$ whose closure is compact and such that $\w{f}^{-1}(E) \cap C_0 = f^{-1}(E)$.
Note that $f^{-1}(E) \subset \w{f}^{-1}(E)$ is a closed union of faces in a manifold with corners; in particular, the inclusion $f^{-1}(E) \hookrightarrow \w{f}^{-1}(E)$ is a proper map.
Using Lemma~\ref{lemma.sm.sing}(2), choose a proper smooth extension:
\[\begin{tikzcd}
	{f^{-1}(E)} & E & {\RR^{2N}} \\
	{\w{f}^{-1}(E)}
	\arrow["{f_|}", from=1-1, to=1-2]
	\arrow[hook', from=1-1, to=2-1]
	\arrow["{\iota_|}", from=1-2, to=1-3]
	\arrow["{\w{\iota_| \circ f_|}}"', from=2-1, to=1-3]
\end{tikzcd}
~.
\]
Define 
\[
\w{\iota \circ f} \colon 
C 
\xra{~c~}
(\w{f}^{-1}(E))^+
\xra{~(\w{\iota_| \circ f_|})^+~}
(\RR^{2N})^+
~,
\]
where $c$ is the collapse map of $C$ onto the one-point compactification of its open subspace $\w{f}^{-1}(E) \subset C$.
This map is a smooth compactly-supported map because it is a composition of such.
Commutativity of the diagram
\[
\begin{tikzcd}
	{C_0} & {\Th(\xi)} & {(\RR^{2N})^+} \\
	C & {(f^{-1}(E))^+} \\
	& {(\w{f}^{-1}(E))^+}
	\arrow["f", from=1-1, to=1-2]
	\arrow[hook', from=1-1, to=2-1]
	\arrow["c"', from=1-1, to=2-2]
	\arrow["\iota", from=1-2, to=1-3]
	\arrow["c"', from=2-1, to=3-2]
	\arrow["{(\iota_| \circ f_|)^+}"', from=2-2, to=1-3]
	\arrow[hook', from=2-2, to=3-2]
	\arrow["{(\w{\iota_| \circ f_|})^+}"', shift right=5, from=3-2, to=1-3]
\end{tikzcd}
\]
ensures this map $\w{\iota \circ f}$ extends $\iota \circ f$.

Using the tubular neighborhood theorem applied to $\iota(E) \subset \RR^{2N}$, together with the complement of a large open ball about the origin in $\RR^{2N}$, choose the following.
\begin{itemize}
    \item An open neighborhood $\iota(\Th(\xi)) \subset U \subset (\RR^{2N})^+$.
    
    \item A deformation retraction $[0,1]\times U \xra{\rho} U$ of $U$ onto $\iota(\Th(\xi))$ such that the restriction $[0,1] \times \rho_1^{-1}(\iota(E)) \xra{\rho_|} \rho_1^{-1}(\iota(E))$ is a smooth deformation retraction of the preimage $\rho_1^{-1}(\iota(E))$ onto $\iota(E)$.  (So, $\rho_0 = \id_U$ and $\rho_1$ is a retraction.)
\end{itemize}

Using that $\w{\iota \circ f}$ is continuous, and that $\iota \circ f(C_0) \subset \iota(\Th(\xi))$, the preimage
\[
C_0
~\subset~
\w{\iota \circ f}^{-1}(U)
~\subset~
C
\]
is an open neighborhood of $C_0$.
Using that $U$ contains an open neighborhood of $+ \in (\RR^{2N})^+$, properness of $\w{\iota \circ f}$ ensures this open subset $\w{\iota \circ f}^{-1}(U) \subset C$ has compact complement.  
Using that $C_0 \subset C$ is properly acyclic, choose a smooth proper embedding $C \xra{e} C$ extending the inclusion $C_0 \hookrightarrow C$ and such that there is containment of its image: $e(C) \subset \w{\iota \circ f}^{-1}(U)$.
The composite map 
\[
\w{f}
\colon
C
\xra{~e~}
\w{\iota \circ f}^{-1}(U)
\xra{~\w{\iota \circ f}_|~}
U
\xra{~\rho_1~}
\iota(\Th(\xi))
\xra{~\iota^{-1}~}
\Th(\xi)
\]
is well-defined and continuous and it extends $f$.
Because $e$ and $\w{\iota \circ f}$ are compactly-supported and smooth, and $\rho_1$ is smooth and $\iota$ is a smooth embedding, it follows that $\w{f}$ is compactly-supported and smooth.  
This completes the proof of statement~(1).

Now suppose the restriction of $f$ along each face in $C_0$ is transverse to the zero-section, $f \pitchfork B$.
Using statement~(1), choose a smooth compactly-supported extension $\w{f}$.
Using that $f$ is compactly-supported, 
and using openness of transversality (e.g., Theorem~2.1 of~\cite{hirsch}),
choose an open neighborhood $C_0 \subset V \subset C$ such that $V \xra{ \w{f}_{|V}} \Th(\xi)$ is transverse to the zero-section, $\w{f}_{|V} \pitchfork B$.
Using that $C_0 \subset C$ is properly acyclic, choose a smooth proper embedding $C \xra{e} C$ extending the inclusion $C_0 \hookrightarrow C$ and such that there is containment of the image $e(C) \subset V$.
Then $\w{f} \circ e$ is a smooth compactly-supported extension of $f$ that is transverse to the zero-section $B \subset E$.

\end{proof}

\subsection{The Kan condition for smooth mapping spaces}

We extend a simplicial object $\bDelta^{\op} \xra{F} \cX$ in an $\infty$-category with limits to a functor $\sSet^{\op} \xra{j_\ast F} \cX$ via right Kan extension along the Yoneda embedding $\bDelta^{\op} \xra{j} \sSet^{\op}$.
Explicitly, this extension evaluates on $S \in \sSet$ as a limit:
\begin{equation}\label{eq.right.kan}
    F(S) 
~:=~ 
(j_\ast F)(S) 
~=~
\underset{([p]\ra S)^\circ \in (\bDelta_{/S})^{\op}}\limit F([p])
~.
\end{equation}
In the case that $\cX = \Spaces$, recognize this value as the space of morphisms $F(S) = \Hom_{\PShv(\bDelta)}(S , F )$ to $F$ from $S$ regarded as a simplicial space.

\begin{definition}
    A simplicial space $\bDelta^{\op} \xra{F} \Spaces$ \bit{satisfies the Kan condition} if for each $p\geq 0$ and $0\leq i\leq p$, the canonical map
    \[
    F([p]) \longrightarrow F(\Lambda_i[p])
    \]
    is surjective on path-components.
\end{definition}

We prove the following technical result at the end of this subsection, after some set-up.
\begin{lemma}\label{lemma.Map.transv.Kan}
All of the simplicial spaces $\cMap_\bullet(M,\Th(\xi))$ and $\cMap^{\sf sm}_\bullet(M,\Th(\xi))$
and
$\cMap^\pitchfork_\bullet(M,\Th(\xi))$
satisfy the Kan condition.
\end{lemma}

Let $S$ be a simplicial set.
The \bit{subdivision} of $S$ is the full subcategory 
\[
\sd(S) 
~\subset~ 
\bDelta_{/S}
\]
consisting of those objects $([p]\in \bDelta , [p] \xra{s} S)$ such that $s$ is not degenerate.

\begin{observation}\label{t90}
    If $S$ is a simplicial subset of a representable simplicial set, then the category $\sd(S)$ is a finite poset.  
    
\end{observation}

\begin{observation}\label{t91}
    Let $S$ be a simplicial set.
Via the surjective-injective factorization system on $\bDelta$, 
the canonical inclusion
\[
\sd(S) 
~\hookrightarrow~ 
\bDelta_{/S}
\]
is a right adjoint, and is therefore final.
Consequently, for $\bDelta^{\op} \xra{F} \cX$ a simplicial object in an $\infty$-category with limits, the canonical morphism
\[
F(S)
\xra{~\simeq~}
\underset{([p]\ra S)^\circ \in \sd(S)^{\op}}\limit F([p])
\]
is an equivalence.

\end{observation}

Temporarily denote the underlying space functor $U\colon\Top \ra \Spaces$.
Let $F\colon\bDelta^{\op} \ra \Top$ be a simplicial topological space, such as $\cMap^{\pitchfork}_\bullet(M,\Th(\xi))$.
Let $S\in \sSet$ be a simplicial set.
Using that homotopy limits in $\Top$ present limits in $\Spaces$ (Theorem 4.2.4.1 of~\cite{HTT}), the canonical map between spaces
\[
U\bigl(
F(S)
\bigr)
\longrightarrow
(UF)(S)
\]
may be identified as $U$ applied to the canonical map between topological spaces
\begin{equation}\label{f1}
\underset{([p]\ra S)^\circ \in \sd(S)^{\op}}{\sf lim} F([p])
\longrightarrow
\underset{([p]\ra S)^\circ \in \sd(S)^{\op}} {\sf holim} F([p])
\end{equation}
from the point-set limit to the homotopy limit (in the sense of Bousfield--Kan~\cite{bousfield.kan}).  
In the cases $F=\cMap_\bullet(M, \Th(\xi))$ and $F=\cMap^{\sf sm}_\bullet(M, \Th(\xi))$ and $F=\cMap^{\pitchfork}_\bullet(M, \Th(\xi))$, let us respectively denote the map~(\ref{f1}) as
\begin{eqnarray}
\label{f2}
\cMap(|S|\times M, \Th(\xi))
&
\longrightarrow
&
\cMap_S(M,\Th(\xi))
~,
\\
\label{f3}
\cMap^{\sm}(|S|\times M, \Th(\xi))
&
\longrightarrow
&
\cMap_S^{\sm}(M,\Th(\xi))
~,
\\
\label{f4}
\cMap^{\pitchfork}(|S|\times M, \Th(\xi))
&
\longrightarrow
&
\cMap_S^{\pitchfork}(M,\Th(\xi))
~.
\end{eqnarray}

We prove the following result later in this subsection, after some setup.
\begin{lemma}\label{lemma.right.Kan}
For $0\leq i \leq p$, let $S$ be either of the simplicial sets $\Lambda_i[p]$ or $\partial [p]$.
In these cases, the canonical maps~(\ref{f2}),~(\ref{f3}), and~(\ref{f4}) are weak homotopy equivalences.

\end{lemma}

Let us explicate the points of the domains and codomains of the maps~(\ref{f2}),~(\ref{f3}), and~(\ref{f4}).
As the notation suggests, a point in the domain of~(\ref{f2}) is a compactly-supported map
\[
|S| \times M
\xra{~f~}
\Th(\xi)
\]
from the product of $M$ with the geometric realization of $S$.
A point in the domain of~(\ref{f3}) is such a compactly-supported map that is \bit{smooth}, meaning, for each non-degenerate simplex $[p] \xra{\sigma }S$, the composite compactly-supported map $\Delta^p \times M \xra{|\sigma| \times M} |S| \times M \xra{f} \Th(\xi)$ is smooth in the sense of Definition~\ref{d40}.
A point in the domain of~(\ref{f4}) is such a smooth compactly-supported map that is \bit{transverse to the zero-section}, meaning, for each non-degenerate simplex $[p] \xra{\sigma}S$, the composite smooth compactly-supported map $\Delta^p \times M \xra{|\sigma| \times M} |S| \times M \xra{f} \Th(\xi)$ is transverse to $B \subset \Th(\xi)$ in the sense of Definition~\ref{d41}.

We now explicate points in the codomains of the maps~(\ref{f2}),~(\ref{f3}), and~(\ref{f4}).
Observation~\ref{t91} implies the geometric realization of $S$ may be expressed as the colimit
\[
|S|
\xla{~\cong~}
\colim\Bigl(
\sd(S)
\xra{\rm forget}
\bDelta
\xra{\Delta^\bullet}
\Top
\Bigr)
~=~
\underset{([p] \xra{\sigma}S) \in \sd(S)} \colim \Delta^p
~.
\]
Consider the \bit{homotopy geometric realization}, by which we mean the homotopy colimit (in the sense of Bousfield--Kan~\cite{bousfield.kan}):
\[
|S|_{\sf h}
~:=~
\hocolim\Bigl(
\sd(S)
\xra{\rm forget}
\bDelta
\xra{\Delta^\bullet}
\Top
\Bigr)
~=~
\underset{([p] \xra{\sigma} S) \in \sd(S)} \hocolim \Delta^p
~=~
\underset{([q] \xra{\rho} [p] \xra{\sigma} S) \in \TwAr(\sd(S))} \colim | \sd(S)^{\sigma \rho/}| \times \Delta^p
~,
\]
which is, again, the geometric realization of a finite simplicial set.

There is a canonical map
\begin{equation}\label{e110}
|S|_{\sf h}
\longrightarrow
|S|
~,
\end{equation}
which is a quotient map by collapsing some simplices in $|S|_{\sf h}$. 
Because the diagram $\sd(S)
\xra{\rm forget}
\bDelta
\xra{\Delta^\bullet}
\Top$
is cofibrant, this quotient map is a homotopy equivalence.
Unpacking the Bousfield--Kan definition of homotopy limits and homotopy colimits reveals an identification between topological spaces
\[
\cMap_S(M,\Th(\xi))
\xra{~\simeq~}
\underset{([p]\xra{\sigma} S)^\circ \in \sd(S)^{\op}} \holim \cMap(\Delta^p \times M,\Th(\xi))
~=~
\cMap\Bigl( |S|_{\sf h} \times M , \Th(\xi) \Bigr)
~.
\]
So a point in the codomain of~(\ref{f2}) is a compactly-supported map 
\[
|S|_{\sf h} \times M
\xra{~f~}
\Th(\xi)
~.
\]
A point in the codomain of~(\ref{f3}) is, in particular, such a compactly-supported map that is \bit{partially smooth}, meaning, for each $(\sigma \rho \xra{\rho} \sigma) \in \TwAr(\sd(S))$, the composite compactly-supported map $|\sd(S)^{\sigma \rho/}| \times \Delta^p \times M \hookrightarrow |S|_{\sf h} \times M \xra{f} \Th(\xi)$ is a $|\sd(S)^{\sigma \rho/}| $-family of smooth compactly-supported maps $\Delta^p \times M \to \Th(\xi)$ in the sense of Definition~\ref{d40}.
A point in the codomain of~(\ref{f4}) is, in particular, such a compactly-supported map that is \bit{partially smooth and transverse}, meaning, for each $(\sigma \rho \xra{\rho} \sigma) \in \TwAr(\sd(S))$, the composite compactly-supported map $|\sd(S)^{\sigma \rho/}| \times \Delta^p \times M \hookrightarrow |S|_{\sf h} \times M \xra{f} \Th(\xi)$ is a $|\sd(S)^{\sigma \rho/}| $-family of smooth compactly-supported maps $\Delta^p \times M \to \Th(\xi)$ that are transverse to the zero-section $B\subset \Th(\xi)$ in the sense of Definition~\ref{d41}.

\begin{proof}[Proof of Lemma~\ref{lemma.right.Kan}]
We prove the relative homotopy groups of each of the maps~(\ref{f2}),~(\ref{f3}), and~(\ref{f4}) is trivial, starting with the case of the map~(\ref{f2}).
So let $q\geq 0$ and consider a pair $(f,\w{f}_0)$ of horizontal maps fitting into a commutative diagram:
\begin{equation}\label{e113}
\begin{tikzcd}
	{\SS^{q-1}} & {\cMap(|S|\times M , \Th(\xi))} \\
	{\DD^q} & {\cMap_S(M , \Th(\xi))}
	\arrow["{\w{f}_0}", from=1-1, to=1-2]
	\arrow[hook', from=1-1, to=2-1]
	\arrow[hook', from=1-2, to=2-2]
	\arrow["{\w{f}}", dotted, from=2-1, to=1-2]
	\arrow["f"', from=2-1, to=2-2]
\end{tikzcd}
~.
\end{equation}
We must construct a homotopy filler $\w{f}$.
The pair of maps $(f,\w{f}_0)$ curry to a pair of compactly-supported maps filling a commutative diagram among topological spaces:
\[\begin{tikzcd}
	{\SS^{q-1} \times |S|_{\sf h} \times M} & {\DD^q \times |S|_{\sf h} \times M} \\
	{\SS^{q-1} \times |S| \times M} & {\Th(\xi)}
	\arrow[hook, from=1-1, to=1-2]
	\arrow["{\SS^{q-1} \times (\ref{e110}) \times M}"', from=1-1, to=2-1]
	\arrow["f", dashed, from=1-2, to=2-2]
	\arrow["\w{f}_0", dashed, from=2-1, to=2-2]
\end{tikzcd}
~.
\]
Consider the subspace
\[
C_0
~:=~
\Bigl( 
\SS^{q-1} \times |S|  \underset{\SS^{q-1} \times |S|_{\sf h}} \coprod \DD^q \times |S|_{\sf h}
\Bigr)
\times M
~\subset~
C
\]
Observe that the maps $f$ and $\w{f}_0$, together, define a compactly-supported map from the pushout
\[
C_0
~=~
\Bigl( 
\SS^{q-1} \times |S|  \underset{\SS^{q-1} \times |S|_{\sf h}} \coprod \DD^q \times |S|_{\sf h}
\Bigr)
\times M
\xra{~f~}
\Th(\xi)
~;
\]
because the left vertical map in the above diagram is a quotient map, the map $\w{f}_0$ is determined by $f$, thereby justifying simply denoting the above map as $f$.

Now consider the mapping cylinder 
\[
{\sf Cyl}\Bigl( |S|_{\sf h} \xra{(\ref{e110})} |S| \Bigr)
~.
\] 
Consider the topological space 
\[
C
~:=~
\Bigl(
\SS^{q-1} \times I \times |S|
\underset{\SS^{q-1} \times {\sf Cyl}\Bigl( |S|_{\sf h} \xra{(\ref{e110})} |S| \Bigr)} \coprod
\DD^q \times {\sf Cyl}\Bigl( |S|_{\sf h} \xra{(\ref{e110})} |S| \Bigr)
\Bigr)
\times M
~,
\]
which is the product of $M$ with a pushout.
Observe a canonical inclusion $C_0 \subset C$.
Finiteness of $S$ implies the pushout factor of $C$ is compact.
Using that the quotient map $|S|_{\sf h} \xra{(\ref{e110})} |S|$ is a homotopy equivalence given by collapsing simplices, this inclusion $|S|_{\sf h} \hookrightarrow {\sf Cyl}\Bigl( |S|_{\sf h} \xra{(\ref{e110})} |S| \Bigr)$ is an acyclic. 
Consequently, the inclusion $C_0 \hookrightarrow C$ is a properly acyclic cofibration between topological spaces.
Therefore, there is an extension of $f$ to a compactly-supported map
\[
\w{f}
\colon 
C
\longrightarrow
\Th(\xi)
~.
\]
The restriction to one end of the cylinder,
\[
\w{f}_|
\colon 
\DD^q \times |S| \times M
\hookrightarrow 
C
\longrightarrow
\Th(\xi)
~,
\]
is then a compactly-supported map.
Furthermore, $\w{f}$ witnesses a homotopy between $f$ and $\w{f}_|$ as well as a homotopy between $\w{f}_0$ and the restriction of $\w{f}_|$ along $\SS^{q-1} \times |S| \times M \hookrightarrow \DD^{q-1} \times |S| \times M$.
Therefore, this map $\w{f}_|$ is the sought homotopy filler in the diagram~(\ref{e113}).
We conclude that the map~(\ref{f2}) is a weak homotopy equivalence, as desired.

We now prove that the relative homotopy groups of the map~(\ref{f3}) are trivial.  
So consider a diagram
\begin{equation}\label{f21}
\begin{tikzcd}
	{\SS^{q-1}} & {\cMap^{\sf sm}(|S|\times M , \Th(\xi))} \\
	{\DD^q} & {\cMap^{\sf sm}_S(M , \Th(\xi))}
	\arrow["{\w{f}_0}", from=1-1, to=1-2]
	\arrow[hook', from=1-1, to=2-1]
	\arrow[hook', from=1-2, to=2-2]
	\arrow["{\w{f}}", dotted, from=2-1, to=1-2]
	\arrow["f"', from=2-1, to=2-2]
\end{tikzcd}
~.
\end{equation}
 
We now use the assumption that $S = \partial[p]$ or $\Lambda_i[p]$ for some $0\leq i \leq p$.
In these cases, this topological space $C$ has a canonical structure of a compact manifold with corners such that the canonical inclusion $C_0 \hookrightarrow C$ is a properly acyclic closed union of faces. 
By proper smooth approximation~(e.g., as follows from~\cite{KM.smoothapprox}), homotope the pair $(f,\w{f}_0)$ to  one for which the continuous map $C_0 \xra{f} \Th(\xi)$ is a smooth compactly-supported map (i.e., its restriction to each face in $C_0$ is a smooth compactly-supported map).
By Lemma~\ref{t85}(1), there exists a smooth compactly-supported extension 
\[
\w{f} \colon C 
\longrightarrow \Th(\xi)
\]
of $f$.
Its restriction
\[
\w{f}_|
\colon 
\DD^q \times |S| \times M
\hookrightarrow 
C
\longrightarrow
\Th(\xi)
\]
is then a smooth compactly-supported map.
Furthermore, $\w{f}$ witnesses a smooth homotopy between $f$ and $\w{f}_|$ as well as a smooth homotopy between $\w{f}_0$ and the restriction of $\w{f}_|$ along $\SS^{q-1} \times |S| \times M$.
Therefore, this map $\w{f}_|$ is the sought homotopy filler in the diagram~(\ref{f21}).
We conclude that the map~(\ref{f3}) is a weak homotopy equivalence, as desired.

Lastly, we prove that the relative homotopy groups of the map~(\ref{f4}) are trivial.  
So consider a diagram
\begin{equation}\label{f22}
\begin{tikzcd}
	{\SS^{q-1}} & {\cMap^{\sf \pitchfork}(|S|\times M , \Th(\xi))} \\
	{\DD^q} & {\cMap^{\pitchfork}_S(M , \Th(\xi))}
	\arrow["{\w{f}_0}", from=1-1, to=1-2]
	\arrow[hook', from=1-1, to=2-1]
	\arrow[hook', from=1-2, to=2-2]
	\arrow["{\w{f}}", dotted, from=2-1, to=1-2]
	\arrow["f"', from=2-1, to=2-2]
\end{tikzcd}
~.
\end{equation}
We again use the assumption that $S = \partial[p]$ or $\Lambda_i[p]$ for some $0\leq i \leq p$.
Again, the topological space $C$ has a canonical structure of a compact manifold with corners such that the canonical inclusion $C_0 \hookrightarrow C$ is a properly acyclic closed union of faces. 
By proper smooth approximation~(e.g., as follows form~\cite{KM.smoothapprox}) followed by genericity of transversality~(Theorem 2.1 of~\cite{hirsch}), homotope the pair $(f,\w{f}_0)$ to  one for which $C_0 \xra{f} \Th(\xi)$ is smooth and transverse to the zero-section $B \subset \Th(\xi)$.
Note that $\w{f}_0$, then, is also smooth and transverse to the zero-section $B\subset \Th(\xi)$.
By Lemma~\ref{t85}(2), there exists a smooth compactly-supported extension 
\[
\w{f} \colon C 
\longrightarrow \Th(\xi)
\]
of $f$ that is transverse to the zero-section $B \subset \Th(\xi)$.
The restriction
\[
\w{f}_|
\colon 
\DD^q \times |S| \times M
\hookrightarrow 
C
\longrightarrow
\Th(\xi)
\]
is then a smooth compactly-supported map that is transverse to the zero-section $B\subset \Th(\xi)$.
Furthermore, $\w{f}$ witnesses a smooth homotopy through transverse maps between $f$ and $\w{f}_|$ as well as a smooth homotopy through transverse maps between $\w{f}_0$ and the restriction of $\w{f}_|$ along $\SS^{q-1} \times |S| \times M$.
Therefore, this map $\w{f}_|$ is the sought homotopy filler in the diagram~(\ref{f22}).
We conclude that the map~(\ref{f4}) is a weak homotopy equivalence, as desired.
\end{proof}

\begin{proof}[Proof of Lemma~\ref{lemma.Map.transv.Kan}]
Through Lemma~\ref{lemma.right.Kan}, the result is implied upon proving each of the maps between topological spaces
\begin{eqnarray*}
\cMap(\Delta^p\times M, \Th(\xi)) 
&
\longrightarrow
&
\cMap(\Lambda^p_i\times M, \Th(\xi)) 
\\
\cMap^{\sf sm}(\Delta^p\times M, \Th(\xi))
&
\longrightarrow
&
\cMap^{\sf sm}(\Lambda^p_i\times M, \Th(\xi)) 
\\
\cMap^\pitchfork(\Delta^p\times M, \Th(\xi)) 
&
\longrightarrow
&
\cMap^\pitchfork(\Lambda^p_i\times M, \Th(\xi)) 
\end{eqnarray*}
is surjective on path-components.

The inclusion $\Lambda^p_i\hookrightarrow \Delta^p$ is a section of a retraction between topologiacl spaces.
The resulting retraction $r: \Delta^p\times M\ra \Lambda^p_i\times M$ is a proper map. 
Precomposition with $r$ then defines a section of the restriction map:
\[
\cMap(\Delta^p\times M, \Th(\xi)) 
\cMap(\Lambda^p_i\times M, \Th(\xi)) 
~.
\]
Therefore, this restriction map is surjective on path-components, thereby proving the first case.

For the next case,
recall that the subspace $C_0=\Lambda^p_i\times M \subset \Delta^p \times M = C$ is a properly acyclic closed union of faces of a manifold with corners.
Let $f \in \cMap^{\sf sm}(\Lambda^p_i \times M , \Th(\xi))$.
Lemma~\ref{t85}(1) ensures the existence of a smooth compactly-supported extension:
\[
\xymatrix{
\Lambda_i^p\times M\ar[d]\ar[r]^-f&\Th(\xi)\\
\Delta^p \times M\ar@{-->}[ur]_-{\w{f}}}
~.
\]
This implies the sought surjectivity on path-components for the second case.

Similarly, if $f\in \cMap^\pitchfork(\Lambda^p_i \times M , \Th(\xi))$ is a smooth compactly-supported map that is transverse with the zero-section $B \subset \Th(\xi)$, then Lemma~\ref{t85}(2) ensures the existence of such a smooth extension $
\w{f} \in \cMap^\pitchfork(\Delta^p \times M , \Th(\xi))$ that is transverse with the zero-section.
This implies the sought surjectivity on path-components for the third case.

\end{proof}

\subsection{Parametrized transversality}

We now state the main result of this section.
\begin{theorem}\label{t80}
    Each of the morphisms in~(\ref{e80}) induces an equivalence between geometric realizations:
    \[
\bigl| \cMap^\pitchfork_\bullet\bigl(M,\Th(\xi)\bigr)\bigr|
\xra{~\simeq~}
\bigl| \cMap^{\sf sm}_\bullet\bigl(M,\Th(\xi)\bigr)\bigr|
\xra{~\simeq~}
\bigl| \cMap_\bullet\bigl(M,\Th(\xi)\bigr) \bigr|
\xla{~\simeq~}
\cMap\bigl(M,\Th(\xi)\bigr)
~.
\]
\end{theorem}

We make use of the following powerful criterion for establishing when a map between simplicial spaces induces an equivalence between geometric realizations.

\begin{definition}\label{defKan}
    A morphism $F \ra G$ between simplicial spaces is a \bit{trivial Kan fibration} if, for each $p\geq 0$, the map
    \[
    F([p]) \longrightarrow F(\partial[p])\underset{G(\partial[p])}\times G([p])
    \]
    is surjective on path-components.
\end{definition}

\begin{terminology}
    We will say that a morphism $F \to G$ between simplicial topological spaces is a \bit{trivial Kan fibration} if the associated morphism between simplicial spaces is. 
    That is, if the canonical map to the homotopy pullback is surjective on path-components, where $F(\partial [p])$ and $G(\partial[p])$ are the homotopy right Kan extensions of $F$ and $G$ to ${\sf sSet}$ as in the formula~(\ref{eq.right.kan}).
\end{terminology}

For the proof of the following, see Proposition~6 of~\cite{lurie.kan.fibration}.

\begin{lemma}\label{lemma.claim}
Let $f\colon F \ra G$ be a morphism between simplicial spaces.
If $f$ is a trivial Kan fibration, then the map between geometric realizations, $|F| \xra{|f|} |G|$,
is an equivalence.
\end{lemma}
\qed

\begin{lemma}\label{prop.smooth.approx}
    The canonical morphism between simplicial spaces
    \[
    \cMap^{\sm}_\bullet(M, \Th(\xi)) \longrightarrow \cMap_\bullet(M, \Th(\xi))
    \]
    is a trivial Kan fibration.
\end{lemma}
\begin{proof}
By Lemma~\ref{lemma.right.Kan}, we have an equivalence
\[
\cMap^{\sm}_{\partial[p]}\bigl(M, \Th(\xi)\bigr)
\simeq
\cMap^{\sm}\bigl(\partial\Delta^p\times M, \Th(\xi)\bigr)
\]
and likewise for continuous mapping spaces.
Consequently, to apply Lemma~\ref{lemma.claim}, it then suffices to show that for each $p\geq 0$, the map
\[
\cMap^{\sm}(\Delta^p\times M, \Th(\xi))\longrightarrow
\cMap^{\sm}(\partial\Delta^p\times M, \Th(\xi))\underset{\cMap(\partial\Delta^p\times M, \Th(\xi))}{\times}\cMap(\Delta^p\times M, \Th(\xi))
\]
is surjective on path-components, where the righthand side is understood as the homotopy pullback.
Because each inclusion $\partial \Delta^p \hookrightarrow \Delta^p$ is a cofibration, then each restriction map
\[
\cMap(\Delta^p\times M, \Th(\xi))\longrightarrow\cMap(\partial\Delta^p\times M, \Th(\xi))
\]
is a Serre fibration.  
Therefore, the canonical map from the pullback to the homotopy pullback is a weak homotopy equivalence.
We momentarily denote the point-set pullback as
\[
\cMap^{\sm}_{\partial\Delta^p\times M}(\Delta^p\times M, \Th(\xi))
~.
\]
Explicitly, this is a topological space consisting of compactly-supported maps $f:\Delta^p\times M\ra \Th(\xi)$ such that the restriction $f_{|M\times\partial \Delta^p}$ is smooth. 
Evidently, the map in question factors through this point-set pullback.  
It thus suffices to show the canonical map
\[
\cMap^{\sm}(\Delta^p\times M, \Th(\xi))
\longrightarrow
\cMap^{\sm}_{\partial\Delta^p\times M}(\Delta^p\times M, \Th(\xi))
\]
is surjective on path-components. 
This follows by proper smooth approximation (e.g., see~\cite{KM.smoothapprox}): for any $f\in \cMap^{\sm}_{\partial\Delta^p\times M}(\Delta^p\times M, \Th(\xi))$, there is a homotopy $f_t$ which is constant on $\partial\Delta^p\times M$ and such that $f_0=f$ and $f_1$ is smooth.

\end{proof}

\begin{lemma}\label{thm.param.transv}
The canonical morphism between simplicial spaces
\[
\cMap_\bullet^\pitchfork(M,\Th(\xi))\longrightarrow \cMap^{\sf sm}_\bullet(M,\Th(\xi))
\]
is a trivial Kan fibration.
\end{lemma}

\begin{proof}
Recall by Lemma~\ref{lemma.right.Kan}, the maps
\[
\cMap^{\sm}(\partial\Delta^p\times M, \Th(\xi)) \longrightarrow \cMap^{\sm}_{\partial[p]}(M,\Th(\xi))
\]
\[
\cMap^{\pitchfork}(\partial\Delta^p\times M, \Th(\xi)) \longrightarrow \cMap^{\pitchfork}_{\partial[p]}(M,\Th(\xi))
\]
are weak homotopy equivalences.
So to verify the map is a trivial Kan fibration, we must verify $\pi_0$-surjectivity of the following map to the homotopy pullback
\[
\xymatrix{
\cMap^{\pitchfork}(\Delta^p\times M, \Th(\xi))\ar[r]\ar@{-->}[rd]&\cMap^{\pitchfork}(\partial\Delta^p\times M, \Th(\xi))\underset{\cMap^{\sm}(\partial\Delta^p\times M, \Th(\xi))}{\times^{\sf h}}\cMap^{\sm}(\Delta^p\times M, \Th(\xi))\\
&\cMap^{\sm}_{\pitchfork \partial\Delta^p\times M}(\Delta^p\times M, \Th(\xi))\ar@{-->}[u]}
\]
which factors through the point-set pullback, which we will temporarily denote
\[
\cMap^{\sm}_{\pitchfork \partial\Delta^p\times M}(\Delta^p\times M, \Th(\xi))~.
\]
Observe that the point-set pullback is exactly the subspace of $\cMap^{\sm}(\Delta^p\times M, \Th(\xi))$ consisting of all maps $f$ for which the restriction $f_{|\partial\Delta^p\times M}$ is transverse to the zero-section $B\subset \Th(\xi)$.

To show $\pi_0$-surjectivity of the composite map, it therefore suffices to show that each of the factored maps are $\pi_0$-surjective separately. 
By standard transversality (e.g., Theorem~2.1 of \cite{hirsch}), the first map
\[
\cMap^{\pitchfork}(\Delta^p\times M, \Th(\xi))
\hookrightarrow
\cMap^{\sm}_{\pitchfork \partial\Delta^p\times M}(\Delta^p\times M, \Th(\xi))
\]
has dense image. Since the codomain is locally path-connected, the map is surjective on path-components. The remainder of the proof establishes that the second factor,
\begin{equation}\label{e82}
\cMap^{\sm}_{\pitchfork \partial\Delta^p\times M}(\Delta^p\times M, \Th(\xi))
\longrightarrow 
\cMap^{\pitchfork}(\partial\Delta^p\times M, \Th(\xi))\underset{\cMap^{\sm}(\partial\Delta^p\times M, \Th(\xi))}{\times^{\sf h}}\cMap^{\sm}(\Delta^p\times M, \Th(\xi))
~,
\end{equation}
is surjective on path-components.
Denote the closed interval $I := [0,1]$.
A point in the codomain of~(\ref{e82}), which is a homotopy limit, is adjoint to a compactly-supported map
\[
I \times \partial \Delta^p \times M
\underset{
\{1\}\times \partial \Delta^p \times M
}
\coprod
\{1\} \times \Delta^p \times M
\xra{~f~}
\Th(\xi)
~,
\]
with the following properties.
\begin{itemize}
    \item The restriction of $f$ to each face of $\Delta^p\times M$ is smooth.

    \item For each $t\in I$, the restriction of $f$ to each face of $\{t\}\times \partial \Delta^p \times M$ is smooth.

    \item The restriction of $f$ to $\{0\} \times \partial \Delta^p \times M$ is transverse to the zero-section.
\end{itemize}
By smooth approximation relative to $\partial I \times \partial \Delta^p\times M$, we can homotope $f$ to be smooth, and so that the homotopy is constant on $\partial I \times \partial \Delta^p\times M$.
We may therefore assume the point in the codomain of~(\ref{e82}) is adjoint to a compactly-supported map 
\[
I \times \partial \Delta^p \times M
\underset{
\{1\}\times \partial \Delta^p \times M
}
\coprod
\{1\}\times \Delta^p \times M
\longrightarrow
\Th(\xi)
\]
whose restriction to each face of $\Delta^p \times M$ and of $I \times \partial \Delta^p \times M$ is a smooth compactly-supported map.

Consider the following manifold with corners $C$, and closed union $C_0$ of faces therein:
\[
\Bigl(
~C_0 
~\subset
~C~
\Bigr) 
~:=~ 
\Bigl(
~
I \times \partial \Delta^p \times M 
\underset{
\{1\}\times \partial \Delta^p \times M
}
\coprod
\{1\}\times \Delta^p \times M
~\subset~ 
I \times \Delta^p \times M
~
\Bigr)
~.
\]
Observe that this pair $C_0 \subset C$ is properly acyclic. 
Therefore, Lemma~\ref{t85} applies to achieve a smooth extension 
\[
\w{f}\colon C = I \times \Delta^p \times M \longrightarrow \Th(\xi)
\]
of $f$.
Now, by assumption, the restriction of $f$ along $\{0\}\times \partial \Delta^p \times M$ is transverse to the zero-section $B \subset \Th(\xi)$.
Because $\w{f}$ extends $f$, the composite map 
\[
\w{f}_0
\colon
\Delta^p \times M = \{0\}\times \Delta^p \times M
~\hookrightarrow~
I \times \Delta^p \times M 
\xra{~\w{f}~}
\Th(\xi)
\]
is therefore an element in the domain of~(\ref{e82}).
Consider the map
\[
\mu
\colon 
I \times C
=
I \times I \times (\Delta^p \times M)
\xra{~m \times (\Delta^p\times M)~}
I \times (\Delta^p\times M)
=
C
~,
\]
where $I\times I \xra{m} I$ is given by $m(s,t) := st$.  
The composite map,
\[
I \times C_0
~\hookrightarrow~
I \times C
\xra{~\mu~}
C
\xra{~\w{f}~}
\Th(\xi)
~,
\]
is adjoint to a path in the codomain of~(\ref{e82}) to the original point from the image of $\w{f}_0$.

\end{proof}

We can prove the main theorem of this section, our parametrized transversality theorem.

\begin{proof}[Proof of Theorem~\ref{t80}]
Using Lemma~\ref{lemma.claim}, Lemma~\ref{thm.param.transv} implies the first map is an equivalence and Lemma~\ref{prop.smooth.approx} implies the second map is an equivalence.  

The map between cosimplicial topological spaces $\Delta^\bullet \ra \ast$ is by proper homotopy equivalencs.
Therefore, the canonical morphism between simplicial spaces,
$\cMap(M,X)
\to
\cMap_\bullet(M,X)
$,
is an equivalence. 
It follows that the last map is an equivalence.

\end{proof}

\section{Codimension-$\xi$ tangles in $\Delta^\bullet \times M$}\label{sec.Bord.xi}

In this section, for $M$ a manifold with boundary, and for $\xi$ a rank-$k$ vector bundle, we construct a simplicial space $\Bord^\xi_\bullet(M)$ whose space of vertices is a moduli space of codimension-$\xi$ tangles in $M$, and whose space of edges is a moduli space of embedded cobordisms between such.

\subsection{Codimension-$k$ submanifolds}

In this subsection, we fix non-negative integers $p,n,k\geq 0$, and we fix a $(p+n+k)$-manifold with corners $C$.
The key construction here is that of a moduli space $\Sub^k(C)$ of compact codimension-$k$ submanifolds of $C$.

\begin{definition}\label{d7}
A \bit{codimension-$k$ submanifold} of $C$ is a subspace $W \subset C$ such that, for each $x\in W$, there exists a smooth open embedding $\RR^k \times \RR^j \times \RR_{\geq 0}^d \xra{\varphi} C$ such that $\varphi(0)=x$ and $\varphi^{-1}(W) = \{0\} \times \RR^j \times \RR_{\geq 0}^d$.
\end{definition}

\begin{remark}
    Note that, for $F\subset C$ a face of $C$, then $F$ is not a submanifold of $C$ in the sense of Definition~\ref{d7}.
\end{remark}

\begin{remark}
Informally, a submanifold of $C$ is a smooth submanifold with corners $W \subset C$ such that $W \pitchfork F$ for each face $F \subset C$, and each face of $W$ is $W \cap F$, an intersection of $W$ with a face of $C$.
\end{remark}

\begin{definition}
\label{d4}
Let $W$ be a compact $(p+n)$-manifold with corners.
    \begin{enumerate}
    \item An \bit{embedding} of $W$ into $C$ is a smooth embedding $f\colon W \hookrightarrow C$ between manifolds with corners such that, for each $x\in W$, there exists smooth open embeddings $\RR^j \times \RR_{\geq 0}^d \xra{\varphi} W$ and $\RR^k \times \RR^j \times \RR_{\geq 0}^d \xra{\psi}C$ such that $\varphi(0)=x$ and such that the diagram
    \[
\begin{tikzcd}
	{\RR^j \times \RR_{\geq 0}^d} && {\RR^k \times \RR^j \times \RR_{\geq 0}^d} \\
	W && C
	\arrow["{(s,t)\mapsto (0,s,t)}", from=1-1, to=1-3]
	\arrow["\varphi"', hook, from=1-1, to=2-1]
	\arrow["\psi", hook, from=1-3, to=2-3]
	\arrow["f", from=2-1, to=2-3]
\end{tikzcd}
    \]
    commutes.

    \item The topological space of \bit{embeddings} of $W$ into $C$ is the subspace
    \[
    \Emb(W,C)
    ~\subset~
    \Map^{\sm}(W,C)
    \]
    of the space of smooth maps from $W$ to $C$ (endowed with the compact-open $\sC^\infty$ topology) consisting of the embeddings.    

    \item The topological space of \bit{diffeomorphisms}
    \[
    \Diff(W)
    ~\subset~
    \Emb(W,W)
    \]
    is the subspace consisting of those embeddings that admit a smooth inverse.    
    
    \end{enumerate}
\end{definition}

\begin{observation}\label{t56}
    Let $W \xra{f} C$ be an embedding.
    Its image $f(W) \subset C$ is a submanifold of $C$ in the sense of Definition~\ref{d7}.
    In particular, $f$ is transverse to each face of $C$.
\end{observation}

\begin{definition}
\label{d3}
    The topological space
    \[
    \Sub^k(C)
    ~:=~
    \left\{
    W
    \subset
    C
    \text{ codimension-$k$ submanifold}
    \right\}
    \]
    has underlying set the compact codimension-$k$ submanifolds of $M$; its topology is the finest for which, for each compact $(p+n)$-manifold with corners $W$, the map
    \[
    \Emb(W,C)
    \xra{~{\sf Image}~}
    \Sub^k(C)
    ~,\qquad
    (W \xra{f} C)
    \longmapsto 
    f(W)
    ~,
    \]
    is continuous.
\end{definition}

\begin{observation}
Let $W$ be a compact $(p+n)$-manifold with corners.
Composition of diffeomorphisms defines a continuous group structure on $\Diff(W)$.
With respect to this group structure, precomposing embeddings by diffeomorphisms defines a continuous action
    \[
    \Diff(W)
    \lacts
    \Emb(W,C)
    ~.
    \]
    This action is free because an embedding is, in particular, injective.
\end{observation}

\begin{observation}
\label{t15}
The definition of an embedding is such that, for $f,g \colon W \to C$ two embeddings from a compact $(p+n)$-manifold with corners, if ${\sf Image}(f)={\sf Image}(g)$ then $W \xra{g^{-1} \circ f} W$ is a diffeomorphism.  
Consequently, by Definition~\ref{d3}, there is a canonical homeomorphism
    \[
    \underset{[W]} \coprod \Emb(W,C)_{/\Diff(C)}
    \xra{~\cong~}
    \Sub^k(C)
    ~,
    \]
    from a coproduct indexed by the set of isomorphism classes of compact $(p+n)$-manifolds with corners.
\end{observation}

\begin{observation}
    \label{t17}
The homeomorphism of Observation~\ref{t15} supplies a map between spaces:
    \begin{equation}
    \label{e18}
    \Sub^k(C)
    \underset{\rm Obs~\ref{t15}}\simeq
    \underset{[W]} \coprod \Emb(W,C)_{/\Diff(W)}
    \longrightarrow
    \underset{[W]} \coprod \BDiff(W)
    ~,\qquad 
    (W \subset C)
    \longmapsto 
    W
    ~,
    \end{equation}
    to the moduli space of compact $n$-manifolds with corners.
    
\end{observation}

The next result gives a convenient subbasis for the topology of $\Sub^k(C)$.
For $W \subset C$ a submanifold, its normal bundle 
\[
\nu_{W \subset C} ~=~ (~E(\nu_{W \subset C}) \xra{~\pi~} W~)
\]
is a smooth vector bundle.  
The topological space $\Gamma^{\sf sm}(\nu_{W \subset C})$ is that of smooth sections of $\nu_{W \subset C}$ as it is endowed with the compact-open $\sC^\infty$ topology.
The tubular neighborhood theorem ensures the existence of an open embedding $E(\nu_{W \subset C}) \xra{\psi} C$ under $W$ between manifolds with corners.
Given such, post-composition with $\psi$ defines a map
\begin{equation}
    \label{e57}
    \Gamma^{\sf sm}(\nu_{W \subset C}) 
    \xra{~\psi_\ast~}
    \Emb(W,C)
    ~,\qquad
    \sigma
    \longmapsto
    \psi \circ \sigma
    ~.
\end{equation}

\begin{lemma}\label{t52}
Let $(W \subset C) \in \Sub^k(C)$ be a compact codimension-$k$ submanifold of $C$.
Let $E(\nu_{W \subset C}) \xra{~\psi~} C$ be a smooth open embedding under $W$ between manifolds with corners.
The resulting map~(\ref{e57}) has following properties.
\begin{enumerate}
    \item The composite map
    \[
    {\sf Image} \circ \psi_\ast
    \colon
    \Gamma^{\sf sm}(\nu_{W \subset C}) 
    \xra{~\psi_\ast~}
    \Emb(W,C)
    \xra{~\sf Image~}
    \Sub^k(C)
    ~,\qquad
    \sigma
    \longmapsto
    \psi\Bigl(
    \sigma(W)
    \Bigr)
    ~,
    \]
    is an open embedding that carries the zero-section to $W$.
    In particular, $\Sub^k(C)$ is locally a topological vector space.  

\item The action map
\[
{\rm action}
\colon
\Diff(W) \times \Gamma^{\sf sm}(\nu_{W\subset C})
\xra{~ \Diff(W) \times \psi_\ast~}
\Diff(W)
\times 
\Emb(W,C)
\xra{~\rm action~}
\Emb(W,C)
\]
fits into a pullback diagram among topological spaces:
\[
\begin{tikzcd}
	{\Diff(W) \times \Gamma^{\sf sm}(\nu_{W\subset C})} & {\Emb(W,C)} \\
	{\Gamma^{\sf sm}(\nu_{W\subset C})} & {\Sub^k(C)}
	\arrow["{{\rm action}}", from=1-1, to=1-2]
	\arrow["\pr"', from=1-1, to=2-1]
	\arrow["\lrcorner"{anchor=center, pos=0.125}, draw=none, from=1-1, to=2-2]
	\arrow["{{\sf Image}}", from=1-2, to=2-2]
	\arrow["{{\sf Image} \circ \psi_\ast}", from=2-1, to=2-2]
\end{tikzcd}
~.
\]
In particular, the map $\Emb(W,C) \xra{\sf Image} \Sub^k(C)$ is a principal-$\Diff(W)$ fiber bundle onto its image.

\end{enumerate}
\end{lemma}

\begin{proof}   
    
    We first prove that the map ${\sf Image} \circ \psi_\ast$ is injective.
    Let $\sigma, \tau \in \Gamma^{\sf sm}(\nu_{W \subset C})$.
    Suppose ${\sf Image}\circ \psi_\ast(\sigma) = {\sf Image} \circ \psi(\tau)$.
    Because $\psi$ is injective, this supposition implies $\sigma(W) = \tau(W)$.
    Because both $\sigma$ and $\tau$ are sections in a retraction, this implies $\sigma=\tau$, as desired.

    Using that the action $\Diff(W) \lacts \Emb(W,C)$ is free, and that ${\sf Image}$ witnesses the $\Diff(W)$-quotient, injectivity of ${\sf Image} \circ \psi_\ast$ implies the square commutes and is a pullback among topological spaces.

    It remains to prove that the map ${\sf Image}\circ \psi_\ast$ is an open embedding.  
    By definition of the quotient topology, and using the pullback square, this is so if and only if the map $\Diff(W) \times \Gamma^{\sf sm}(\nu_{W \subset C})\xra{\rm action} \Emb(W,C)$ is an open embedding.
    Because ${\sf Image} \circ \psi_\ast$ is injective, so is this map ${\rm action}$.  
    Note that the image of this map ${\rm action}$ is the subspace $\Emb^\psi(W,C) \subset \Emb(W,C)$ consisting of those embeddings $W \xra{f} C$ with the following properties.
    \begin{itemize}
        \item ${\sf Image}(f) \subset {\sf Image}(\psi)$.

        \item Using that $\psi$ is a smooth open embedding, the resulting composite smooth map 
        \[
        \pi \circ \psi^{-1} \circ f
        \colon
        W \xra{f} {\sf Image}(\psi) \xra{\psi^{-1}} E(\nu_{W \subset C}) \xra{\pi} W
        \]
        is a diffeomorphism.  
        
    \end{itemize}
    Using that the compact-open $\sC^\infty$ topology of $\Emb(W,C)$ is finer than the compact-open topology, the subset of $\Emb(W,C)$ consisting of those embeddings with the first property is open.
    Now, consider the topological space $\Map^{\sf sm,\partial}(W,W)$ of smooth maps from $W$ to $W$ that carry faces to faces, as it is endowed with the subspace of the compact-open $\sC^\infty$ topology.
    In this topology, using that $W$ is compact, the subspace $\Diff(W) \subset \Map^{\sf sm, \partial}(W,W)$ is open -- indeed, it is a union of connected components the submersions.
    Consequently, the subset $\Emb^\psi(W,C) \subset \Emb(W,C)$ is open.

    It remains to show that the resulting continuous bijection $\Diff(W) \times \Gamma^{\sf sm}(\nu_{W\subset C})\to \Emb^\psi(W,C)$ is a homeomorphism.
    An inverse is given by the map
    \[
    \Emb^\psi(W,C)
    \longrightarrow
    \Gamma^{\sf sm}(\nu_{W\subset C}) \times \Diff(W)
    ~,\qquad
    f
    \longmapsto
    \Bigl( 
    ~
    \pi \circ \psi^{-1} \circ f
    ~,~
    \psi^{-1} \circ f \circ .\pi \circ \psi^{-1} \circ f
    ~
    \Bigr)
    ~.
    \]
    
\end{proof}

\begin{lemma}\label{t83}
The topological space $\Sub^k(C)$ is paracompact and Hausdorff.  
\end{lemma}
\begin{proof}
The topological spaces $\Diff(W)$ and $\Emb(W,C)$ are each subspaces of compact-open $\sC^\infty$ topologies on sets of smooth maps, and are therefore metrizable, normal, paracompact, and Hausdorff. 
Lemma~\ref{t52}(2) implies the map $\Emb(W,C) \to \Sub^k(C)$ is an open map.
Paracompactness and Hausdorffness of $\Sub^k(C)$ follows.  

\end{proof}

\begin{remark}\label{r.smoothness}
    Lemma~\ref{t52}(1) suggests $\Sub^k(C)$ has a canonical structure of an infinite-dimensional smooth manifold.

\end{remark}

\subsection{Face restriction}
In this subsection, we fix non-negative integers $p,n,k\geq 0$, and we fix a $(p+n+k)$-manifold with corners $C$.
Let $D$ be a manifold with corners.
Here, we show that intersection with a closed union of faces $C_0$ of $C$ determines a Serre fibration: $\Sub^k(C) \to \Sub^k(C_0)$.

\begin{definition}\label{d6}
A \bit{proper face-submersion} from $D$ to $C$ is a smooth map $D \xra{\sigma} C$ that satisfies the following conditions.
\begin{itemize}
    \item The map $\sigma$ is proper.

    \item The image under $\sigma$ of each face of $D$ is a face of $C$.

    \item The restriction of $\sigma$ along a face of $D$ is a submersion onto its image.

\end{itemize}
A \bit{proper face-embedding} from $D$ to $C$ is a proper face-submersion that is an embedding.  

\end{definition}

\begin{example}
The map from the closed square to itself,
\[
[0,1]^2
\longrightarrow
[0,1]^2
~,\qquad
(x,y)
\longrightarrow
(x,0)
~,
\]
is a proper face-submersion.
    
\end{example}

\begin{example}
For each morphism $[p] \xra{\sigma} [q]$ in $\bDelta$, the map $\Delta^p \xra{\sigma_\ast} \Delta^q$ is a proper face-submersion.
If $\sigma$ is injective, then $\sigma_\ast$ is a proper face-embedding.
    
\end{example}

\begin{example}
Let $C_0 \subset C$ be a closed union of faces which inherits the structure of a manifold with corners.
The inclusion $C_0 \hookrightarrow C$ is a proper face-embedding.  
In fact, every proper face-embedding is isomorphic to such an inclusion.
    
\end{example}

\begin{remark}
    Note that a proper face-embedding is not an embedding in the sense of Definition~\ref{d4}.
\end{remark}

\begin{observation}\label{a8}
The composition of proper face-submersions is a proper face-submersion.
\end{observation}

\begin{terminology}
    Let $D \xra{\sigma} C$ be a proper face-submersion.
    Let $W \subset C$ be a submanifold.
    We say $W$ is \bit{transverse} to $\sigma$ if, for each face $E \subset D$ with image face $F=\sigma(E)$, the smooth map between smooth manifolds $E \xra{\sigma_{|E}} F$ is transverse to the smooth submanifold $W \cap F \subset F$.
\end{terminology}

\begin{lemma}
\label{t57}
Each submanifold $W \subset C$ is transverse to each proper face-submersion $D \xra{\sigma} C$.
Furthermore, if $W$ is compact and has codimension $k$, then $\sigma^{-1}(W) \subset D$ is a compact codimension-$k$ submanifold of $D$.

\end{lemma}

\begin{proof}
Let $E \subset D$ be a face.
Let $F := \sigma(E) \subset C$ be its image.
The definition of a proper face-submersion is such that $F\subset C$ is a face of $C$.
Furthermore, the restriction of $\sigma$ along $E$ factors,
\[
\sigma_{|E}
\colon 
E
\xra{~\tau_{\sf sub}~}
F
\xra{~\tau_{\sf emb}~}
C
~,
\]
as a submersion onto $F$ followed by the inclusion of $F$.
Given Observation~\ref{t56}, we have that $W$ is transverse to $F$.
Consequently, $\tau_{\sf emb} \pitchfork W$ and $\tau_{\sf emb}^{-1}(W) = F \cap W \subset F$ is a codimension-$k$ submanifold. 
Next, submersions are transverse to any submanifold.
Consequently $\tau_{\sf sub}\pitchfork (\tau_{\sf emb}^{-1}(W))$ and $\tau_{\sf sub}^{-1}( \tau_{\sf emb}^{-1}(W)) = \sigma_{|E}^{-1}(W) \subset E$ is a codimension-$k$ submanifold.
This implies $\sigma$ is transverse to $W$, and that $\sigma^{-1}(W) \subset D$ is a codimension-$k$ submanifold.
The assumed properness of $\sigma$ ensures $\sigma^{-1}(W)$ is compact if $W$ is compact.  

\end{proof}

Let $D \xra{\sigma} C$ be a proper face-submersion.
Lemma~\ref{t57} gives a well-defined map between sets
\begin{equation}\label{e56}
\Sub^k(C)
\xra{~\sigma^{-1}(-)~}
\Sub^k(D)
~,\qquad
(W\subset C)
\longmapsto 
(\sigma^{-1}(W) \subset D)
~.
\end{equation}

\begin{lemma}\label{t58}
Let $D \xra{\sigma} C$ be a proper face-submersion.
The map~(\ref{e56}) is continuous.
\end{lemma}

\begin{proof}
We first prove~(\ref{e56}) is continuous at each element $(W \subset C) \in \Sub^k(C)$.
Through the transversality statement of Lemma~\ref{t57}, and the fact that normal bundles pull back along transverse maps, there is a canonical isomorphism between vector bundles over $\sigma^{-1}(W)$:
\[
\nu_{\sigma^{-1}(W) \subset D}
\overset{\alpha}{~\cong~}
(\sigma_{|\sigma^{-1}(W)})^\ast \nu_{W \subset C}
~.
\]
Such an isomorphism determines a canonical map
\begin{equation}\label{e58}
\Gamma^{\sf sm}(\nu_{W \subset C})
\longrightarrow
\Gamma^{\sf sm}\Bigl(
{\sigma_{|\sigma^{-1}(W)}}^\ast \nu_{W \subset C}
\Bigr)
~\cong~
\Gamma^{\sf sm}( \nu_{\sigma^{-1}(W) \subset D} )
~.
\end{equation}
Next, choose compatible tubular neighborhoods for $W \subset C$ and for $\sigma^{-1}(W) \subset D$.
Specifically, choose smooth open embeddings $E(\nu_{\sigma^{-1}(W) \subset D}) \xra{\psi^D} D$ and $E(\nu_{W \subset C}) \xra{\psi^C} C$ between manifolds with corners, respectively under $\sigma^{-1}(W)$ and under $W$, such that the diagram
\[
\begin{tikzcd}
	{E(\nu_{\sigma^{-1}(W) \subset D})} & D \\
	{E(\nu_{W \subset C})} & C
	\arrow["{\psi^D}", from=1-1, to=1-2]
	\arrow["{\pr \circ \alpha}"', from=1-1, to=2-1]
	\arrow["\sigma", from=1-2, to=2-2]
	\arrow["{\psi^C}", from=2-1, to=2-2]
\end{tikzcd}
\]
commutes and is a pullback.
By way of this commutative pullback diagram, observe that the diagram among sets
\begin{equation}\label{x33}
\begin{tikzcd}
	{\Gamma^{\sf sm}(\nu_{W \subset C})} & {\Gamma^{\sf sm}( \nu_{\sigma^{-1}(W) \subset D} )} \\
	{\Sub^k(C)} & {\Sub^k(D)}
	\arrow["{(\ref{e58})}", from=1-1, to=1-2]
	\arrow["{{\sf Image} \circ \psi^C_\ast}"', from=1-1, to=2-1]
	\arrow["{{\sf Image} \circ \psi^D_\ast}", from=1-2, to=2-2]
	\arrow["{\sigma^{-1}(-)}", from=2-1, to=2-2]
\end{tikzcd}
\end{equation}
commutes.
Lemma~\ref{t52}(1) states that the vertical maps in this diagram are open embeddings, which carry the respective zero-sections to $(W \subset C)$ and to $(\sigma^{-1}(W) \subset D)$.
Continuity of~(\ref{e58}) therefore implies continuity of $\sigma^{-1}(-)$ about $W$, as desired.  

\end{proof}

\begin{lemma}\label{t61}
Let $D \xra{\sigma} C$ be a proper face-embedding.
The map $\Sub^k(C) \xra{\sigma^{-1}(-)} \Sub^k(D)$ is a Serre fibration.

\end{lemma}

\begin{proof}
Let $W$ be a compact manifold with corners.
For $W \xra{f} C$ an embedding,
denote the pullback diagram
\[
\begin{tikzcd}
	{\sigma^\ast W} & D \\
	W & C
	\arrow["{\sigma^\ast f}", from=1-1, to=1-2]
	\arrow["{\sigma_{|W}}"', from=1-1, to=2-1]
	\arrow["\lrcorner"{anchor=center, pos=0.125}, draw=none, from=1-1, to=2-2]
	\arrow["\sigma", from=1-2, to=2-2]
	\arrow["f", from=2-1, to=2-2]
\end{tikzcd}
~.
\]
Lemma~\ref{t57} implies $\sigma^\ast W$ is a manifold with corners, and $\sigma^\ast f$ is a codimension-$k$ embedding; furthermore, $\sigma_{|W}$ is a proper face-embedding.
Therefore, $\sigma_{|W}$ canonically factors
\[
\sigma_{|W}
\colon
\sigma^\ast W
\xra{~\cong~}
W_{\sigma}
\xra{~\rm inclusion~}
W
\]
as a diffeomorphism onto a closed union of faces $W_\sigma \subset W$ that is, itself, a compact manifold with corners.

Let $W_0 \subset W$ be a closed union of faces that is, itself, a compact manifold with corners.
Consider the subset $\Emb^{[W_0]}(W,C)\subset \Emb(W,C)$ consisting of those embeddings $W \xra{f} C$ such that the two closed unions of faces $W_\sigma = W_0$ agree.
Observe the identity between subsets of $D$:
\begin{equation}\label{e62}
\sigma^{-1}(f(W))
~=~ 
(\sigma^\ast f)
( W_\sigma)
~.
\end{equation}
As there are only finite many closed unions of faces of $W$, this subset $\Emb^{[W_0]}(W,C) \subset \Emb(W,C)$ is evidently a union of connected components.
Observe the filler in the diagram among sets,
\begin{equation}\label{e63}
\begin{tikzcd}
	{\Emb^{[W_0]}(W,C)} & {\Emb(W,C)} \\
	{\Emb(W_0,D)} & {\Map^{\sf emb}(W_0,C)}
	\arrow["{{\rm inclusion}}", from=1-1, to=1-2]
	\arrow[dashed, from=1-1, to=2-1]
	\arrow["{{\rm restriction}}", from=1-2, to=2-2]
	\arrow["{\sigma \circ -}", from=2-1, to=2-2]
\end{tikzcd}
~
\end{equation}
forming a pullback.
Here, $\Map^{\sf emb}(W_0,C) \subset \Map^{\sf sm}(W,C)$ is the subspace consisting of those smooth maps that are locally of the form $\{0\} \times \RR^j \times \RR_{\geq 0}^S \hookrightarrow \RR^k \times \RR^j \times \RR_{\geq 0}^d$ for some $j,d \geq 0$ and some finite subset $S \subset \{1,\dots,d\}$.\footnote{
Recall that the proper face-embedding $D\hookrightarrow C$ is not an embedding in the sense of Definition~\ref{d4}. Consequently, if $W_0\hookrightarrow D$ is an embedding, then the composite $W_0\hookrightarrow D\hookrightarrow C$ will only satisfy the local condition given, defining an element of $\Map^{\sf emb}(W_0,C)$ but \emph{not} of $\Emb(W_0,C)$.
}
Because $\sigma$ is an embedding, the lower map is a monomorphism between topological spaces.  
The two other solid maps are manifestly continuous.
Therefore, the dashed function is continuous, and the square is a pullback among topological spaces.
It follows that the function
\begin{equation}\label{e61}
\Emb(W,C)
\longrightarrow
\underset{W_0 \subset W} \coprod
\Emb(W_0,D)
~,\qquad
(W \xra{f} C)
\longmapsto 
\Bigl( 
W_\sigma \cong \sigma^\ast \xra{\sigma^\ast f} D
\Bigr)
~,
\end{equation}
is continuous,
where the coproduct is indexed by the finite set of closed unions of faces that are, themselves, compact manifolds with corners.  
Taking coproducts gives a map
\begin{equation}\label{e64}
\underset{[W]} \coprod \Emb(W,C)
\xra{~\underset{[W]} \coprod (\ref{e61})~}
\underset{[W_0 \subset W]} \coprod \Emb(W_0,D)
\xra{~{\rm forget}~W{\rm s}~}
\underset{[W_0]} \coprod \Emb(W_0,D)
~.
\end{equation}

Now, the isotopy extension theorem (e.g., \cite{palais}) implies the right vertical map in~(\ref{e63}) is a fiber bundle.
Because the square is a pullback, the left vertical map in~(\ref{e63}) is a fiber bundle as well.
In fact, as in Palais's proof of the isotopy extension theorem in \cite{palais}, each of these embedding spaces is locally contractible.\footnote{Given $(W_0 \xra{f} D) \in \Emb(W_0,D)$, a basis for the topology of $\Emb(W_0,D)$ about $f_0$ consisting of contractible open subsets can be as follows.
Choose a complete Riemannian metric on $C$ that is a product metric in a neighborhood of each face of $C$.
This Riemannian metric pulls back along $\sigma$ as one on $D$, with respect to which $\sigma$ is a totally geodesic isometric embedding.
The exponential map defines a map $\Gamma^{\sf sm}( f_0^\ast \tau_D) \xra{\sf exp} \Emb(W_0,D)$ that carries the zero-section to $f_0$.
This map restricts as an open embedding along some convex open neighborhood about the zero-section.
Through this open embedding, the images of all convex open subsets about the zero-section supplies the desired basis about $f_0$.} 
Therefore, there exists an open cover of $\Emb(W_0,D)$ by contractible subspaces.
Such an open cover is necessarily a trivializing open cover for any fiber bundle over $\Emb(W_0,D)$.
We conclude that the map~(\ref{e64}) is a fiber bundle.

The identity~(\ref{e62}) implies the diagram
\[
\begin{tikzcd}
	{\underset{[W]} \coprod \Emb(W,C)} && {\underset{[W_0]} \coprod \Emb(W_0,D)} \\
	{\Sub^k(C)} && {\Sub^k(D)}
	\arrow["{(\ref{e64})}", from=1-1, to=1-3]
	\arrow["{{\sf Image}}"', from=1-1, to=2-1]
	\arrow["{{\sf Image}}", from=1-3, to=2-3]
	\arrow["{\sigma^{-1}(-)}", from=2-1, to=2-3]
\end{tikzcd}
\]
commutes. 
Lemma~\ref{t52}(2) states that the vertical maps in this diagram are fiber bundles.  
Because the composition of fiber bundles is a fiber bundle, we conclude that the map $\underset{[W]} \coprod \Emb(W,C) \to \Sub^k(D)$ is a fiber bundle.
Lemma~\ref{t62}, stated and proved below, then completes this proof.

\end{proof}

\begin{lemma}\label{t62}
Let $A \xra{f} B \xra{g} C$ be a pair of composable maps between topological spaces.
Suppose $f$ is surjective.  
If $f$ and $g \circ f$ are Serre fibrations, then $g$ is a Serre fibration.
\end{lemma}

\begin{proof}
Recall that a map is a Serre fibrations if and only if it has the right lifting property with respect to the inclusion $\DD^q \xra{0} \DD^q \times I$ for all $q \geq 0$, where $I = [0,1]$ is the closed unit interval and the map $0$ is the inclusion of the subspace $\DD^q \times \{0\}$.
So let $q \geq 0$.
Let $\DD^q \xra{v} B$ and $\DD^q \times I \xra{\w{w}} C$ be maps fitting into a solid commutative diagram among topological spaces:
\[
\begin{tikzcd}
	{\DD^q} & B \\
	{\DD^q \times I} & C
	\arrow["v", from=1-1, to=1-2]
	\arrow["0"', from=1-1, to=2-1]
	\arrow["g", from=1-2, to=2-2]
	\arrow["{\w{v}}", dashed, from=2-1, to=1-2]
	\arrow["{\w{w}}", from=2-1, to=2-2]
\end{tikzcd}
~.
\]
We argue a filler $\w{v}$ exists by referencing the following diagram among topological spaces:
\[
\begin{tikzcd}
	{\SS^{q-1}} & \ast & A \\
	{\SS^{q-1}\times I} & {\DD^q} & B \\
	& {\DD^q \times I} & C
	\arrow["{!}"{description}, from=1-1, to=1-2]
	\arrow["0"{description}, from=1-1, to=2-1]
	\arrow["a"{description}, dotted, from=1-2, to=1-3]
	\arrow["0"{description}, from=1-2, to=2-2]
	\arrow["f"{description}, from=1-3, to=2-3]
	\arrow["{\w{a}}"{pos=0.3}, shift right=3, dotted, from=2-1, to=1-3]
	\arrow["s"{description}, from=2-1, to=2-2]
	\arrow["u"{description}, dotted, from=2-2, to=1-3]
	\arrow["v"{description}, from=2-2, to=2-3]
	\arrow["0"{description}, from=2-2, to=3-2]
	\arrow["g"{description}, from=2-3, to=3-3]
	\arrow["{\w{u}}"{description, pos=0.7}, dotted, from=3-2, to=1-3]
	\arrow["{\w{v}}"{description}, dashed, from=3-2, to=2-3]
	\arrow["{\w{w}}"{description}, from=3-2, to=3-3]
\end{tikzcd}
~.
\]
Here, the map $\SS^{q-1} \times I \xra{s} \DD^q$ is the sphereical coordinates map: $s(p,t) := tp$.  
The upper left square is therefore a pushout.
Now, using that $f$ is assumed surjective, choose a map $\ast \xra{a} A$ making the upper right square commute -- such a choice is simply a preimage of $v(0) \in B$.
Using that Serre fibrations have the right lifting property with respect to the map $\SS^{q-1} \xra{0} \SS^{q-1} \times I$, choose a map $\SS^{q-1} \times I \xra{\w{a}} A$ making the upper left triangle commute.  
Using that the upper left square is a pushout, the upper left triangle determines a map $u$ filling the upper right square.  
Using that $g\circ f$ is assumed a Serre fibration, choose a map $\w{u}$ as in the commutative diagram.
Finally, take $\w{v} := f \circ \w{u}$.
\end{proof}

\subsection{Codimension-$k$ tangles in $\Delta^\bullet \times M$}

In this subsection, we fix $n,k\geq 0$ and an $(n+k)$-manifold $M$ with boundary.
The key construction here is that of a simplicial topological space $\Bord^k_\bullet(M)$ whose topological space of $[p]$-points is the moduli space $\Sub^k(\Delta^p \times M)$ of compact codimension-$k$ submanifolds of $\Delta^p \times M$.

\begin{observation}\label{t60}
    Let $M$ be a manifold with boundary.
    For morphism $[p] \xra{\sigma} [q]$ in $\bDelta$, the map
    \[
    \Delta^p \times M
    \xra{~ \sigma _\ast \times M~}
    \Delta^q \times M
    \]
    is a proper face-submersion. 
\end{observation}

\begin{notation}
\label{d2}
For each $p\geq 0$, denote the topological space
    \[
    \Bord^k_p(M)
    ~:=~
    \Sub^k(\Delta^p \times M)
    ~.
    \]
   
\end{notation}

\begin{lemma}
    \label{t1}
    The assignment
    \[
    [p]
    \longmapsto 
    \Bord^k_p(M)
    \]
    assembles as a simplicial topological space:
    \[
    \Bord^k_\bullet(M) \colon \bDelta^{\op} \longrightarrow \Top
    ~.
    \]
\end{lemma}

\begin{proof}
What remains is to define this functor on morphisms, then check that the assignment respects composition.  Let $[p] \xra{\sigma}[q]$ be a morphism in $\bDelta$.
Using Observation~\ref{t60}, Lemma~\ref{t58} implies the map
\[
\sigma^\ast
\colon 
\Bord^k_q(M) = \Sub^k(\Delta^q \times M)
\xra{~(\sigma_\ast \times M)^{-1}(-)~}
\Sub^k(\Delta^p \times M)
=
\Bord^k_p(M)
~,
\]
\begin{equation}\label{e1}
(W \subset \Delta^q \times M)
\longmapsto
\Bigl(
(\sigma_\ast \times M)^{-1}(W) \subset \Delta^p \times M
\Bigr)
~,
\end{equation}
is well-defined and continuous.

For $[\ell] \xra{\tau} [p]$ another morphism in $\bDelta$,
observe the identity between subspaces of $\Delta^\ell \times M$:
\[
((\sigma \tau)_\ast \times M )^{-1}(W) \subset \Delta^p \times M
~=~
(\tau_\ast^{-1} \times M
\Bigl(
(\sigma_\ast \times M)^{-1}(W) 
\Bigr)
~.
\]
This identity immediately implies the named assignments respect composition: $(\sigma \circ \tau)^\ast = \tau^\ast \circ \sigma^\ast$.

\end{proof}

\subsection{Codimension-$\xi$ submanifolds}\label{sec.universal}

In this subsection, we fix non-negative integers $p,n,k\geq 0$, and we fix a $(p+n+k)$-manifold with corners $C$.
Here, we construct a universal fiber bundle $\w{\Sub}^k(C)$ over the moduli space $\Sub^k(C)$, as it is equipped with a normal vector bundle.
We use this to define the moduli space $\Sub^\xi(C)$ of compact codimension-$\xi$ submanifolds of $C$.

Consider the subspace
\[
\w{\Sub}^k(C)
~:=~
\left\{
\Bigl( 
W \subset C
,
x \in C \Bigr)
\mid 
x\in W
\right\}
~\subset~
\Sub^k(C) \times C
~.
\]
Projection onto the first coordinate defines a map
\[
{\sf fgt}
\colon 
\w{\Sub}^k(C)
\longrightarrow
\Sub^k(C)
~,\qquad
(W \subset C , x \in W)
\longmapsto 
(W\subset C)
~,
\]
whose fiber over $(W\subset C)$ is $W$.
Evaluation defines a map
\begin{equation}
\label{e19}
\ev\colon
\Emb(W,C) \times W
\longrightarrow
\w{\Sub}^k(C)
~,\qquad
(f,x)
\longmapsto 
\Bigl( f(W) , f(x) \Bigr)
~, 
\end{equation}
which fits into a commutative diagram among topological spaces:
\[
\begin{tikzcd}
	{\Emb(W,C)\times W} && {\w{\Sub}^k(C)} \\
	{\Emb(W,C)} && {\Sub^k(C)}
	\arrow["\ev", from=1-1, to=1-3]
	\arrow["\pr"', from=1-1, to=2-1]
	\arrow["{{\rm forget}}", from=1-3, to=2-3]
	\arrow["{{\sf Image}}", from=2-1, to=2-3]
\end{tikzcd}
~.
\]

\begin{observation}
    \label{t19}

\begin{enumerate}
    \item[]

\item The topology of $\w{\Sub}^k(C)$ is the finest for which, for each compact $(p+n)$-manifold $W$ with corners, the map~(\ref{e19}) is continuous.
Consequently, the canonical map,
\[
\underset{[W]} \coprod \Emb(W,C) \underset{\Diff(W)} \times W
\xra{~\cong~}
\w{\Sub}^k(C)
~,
\]
is a homeomorphism.

\item 
In particular, after Lemma~\ref{t52}(2), the map
\[
\w{\Sub}^k(C)
\to 
\Sub^k(C)
\]
has the structure of a fiber bundle of compact $(p+n)$-manifolds with corners.

\item Furthermore, the topological subspace
\begin{equation}\label{e33}
\w{\Sub}^k(C)
~\hookrightarrow~
\Sub^k(C) \times C
\end{equation}
has the structure of a fiber bundle over $\Sub^k(C)$ of compact comimension-$k$ submanifolds of $C$.

\end{enumerate}
\end{observation}

The fiberwise (over $\Sub^k(C)$) normal bundle of~(\ref{e33}) is a rank-$k$ vector bundle over $\w{\Sub}^k(C)$:
\begin{equation}\label{e37}
\nu^{\sf fib}(C)
~:=~
\Bigl(
E(\nu^{\sf fib}(C))
\to 
\w{\Sub}^k(C)
\Bigr)
~;
\end{equation}
By construction, its fiber over $(x\in W\subset C)$ is the normal vector space $\sN_x$ of $W \subset C$ at $x$.
This rank-$k$ vector bundle is classified by a map between spaces:
\begin{equation}\label{e34}
\nu^{\sf fib}(C)
\colon
\w{\Sub}^k(C)
\longrightarrow
\BO(k)
~.
\end{equation}

\begin{observation}
The fiber bundle of Observation~\ref{t19}(2) is classified by the map~(\ref{e18}).  
Consequently, the map~(\ref{e34}) supplies a lift among spaces:
\[
\begin{tikzcd}
	&& {\underset{[W]} \coprod \Map(W,\BO(k))_{/\Diff(W)}} \\
	{\Sub^k(C)} && {\underset{[W]} \coprod \BDiff(W)}
	\arrow[from=1-3, to=2-3]
	\arrow["{\nu^{\sf fib}(C)}", from=2-1, to=1-3]
	\arrow["{(\ref{e18})}"', from=2-1, to=2-3]
\end{tikzcd}
~.
\]
\end{observation}

\begin{definition}
\label{d5}
    The space of \bit{codimension-$\xi$ submanifolds} of $C$ is the pullback among spaces:
    \[
\begin{tikzcd}
	{\Sub^\xi(C)} & {\underset{[W]} \coprod \Map(W,B)_{/\Diff(W)}} \\
	{\Sub^k(C)} & {\underset{[W]} \coprod \Map(W,\BO(k))_{/\Diff(W)}}
	\arrow[from=1-1, to=1-2]
	\arrow[from=1-1, to=2-1]
	\arrow[from=1-2, to=2-2]
	\arrow["{\nu^{\sf fib}(C)}", from=2-1, to=2-2]
\end{tikzcd}
~.
    \]
    
\end{definition}

\begin{remark}
Explicitly, a point in $\Sub^k(C)$ is a codimension-$k$ submanifold $W \subset C$, together with a map $g\colon W \ra B$, as well as an isomorphism $\nu_{W \subset C} \cong g^\ast \xi$ between vector bundles over $W$.
\end{remark}

\begin{observation}\label{a27}
In the case that $\xi = \gamma_k$ is the universal rank-$k$ vector bundle over $\BO(k)$, the canonical map between spaces
\[
\Sub^{\gamma_k}(C)
\xra{~\simeq~}
\Sub^k(C)
\]
is an equivalence since the right vertical map in Definition~\ref{d5} is an equivalence for $B=\BO(k)$.
\end{observation}

    Forgetting smooth structure defines a functor between $\infty$-categories, $\BDiff(W) \to \Spaces$, for each compact $n$-manifold $W$ with corners.  
    Post-composing the map~(\ref{e18}) with such functors defines a functor between $\infty$-categories,
    \[
    \Sub^k(C)
    \longrightarrow 
    \Spaces
    ~,
    \]
    which classifies the fibration underlying the fiber bundle of Observation~\ref{t19}(2).
    The map~(\ref{e34}), in turn, supplies a lift of this functor:
    \begin{equation}\label{e55}
    \Sub^k(C)
    \xra{~\nu^{\sf fib}(C)~}
    \Spaces_{/\BO(k)}
    ~.
    \end{equation}
    This lift affords an alternative description of the space $\Sub^\xi(C)$ as the pullback among $\infty$-categories:
    \[
\begin{tikzcd}
	{\Sub^\xi(C)} & {\Spaces_{/B}} \\
	{\Sub^k(C)} & {\Spaces_{/\BO(k)}}
	\arrow[from=1-1, to=1-2]
	\arrow[from=1-1, to=2-1]
	\arrow["\lrcorner"{anchor=center, pos=0.125}, draw=none, from=1-1, to=2-2]
	\arrow["{\xi \circ -}", from=1-2, to=2-2]
	\arrow["{\nu^{\sf fib}(C)}", from=2-1, to=2-2]
\end{tikzcd}
~.
    \]


\subsection{Codimension-$\xi$ tangles in $\Delta^\bullet \times M$}\label{sec.Un}

In this subsection, we fix $n,k\geq 0$ and an $(n+k)$-manifold $M$ with boundary; we fix a rank-$k$ vector bundle $\xi = (E \to B)$, and we denote the map classifying it again as $B \xra{\xi} \BO(k)$.
The key construction here is that of a simplicial space $\Bord^\xi_\bullet(M)$ whose space of $[p]$-points is the moduli space $\Sub^\xi(\Delta^p \times M)$ of compact codimension-$\xi$ submanifolds of $M$.

The proof of the following technical result is~\S\ref{sec.normal}.
Its statement references the $\infty$-category 
\[
\bDelta_{/\Bord^k_\bullet(M)}
~, 
\]
which is the $\infty$-overcategory of $\Bord^k_\bullet(M)$ with respect to the Yoneda functor $\bDelta \to \PShv(\bDelta)$.
The canonical projection 
\[
\bDelta_{/\Bord^k_\bullet(M)}
\xra{~\rm forget~}
\bDelta
\]
is therefore a right fibration; it is the unstraightening of the functor $\bDelta^{\op} \xra{\Bord^k_\bullet(M)} \Spaces$.
In particular, its fiber over $[p]$ is the space $\Bord^k_p(M)$ of $p$-simplices of the simplicial space $\Bord^k_\bullet(M)$.
\begin{lemma}\label{t50}
There is a functor
\[
\nu^{\sf fib}_\bullet(M)
\colon
\bDelta_{/\Bord^k_\bullet(M)}
\longrightarrow
\Spaces_{/\BO(k)}
\]
whose restriction to the fiber of $\bDelta_{/\Bord^k_\bullet(M)} \xra{\rm forget} \bDelta$ over $[p]$ is the map 
\[
\nu^{\sf fib}_p(M)
\colon
\Bord^k_p(M) = \Sub^k(\Delta^p \times M) \xra{\nu^{\sf fib}(\Delta^p\times M)} \Spaces_{/\BO(k)}
\]
from~(\ref{e55}).
\end{lemma}

\begin{notation}
    We denote the pullback $\infty$-category:
    \[
\begin{tikzcd}
	{\Un^\xi(M)} & {\Spaces_{/B}} \\
	{\bDelta_{/\Bord^k_\bullet(M)}} & {\Spaces_{/\BO(k)}}
	\arrow[from=1-1, to=1-2]
	\arrow[from=1-1, to=2-1]
	\arrow["\lrcorner"{anchor=center, pos=0.125}, draw=none, from=1-1, to=2-2]
	\arrow["{\xi \circ -}", from=1-2, to=2-2]
	\arrow["{\nu^{\sf fib}_\bullet(M)}", from=2-1, to=2-2]
\end{tikzcd}
~.
    \]
\end{notation}

\begin{observation}
\label{t22}
    The composite functor $\Un^\xi(M) \to \bDelta_{/\Bord^k_\bullet(M)} \to \bDelta$ is right fibration.  Indeed, the base-change of a right fibrations is a right fibration, and the composition of right fibrations is a right fibration.  
\end{observation}

\begin{definition}\label{def.Bord.s.space}
The simplicial space of \bit{codimension-$\xi$ tangles in $\Delta^\bullet \times M$},
\[
\Bord^\xi_\bullet(M)
\colon 
\bDelta^{\op}
\longrightarrow
\Spaces
~,
\]
is that classifying the right fibration $\Un^\xi(M) \to \bDelta$. 
\end{definition}

\begin{remark}
    Equivalently, $\Bord^\xi_\bullet(M)$ is the simplicial space equipped with an identification $\Un^\xi(M) \simeq \bDelta_{/\Bord^\xi_\bullet(M)}$ over $\bDelta$. That is, the unstraightening of $\Bord^\xi_\bullet(M)$ is $\Un^\xi(M) \to \bDelta$.
\end{remark}

\begin{observation}
\label{t25}
    By Definition~\ref{d5} as a pullback, Lemma~\ref{t50} reveals, for each $p \geq 0$, a canonical identification between spaces:
    \[
    \Bord^\xi_p(M)
    ~\simeq~
    \Sub^\xi(\Delta^p \times M)
    ~.
    \]
    
\end{observation}

\begin{remark}
A 0-simplex in $\Bord_\bullet^{\xi}(M)$ is a compact $n$-submanifold $W\subset M$, together with a map $g\colon W \ra B$, and an isomorphism between vector bundles $\nu_{W \subset M} \cong g^\ast \xi$ from the normal bundle to the pullback bundle.
A 1-simplex in $\Bord^\xi_\bullet(M)$ is an embedded cobordism between such. Generally, a $p$-simplex in $\Bord_\bullet^{\xi}(M)$ is a compact $(n+p)$-submanifold $W\subset \Delta^p \times M$, together with a map $g\colon W \ra B$, and an isomorphism between vector bundles $\nu_{W \subset \Delta^p \times M} \cong g^\ast \xi$ from the normal bundle to the pullback bundle.

\end{remark}

\begin{observation}\label{a26}
In the case that $\xi = \gamma_k$ is the universal rank-$k$ vector bundle over $\BO(k)$, Observation~\ref{a27} implies the canonical morphism between simplicial spaces
\[
\Bord^{\gamma_k}_\bullet(M)
\xra{~\simeq~}
\Bord^k_\bullet(M)
\]
is an equivalence.
\end{observation}

\subsection{Proof of Lemma~\ref{t50}}\label{sec.normal}

A priori, constructing a functor to the $\infty$-category $\Spaces_{/\BO(k)}$ requires an infinite amount of homotopy coherence.  
For the case at hand, we exploit the geometry intrinsic to $\Bord^k_\bullet(M)$ to resolve this a priori infinite homotopy coherence problem.  
Specifically, we implement the following. 
\begin{itemize}
    \item We construct a topological category $\Un^k(M)$ whose underlying $\infty$-category is $\bDelta_{/\Bord^k_\bullet(M)}$.

    \item We construct a topological category $\w{\Un}^k(M)$ equipped with a functor to $\Un^k(M)$ whose underlying functor between $\infty$-categories is a left fibration.

    \item We construct a vector bundle $\nu^{\sf fib}_\bullet(M)$ over $\w{\Un}^k(M)$, in the sense of~\S\ref{sec.vbdl}.  
\end{itemize}
The vector bundle $\nu^{\sf fib}_\bullet(M)$ is classified by a functor from the underlying $\infty$-category, $\w{\Un}^k(M) \to \BO(k)$, which thereafter straightens as a functor from the underlying $\infty$-category, $\bDelta_{/\Bord^k_\bullet(M)} \simeq \Un^k(M) \to \Spaces_{/\BO(k)}$.
Tracking this construction reveals that its restriction to $\Bord^k_p(M)$ is as desired.

Consider the Grothendieck construction (also known as the unstraightening) of the functor $\bDelta^{\op} \xra{\Bord^k_\bullet(M)} \Top$ of Lemma~\ref{t1}:
\[
\Un^k(M)
:=
\Un\Bigl(\Bord^k_\bullet(M)\Bigr)
\longrightarrow
\bDelta
~.
\]
Explicitly, it is a category-object in $\Top$ over $\bDelta$ given as follows. 
\begin{itemize}
    \item Its topological space of objects over an object $[p]\in \bDelta$ is
\[
\Obj\Bigl( 
\Un^k(M)
\Bigr)_{|[p]}
~:=~
\Bord^k_p(M)
~.
\]

\item Its topological space of morphisms over a morphism $[p]\xra{\sigma}[q]$ in $\bDelta$ is
\[
\Mor\Bigl( 
\Un^k(M)
\Bigr)_{|\sigma}
~:=~
\Bord^k_q(M)
~.
\]

\item Its source and target maps,
\[
\Obj\Bigl( 
\Un^k(M)
\Bigr)
\xla{~\rm source~}
\Mor\Bigl( 
\Un^k(M)
\Bigr)
\xra{~\rm target~}
\Obj\Bigl( 
\Un^k(M)
\Bigr)
~,
\]
are given over a morphism $[p]\xra{\sigma}[q]$ in $\bDelta$ as
\[
\Bord^k_p(M)
\xla{~\sigma^\ast~}
\Bord^k_q(M)
\xra{~=~}
\Bord^k_q(M)
~.
\]

\item The composition rule is the map,
\[
\Mor\Bigl( 
\Un^k(M)
\Bigr)
\underset{\Obj\Bigl( 
\Un^k(M)
\Bigr)} \times
\Mor\Bigl( 
\Un^k(M)
\Bigr)
\xra{~\circ~}
\Mor\Bigl( 
\Un^k(M)
\Bigr)
\]
given over a composable pair of morphisms $[p]\xra{\sigma}[q] \xra{\tau} [r]$ in $\bDelta$ by
\[
\Bord^k_r(M)
\xra{~\tau^\ast~}
\Bord^k_q(M)
~.
\]

\end{itemize}

\begin{observation}
\begin{enumerate}
    \item[]

    \item 
    The target morphism for $\Un^k(M)$ is evidently a covering space.
    Therefore $\Un^k(M)$ is a topological category in the sense of~\S\ref{sec.vbdl}.

    \item The functor $\Un^k(M) \to \bDelta$ is a right fibration between topological categories, which is to say commutative square of topological spaces
\[
\begin{tikzcd}
	{\Mor(\Un^k(M))} & {\Obj(\Un^k(M))} \\
	{\Mor(\bDelta)} & {\Obj(\bDelta)}
	\arrow["{{\rm target}}", from=1-1, to=1-2]
	\arrow[from=1-1, to=2-1]
	\arrow[from=1-2, to=2-2]
	\arrow["{{\rm target}}", from=2-1, to=2-2]
\end{tikzcd}
\]
is a pullback.

    \item The functor between $\infty$-categories underlying $\Un^k(M) \to \bDelta$ is the right fibration $\bDelta_{/\Bord^k_\bullet(M)} \to \bDelta$ classified by the simplicial space $\bDelta^{\op} \xra{\Bord^k_\bullet(M)} \Top \xra{\rm forget} \Spaces$.  
    \end{enumerate}
\end{observation}

The universal fiber bundles from \S\ref{sec.universal} are functorial in a certain sense.
For starters, for each morphism $[p] \xra{\sigma} [q]$ in $\bDelta$, we have a commutative diagram among topological spaces,
\begin{equation}\label{e36}
\begin{tikzcd}
	{\w{\Bord}^k_p(M)} & {\w{\Bord}^k_p(M) \underset{\Bord^k_p(M)} \times \Bord^k_q(M)} & {\w{\Bord}^k_q(M)} \\
	{\Bord^k_p(M)} & {\Bord^k_q(M)} & {\Bord^k_q(M)}
	\arrow[from=1-1, to=2-1]
	\arrow["\pr"', from=1-2, to=1-1]
	\arrow["{\sigma_\ast \times M}", from=1-2, to=1-3]
	\arrow["\pr", from=1-2, to=2-2]
	\arrow[from=1-3, to=2-3]
	\arrow["{\sigma^\ast}"', from=2-2, to=2-1]
	\arrow["{=}", from=2-2, to=2-3]
\end{tikzcd}
~,
\end{equation}
in which the left square is a pullback and the upper rightward map is given by
\[
\sigma_\ast \times M
\colon
\Bigl( W \subset \Delta^q \times M , x \in (\sigma_\ast \times M)^{-1}(W) \Bigr)
\longmapsto
\Bigl( W \subset \Delta^q \times M , (\sigma_\ast \times M)(x) \in W \Bigr)
~.
\]
Better, this functoriality may be organized as a category-object in $\Top$ over 
$\Un^k(M)$,
\[
\w{\Un}^k(M)
\longrightarrow
\Un^k(M)
~,
\]
which we now define.
\begin{itemize}
    \item Its topological space of objects over an object $[p]\in \bDelta$ is
\[
\Obj\Bigl( 
\w{\Un}^k(M)
\Bigr)_{|[p]}
~:=~
\w{\Bord}^k_p(M)
~.
\]

\item Its topological space of morphisms over a morphism $[p]\xra{\sigma}[q]$ in $\bDelta$ is
\[
\Mor\Bigl( 
\w{\Un}^k(M)
\Bigr)_{|\sigma}
~:=~
\w{\Bord}^k_p(M)
\underset{\Bord^k_p(M)} \times
\Bord^k_q(M)
~.
\]

\item Its source and target maps,
\[
\Obj\Bigl( 
\w{\Un}^k(M)
\Bigr)
\xla{~\rm source~}
\Mor\Bigl( 
\w{\Un}^k(M)
\Bigr)
\xra{~\rm target~}
\Obj\Bigl( 
\w{\Un}^k(M)
\Bigr)
~,
\]
are given over a morphism $[p]\xra{\sigma}[q]$ in $\bDelta$ as
\[
\w{\Bord}^k_p(M)
\xla{~\pr~}
\w{\Bord}^k_p(M)
\underset{\Bord^k_p(M)} \times
\Bord^k_q(M)
\xra{~\sigma_\ast\times M~}
\w{\Bord}^k_q(M)
~.
\]

\item The composition rule is the map,
\[
\Mor\Bigl( 
\w{\Un}^k(M)
\Bigr)
\underset{\Obj\Bigl( 
\w{\Un}^k(M)
\Bigr)} \times
\Mor\Bigl( 
\w{\Un}^k(M)
\Bigr)
\xra{~\circ~}
\Mor\Bigl( 
\w{\Un}^k(M)
\Bigr)
\]
given over a composable pair of morphisms $[p]\xra{\sigma}[q] \xra{\tau} [r]$ in $\bDelta$ by
\[
\w{\Bord}^k_p(M)
\underset{\Bord^k_p(M)} \times
\Bord^k_r(M)
\xra{~ \id \times \tau^\ast~}
\w{\Bord}^k_p(M)
\underset{\Bord^k_p(M)} \times
\Bord^k_q(M)
~.
\]

\end{itemize}

\begin{observation}\label{t64}
The projection $\w{\Bord}^k_p(M) \to \Bord^k_p(M)$ for each $p \geq 0$ defines a morphism $\w{\Un}^k(M) \to \Un^k(M)$ between category-objects in $\Top$.
Furthermore, the definition of $\w{\Un}^k(M)$ is such that $\w{\Un}^k(M) \to \Un^k(M)$ is a left fibration between category-objects internal to $\Top$, which is to say the commutative square of topological spaces
\[
\begin{tikzcd}
	{\Mor(\w{\Un}^k(M))} & {\Obj(\w{\Un}^k(M))} \\
	{\Mor(\Un^k(M))} & {\Obj(\Un^k(M))}
	\arrow["{{\rm source}}", from=1-1, to=1-2]
	\arrow[from=1-1, to=2-1]
	\arrow[from=1-2, to=2-2]
	\arrow["{{\rm source}}", from=2-1, to=2-2]
\end{tikzcd}
\]
is a pullback.  
\end{observation}

\begin{lemma}
\label{t51}
    $\w{\Un}^k(M)$ is a topological category in the sense of~\S\ref{sec.vbdl}.
    Moreover, the functor between $\infty$-categories underlying $\w{\Un}^k(M) \to \Un^k(M)$ is a left fibration.
    Furthermore, this left fibration is classified by a functor 
    \begin{equation}\label{e39'}
    \Un^k(M)
    \simeq
    \bDelta_{/\Bord^k_\bullet(M)}
    \longrightarrow
    \Spaces
    \end{equation}
    whose restriction to $\Bord^k_p(M) \subset \Un^k(M)$ classifies the fibration $\w{\Bord}^k_p(M) \to \Bord^k_p(M)$ and whose coCartesian monodromy functors are given by the maps $\sigma_\ast \times M$ from~(\ref{e36}).  
\end{lemma}
\begin{proof}
Provided $\w{\Un}^k(M)$ is a topological category, Observation~\ref{t64} then implies $\w{\Un}^k(M) \to \Un^k(M)$ is a left fibration between $\infty$-categories.
The description of spaces over $\Bord^k_p(M)$ then follows directly from the definition of $\w{\Un}^k(M)$.
So it remains to verify $\w{\Un}^k(M)$ is a topological category.

We must prove, for each $r\geq 0$, that the canonical map
\begin{equation}\label{e66}
\Hom([r],\w{\Un}^k(M))
\longrightarrow
\Hom^{\sf h}([r],\w{\Un}^k(M))
\end{equation}
is a weak homotopy equivalence.  
The continuous functor $\w{\Un}^k(M) \to \Un^k(M)$ determines a commutative diagram among topological spaces:
\begin{equation}\label{e65}
\xymatrix{
\Hom([r],\w{\Un}^k(M))
\ar[rr]
\ar[d]
&&
\Hom^{\sf h}([r],\w{\Un}^k(M))
\ar[d]
\\
\Hom([r],\Un^k(M))
\ar[rr]
&&
\Hom^{\sf h}([r],\Un^k(M))
.
}
\end{equation}
Unwinding the definition of the continuous functor $\w{\Un}^k(M) \to \Un^k(M)$, Observation~\ref{t19}(3) reveals that left vertical map in this square is a fiber bundle, and its fiber over such a point is the topological space $W$.  
Now, because $\w{\Un}^k(M) \to \Un^k(M)$ is a left fibration between category-objects in $\Top$, for each $0\leq i < r$, the commutative square among topological spaces
\[
\xymatrix{
\Hom(\{i<i+1\},\w{\Un}^k(M))
\ar[rr]
\ar[d]
&&
\Hom(\{i\} , \w{\Un}^k(M))
\ar[d]
\\
\Hom(\{i<i+1\},\Un^k(M))
\ar[rr]
&&
\Hom(\{i\}, \Un^k(M))
}
\]
is a pullback.  
Furthermore, Observation~\ref{t19}(3) implies the vertical maps in this diagram are fiber bundles.  
Consequently, the canonical map from each fiber of the left vertical map in~(\ref{e65}) to the corresponding fiber of the right vertical map in~(\ref{e65}) is a weak homotopy equivalence.
The 5-Lemma then implies the map~(\ref{e66}) is a weak homotopy equivalence, as desired.  
    
\end{proof}

Recall the following standard fact about transversality.
\begin{fact}\label{trans.fact}
Let $B \xra{\sigma} A$ be smooth maps between manifolds with corners.  
    Let $W \subset A$ be a submanifold.
Suppose $\sigma$ is transverse to $W$.
\begin{enumerate}
    \item 
    Then $\sigma^{-1}(W) \subset B$ is a submanifold, and the restriction along $\sigma^{-1}(W) \subset B$ of the morphism between tangent bundles $\sT B\xra{\sD \sigma} \sigma^\ast \sT A$ over $B$ descends as an isomorphism between normal bundles over $\sigma^{-1}(W)$:
    \[
    \nu_{\sigma^{-1}(W) \subset B} \xra{\cong} \sigma^\ast \nu_{W \subset A}    
    ~,
    \]
    where, through a slight abuse of notation, the vector bundle on the right is $\sigma^\ast \nu_{W \subset A} := {\sigma_{|\sigma^{-1}(W)}}^\ast \nu_{W\subset A}$, the pullback of $\nu_{W \subset A}$ along the restriction $\sigma^{-1}(W) \xra{\sigma_{|\sigma^{-1}(W)}} W$.

    \item Let $C \xra{\tau} B$ be a smooth map between manifolds with corners.
    Then $\tau$ is transverse to $\sigma^{-1}(W)$ if and only if $\sigma \circ \tau$ is transverse to $W$.
    In this case, the resulting diagram among vector bundles over $(\sigma \circ \tau)^{-1}(W)$,
\[
\begin{tikzcd}
	{\nu_{(\sigma \circ \tau)^{-1}(W) \subset C}} & {\nu_{\tau^{-1}(\sigma^{-1}(W)) \subset C}} & {\tau^\ast \nu_{\sigma^{-1}(W) \subset B}} \\
	{(\sigma \circ \tau)^\ast \nu_{W \subset A}} && {\tau^\ast \sigma^\ast \nu_{W \subset A}}
	\arrow["{=}"{description}, draw=none, from=1-1, to=1-2]
	\arrow["\cong"', from=1-1, to=2-1]
	\arrow["\cong", from=1-2, to=1-3]
	\arrow["\cong", from=1-3, to=2-3]
	\arrow["\cong", from=2-1, to=2-3]
\end{tikzcd}
~,
\]
    commutes.

    \end{enumerate}
\end{fact}
The fiberwise normal bundles $\nu^{\sf fib}(\Delta^\bullet \times M)$ from~(\ref{e37}) respect the above functoriality of the universal bundles in the following sense.  
Let $[p]\xra{\sigma}[q]$ be a morphism in $\bDelta$.  
Using that the map $\sigma_\ast \times M$ is transverse to each $(W \subset \Delta^q \times M) \in \Bord^k_q(M)$ (Lemma~\ref{t57}), 
Fact~\ref{trans.fact}(1) yields a canonical isomorphism $\alpha_\sigma$ between pullback vector bundles:
\begin{equation}\label{e38}
\begin{tikzcd}
	{E(\nu^{\sf fib}(\Delta^p\times M)}) & {\pr^\ast E(\nu^{\sf fib}(\Delta^p\times M)) \overset{\alpha_\sigma}\cong (\sigma_\ast \times M)^\ast E(\nu^{\sf fib}(\Delta^q \times M))} & {E(\nu^{\sf fib}(\Delta^q\times M))} \\
	{\w{\Bord}^k_p(M)} & {\w{\Bord}^k_p(M) \underset{\Bord^k_p(M)} \times \Bord^k_q(M)} & {\w{\Bord}^k_q(M)}
	\arrow[from=1-1, to=2-1]
	\arrow[from=1-2, to=1-1]
	\arrow[from=1-2, to=1-3]
	\arrow[from=1-2, to=2-2]
	\arrow[from=1-3, to=2-3]
	\arrow["\pr"', from=2-2, to=2-1]
	\arrow["{\sigma_\ast \times M}", from=2-2, to=2-3]
\end{tikzcd}
~.
\end{equation}
Next, let $[q] \xra{\tau} [r]$ be another morphism in $\bDelta$.  
Consider the span among topological spaces:
\[
\w{\Bord}^k_p(M) \underset{\Bord^k_p(M)} \times \Bord^k_q(M)
\xla{~\pr~}
\w{\Bord}^k_p(M) \underset{\Bord^k_r(M)} \times \Bord^k_r(M)
\xra{~\sigma_\ast \times M~}
\w{\Bord}^k_q(M) \underset{\Bord^k_q(M)} \times \Bord^k_r(M)
~.
\]
Observe that the target of $\pr^\ast \alpha_\sigma$ agrees with the source of $(\sigma_\ast \times M)^\ast \alpha_\tau$.
Fact~\ref{trans.fact}(2) implies 
\[
\Bigl(
(\sigma_\ast \times M)^\ast \alpha_\tau
\Bigr) \circ 
( \pr^\ast \alpha_\sigma ) 
~=~
\alpha_{\tau \circ \sigma}
~.
\]

In the sense of~\S\ref{sec.vbdl}, these data are a rank-$k$ vector bundle over the topological category $\w{\Un}^k(M)$.
This vector bundle is then classified by a functor from the underlying $\infty$-category:
\[
\w{\Un}^k(M)
\xra{~\nu^{\sf fib}(M)~}
\BO(k)
~.
\]
Using that the functor between underlying $\infty$-categories $\w{\Un}^k(M) \to \Un^k(M)$ is a left fibration, the functor $\nu^{\sf fib}(M)$ is the datum of a lift of the functor~(\ref{e39'}) among $\infty$-categories:
    \[
\begin{tikzcd}
	& {\Spaces_{/\BO(k)}} \\
	{\Un^k(M)} & \Spaces
	\arrow["{{\rm forget}}", from=1-2, to=2-2]
	\arrow["\nu^{\sf fib}_\bullet(M)", dashed, from=2-1, to=1-2]
	\arrow["{(\ref{e39'})}"', from=2-1, to=2-2]
\end{tikzcd}
~.
    \]
Unwinding the construction of $\nu^{\sf fib}_\bullet(M)$, its restriction along $\Bord^k_p(M) = \Sub^k(\Delta^p \times M) \subset \Un^k(M)$ is the functor $\nu^{\sf fib}(\Delta^p\times M)$ of~(\ref{e55}).

\section{Parametrized Pontryagin--Thom theory}\label{sec.PT}
In this section, we fix $p,n,k \geq 0$, an $(n+k)$-manifold $M$ with boundary, and a smooth rank-$k$ vector bundle $\xi = (E \to B)$. The following is the key result in this section.
\begin{theorem}
    \label{t26}
    There is a canonical equivalence between simplicial spaces
    \[
    \cMap^\pitchfork_\bullet(M,\Th(\xi))
    \xra{~\simeq~}
    \Bord^\xi_\bullet(M)
    \]
    which is functorial with respect to open embeddings in $M$.
\end{theorem}

\begin{cor}\label{cor.bord.kan}
    The simplicial spaces $\Bord^k_\bullet(M)$ and $\Bord_\bullet^\xi(M)$ both satisfy the Kan condition.

\end{cor}
\begin{proof} 
Through the equivalence between simplicial spaces $\Bord_\bullet^\xi(M)\simeq \cMap_\bullet^{\pitchfork}(M,\Th(\xi))$ of Theorem~\ref{t26}, Lemma~\ref{lemma.Map.transv.Kan} immediately implies $\Bord^\xi_\bullet(M)$ satisfies the Kan condition.  
Observation~\ref{a26} then implies the simplicial space $\Bord^k_\bullet(M)$ satisfies the Kan condition.
\end{proof}

Theorem~\ref{t26} enables the first part of our main theorem (Theorem~\ref{thm.main}). 
The rest of the proof, involving $\dBord^\xi_\bullet(M)$, is completed through Corollary~\ref{cor.dBord.Bord}. 
\begin{proof}[Proof of Theorem~\ref{thm.main} assuming Theorem~\ref{t26}]
Assume $\xi$ is a smooth vector bundle.
The equivalence of the theorem then follows from the following sequences of equivalence between spaces, which we explain below:
\[
    \bigl|\Bord^\xi_\bullet(M)\bigr|
    ~
    \simeq
    ~
   \left |\cMap^\pitchfork_\bullet(M,\Th(\xi)) \right|
    ~
    \simeq
    ~
    \cMap(M,\Th(\xi))
    ~.
\]
The first equivalence is induced by the equivalence of Theorem~\ref{t26}.
The second equivalence is Lemma~\ref{thm.param.transv}.

Next, suppose the base $B$ of $\xi$ has is homotopy equivalent with the geometric realization of a finite simplicial complex $B''$.
Choose a piecewise linear embedding $B''\hookrightarrow \RR^N$ int a large Euclidean space.
A sufficiently small open regular neighborhood of this embedding is a smooth manifold $B'$ with the same homotopy type: $B \simeq B'' \simeq B'$.
Further, by smooth approximation, the pullback of $\xi$ to $B'$ is isomorphic with a smooth vector bundle $\xi'$ over the manifold $B'$. 
The argument above then gives an equivalence $|\Bord^{\xi'}_\bullet(M)|\simeq \cMap(M,\Th(\xi'))$. However, we have evident natural equivalences $|\Bord^{\xi'}_\bullet(M)|\simeq |\Bord^\xi_\bullet(M)|$ and $\cMap(M,\Th(\xi'))\simeq \cMap(M,\Th(\xi))$, from which the result follows.

Lastly, for the general case of $\xi$, its base may be witnessed as a filtered colimit of geometric realizations of finite simplicial complexes: $\colim_{i\in \cI} \xi_i \simeq \xi$.
Compactness of the submanifolds of $\Delta^\bullet \times M$ ensures the canonical map between simplicial spaces $\colim_{i\in \cI} \Bord^{\xi_i}_\bullet(M) \to \Bord^\xi_\bullet(M)$ is an equivalence.
Compact support ensures the canonical map between spaces $\colim_{i\in \cI} \cMap(M,\Th(\xi_i)) \to \cMap(M,\Th(\xi))$ is an equivalence.
Using that geometric realization commutes with filtered colimits, this general case then follows from the previous case.

\end{proof}

\subsection{Parametrized transverse intersection}

In this subsection, we fix a $(p+n+k)$-manifold $C$ with corners.
In this section, we construct a \bit{family} of manifolds over $\cMap^{\pitchfork}(C,\Th(\xi))$ given by taking preimages of the zero-section.

Consider the topological subspace
\[
\w{\cMap}^\pitchfork\bigl(C,\Th(\xi)\bigr)
~:=~
\Bigl\{
(f,x) \mid f(x) \in B
\Bigr\}
~\subset~
\cMap^\pitchfork(C,\Th(\xi))
\times
C
~.
\]
Projection onto the first coordinate defines a map 
\begin{equation}
    \label{e4}
\w{\cMap}^\pitchfork(C,\Th(\xi))
\xra{~\pi~}
\cMap^\pitchfork(C,\Th(\xi))
~.
\end{equation}
\begin{observation}
\label{t2}
For each $f\in \cMap^\pitchfork(C,\Th(\xi))$, there is an identity 
\[
\pi^{-1}(f) 
~=~ 
f^{-1}(B)
~ 
\subset
~
C
~.
\]
\end{observation}

The following result is the technical underpinning of this article.  
It makes specific use of the topologies on $\cMap^\pitchfork(C,\Th(\xi))$ and $\Sub^k(C)$.
Its proof comprises~\S\ref{sec.technical}.
\begin{lemma}
\label{t11}
    Taking the preimage of the zero-section $B \subset \Th(\xi)$ defines a map
    \begin{equation}
    \label{e2}
    (-)^{-1}(B)
    \colon
    \cMap^\pitchfork\bigl(C,\Th(\xi) \bigr)
    \longrightarrow
    \Sub^k(C)
    ~,
    \end{equation}
    \[
    \Bigl(C \xra{f} \Th(\xi) \Bigr)
    \longmapsto
    \Bigl(
    f^{-1}(B) \subset C
    \Bigr)
    ~.
    \]
\end{lemma}

\begin{observation}
    \label{t43}
    Through Observation~\ref{t2} there is a canonical identification of $\w{\cMap}^\pitchfork(C,\Th(\xi))$ as a base-change among topological spaces:
    \[
\begin{tikzcd}
	{\w{\cMap}^\pitchfork(C,\Th(\xi))} && {\w{\Sub}^k(C)} \\
	{\cMap^\pitchfork(C,\Th(\xi))} && {\Sub^k(C)}
	\arrow[from=1-1, to=1-3]
	\arrow["\pi"', from=1-1, to=2-1]
	\arrow[from=1-3, to=2-3]
	\arrow["{(-)^{-1}(B)}", from=2-1, to=2-3]
\end{tikzcd}
~.
    \]
    Indeed, both $\w{\cMap}^\pitchfork(C,\Th(\xi))$ and the base-change are subspaces of $\cMap^\pitchfork(C,\Th(\xi)) \times C$, which can readily be check as identical.
    
\end{observation}

\begin{lemma}
\label{t3}
The map $\w{\cMap}^\pitchfork(C,\Th(\xi)) \xra{\pi} \cMap^\pitchfork(C,\Th(\xi))$ is a fiber bundle of compact codimension-$k$ submanifolds of $C$.
In particular, $\pi$ is locally trivial.  
\end{lemma}

\begin{proof}
Observation~\ref{t19}(2) states that the right vertical map is a fiber bundle of compact codimension-$k$ submanifolds of $C$.
The result then follows from Observation~\ref{t43}, because the base-change of such is again a fiber bundle of compact codimension-$k$ submanifolds of $C$.
    
\end{proof}

\begin{remark}
    After Lemma~\ref{t3}, we may interpret the map $\pi$ as a $\cMap^\pitchfork(C,\Th(\xi))$-family of compact codimension-$k$ submanifolds of $C$.  
\end{remark}

\begin{lemma}
\label{t9}
    Taking the preimage of the zero-section $B \subset \Th(\xi)$ defines a map between spaces
    \begin{equation}
    \label{e25}
    (-)^{-1}(B)
    \colon
    \cMap^\pitchfork\bigl(C,\Th(\xi) \bigr)
    \longrightarrow
    \Sub^\xi(C)
    ~,
    \end{equation}
    \[
    \Bigl(C \xra{f} \Th(\xi) \Bigr)
    \longmapsto
    \Bigl(
    C \supset f^{-1}(B) \xra{f_|} B \xra{\xi} \BO(k)
    \Bigr)
    ~.
    \]
\end{lemma}

\begin{proof}
The composite map
\[
\cMap^\pitchfork(C,\Th(\xi))
\xra{~(-)^{-1}(B)~}
\Sub^k(C)
\xra{~\nu^\pi~}
\underset{[W]} \coprod \Map(W,\BO(k))_{/\Diff(W)}
\]
classifies the fiber bundle $\w{\cMap}^\pitchfork(C,\Th(\xi)) \xra{\pi} \cMap^\pitchfork(C,\Th(\xi))$ of Lemma~\ref{t3} equipped with the map
\[
\w{\cMap}^\pitchfork(C,\Th(\xi))
\xra{~\nu^\pi~}
\BO(k)
\]
classifying the rank-$k$ vector bundle $\nu^\pi$ over $\w{\cMap}^\pitchfork(C,\Th(\xi))$ that is the fiberwise (with respect to Lemma~\ref{t3}) $\cMap^\pitchfork(C,\Th(\xi))$) normal bundle of $\w{\cMap}^\pitchfork(C,\Th(\xi)) \subset \cMap^\pitchfork(C,\Th(\xi)) \times C$.  

Consider the evaluation map
\begin{equation}
    \label{e42}
\ev
\colon
\w{\cMap}^\pitchfork(C,\Th(\xi))
\longrightarrow
B
~,\qquad
(f,x)
\longmapsto
f(x)
~.
\end{equation}
After Lemma~\ref{t3}, 
this map is classified by a map between spaces
\begin{equation}
\label{e29}
\cMap^\pitchfork(C,\Th(\xi))
\longrightarrow
\underset{[W]} \coprod \Map(W,B)_{/\Diff(W)}
~.
\end{equation}
Observe the canonical identification between vector bundles $\nu_{B \subset E} \cong \xi$ from the normal bundle of the zero-section $B \subset E$.
Using that normal bundles pullback along transverse maps, we have a canonical isomorphism between vector bundles over $\w{\cMap}^\pitchfork(C,\Th(\xi))$:
\begin{equation}
\label{e30}
\nu^\pi
~\cong~
\ev^\ast \nu_{B \subset E}
~\cong~
\ev^\ast \xi
~.
\end{equation}
This canonical isomorphism is a canonical 2-cell witnessing a commutative diagram among spaces
\[\begin{tikzcd}
	{\cMap^\pitchfork(C,\Th(\xi)) }& {\underset{[W]} \coprod \Map(W,B)_{/\Diff(W)}} \\
	{\Sub^k(C)} & {\underset{[W]} \coprod \Map(W,\BO(k))_{/\Diff(W)}}
	\arrow["(\ref{e29})", from=1-1, to=1-2]
	\arrow["(-)^{-1}(B)", swap, from=1-1, to=2-1]
	\arrow[from=1-2, to=2-2]
	\arrow["\nu^{\sf fib}(C)", from=2-1, to=2-2]
\end{tikzcd}
~.
\]
By Definition~\ref{d5} of $\Sub^\xi(C)$ as a pullback, this commutative diagram among spaces defines a map between spaces,
\[
\cMap^\pitchfork(C,\Th(\xi))
\longrightarrow
\Sub^\xi(C)
~.
\]
\end{proof}

\subsection{Proof of Lemma~\ref{t11}}\label{sec.technical}

Choose $f \in  \cMap^\pitchfork(C,\Th(\xi))$.
That is, $f$ is a compactly-supported map $C \ra \Th(\xi)$ that is smooth over $E \subset \Th(\xi)$ and transverse to the zero-section $B \subset E$.
Transversality implies the preimage $f^{-1}(B) \subset C$ is a submanifold: because the codimension of $B \subset E \underset{\rm opn} \subset \Th(\xi)$ is $k$, the codimension of $f^{-1}(B) \subset C$ is also $k$.
Because the zero-section $B \subset \Th(\xi)$ can be separated from the base-point of $\Th(\xi)$ by disjoint open subsets, then $f^{-1}(B) \subset {\sf Supp}(f)$.
Consequently, compactness of the support of $f$ implies compactness of $f^{-1}(B)$.
Therefore, we have $f^{-1}(B) \in \Sub^k(C)$.
Consequently, the assignment $f\mapsto f^{-1}(B)$ is a well-defined function between sets~(\ref{e2}).

In the remainder of this proof, we show~(\ref{e2}) is continuous at each point in $\cMap^\pitchfork\bigl(C,\Th(\xi) \bigr)$. That is, for each point \[
\Bigl(C \xra{f} \Th(\xi) \Bigr)\in \cMap^\pitchfork\bigl(C,\Th(\xi) \bigr)\]
there exists an open neighborhood $V$ of $f$ for which the composite map
\[
V
\xra{~\rm inclusion~}
\cMap^\pitchfork\bigl(C,\Th(\xi) \bigr)
\xra{(-)^{-1}(B)}
\Sub^k(C)
\]
is continuous. We do this by choosing $V$ sufficiently small that we can construct a function $\Phi$ fitting into a commutative diagram among sets
\begin{equation}
    \label{e31}
\xymatrix{
V
\ar@{-->}[rr]^-{\Phi}
\ar[d]_-{\rm inclusion}
&&
\Emb\bigl(f^{-1}(B) , C \bigr)
\ar[d]^-{\sf Image}
\\
\cMap^\pitchfork\bigl(C,\Th(\xi) \bigr)
\ar[rr]^-{(-)^{-1}(B)}
&&
\Sub^k(C)
}
\end{equation}
and then establishing continuity of $\Phi$.
We implement this through the following {\bf Steps}:
\begin{enumerate}
    \item\label{step.one} Using that the topology on $\cMap^\pitchfork\bigl(C,\Th(\xi) \bigr)$ is finer than the compact-open topology, we construct an open neighborhood 
    \[
    f 
    ~\in~ 
    V' 
    ~\subset~ 
    \cMap^\pitchfork\bigl(C,\Th(\xi) \bigr)
    \]
    with the following property.
    \begin{itemize}
        \item[] There is a closed neighborhood $B \subset \ov{U}\subset E$ such that the map given by restriction along $f^{-1}(\ov{U})$ factors:
        \begin{equation}
            \label{e5}
        \xymatrix{
        V'
        \ar[rr]^-{\rm inclusion}
        \ar@{-->}[d]
        &&
        \cMap^\pitchfork\bigl(C,\Th(\xi) \bigr)
        \ar[d]^-{\rm restriction}
        \\
        \Map^{\sf sm}(f^{-1}(\ov{U}),E)
        \ar[rr]^-{\rm inclusion}
        &&
        \Map^{\sf sm}\bigl( f^{-1}(\ov{U}),\Th(\xi) \bigr)
        .
        }
        \end{equation}
    \end{itemize}
    \item\label{step.two} Using the compact-open $\sC^\infty$ topology of $\Map^{\sf sm}(f^{-1}(\ov{U}),E)$, we identify an open neighborhood of $f_{|f^{-1}(\ov{U})} \in \Map^{\sf sm}(f^{-1}(\ov{U}),E)$ with a convex open neighborhood of the origin
    \[
    \Gamma^\tau 
    ~\subset~ 
    \Gamma\Bigl((f_{|f^{-1}(\ov{U})})^\ast \sT E \to f^{-1}(\ov{U}) \Bigr)
    \]
    in the topological vector space of smooth sections of the pullback of the tangent bundle of $E$ along $f_{|f^{-1}(\ov{U})}$.
    We define the smaller open neighborhood 
    \[
    f 
    ~\in~
    V 
    ~\subset~ 
    V'
    \]
    as the preimage of $\Gamma^\tau$ along the dashed map of the previous point.
   
    \item\label{step.three} Using the convexity of $\Gamma^\tau$, we construct a homeomorphism over $V$,
    \[
    \varphi
    \colon
    V \times f^{-1}(B)
    ~\cong~
    \pi^{-1}(V)
    =
    \w{\cMap}^\pitchfork(C,\Th(\xi))_{|V}
    ~,
    \]
    which is smooth in an appropriate sense.

\item\label{step.four} Using smoothness of $\varphi$, the adjoint of the composite map 
\[
V \times f^{-1}(B) \xra{\varphi} \pi^{-1}(V) \hookrightarrow  \cMap^\pitchfork\bigl(C,\Th(\xi) \bigr)\times C \xra{\pr} C
\]
defines a map
\[
\Phi
\colon
V
\longrightarrow
\Emb\bigl(f^{-1}(B) , C \bigr)
~.
\]
We finish by observing that this map $\Phi$ fits into the commutative diagram~(\ref{e31}).

\end{enumerate}

{\bf Step~(\ref{step.one})}: Once and for all, choose a complete Riemannian metric on $B$, and choose an inner-product on the vector bundle $\xi = (E \to B)$.  
Together, these structures determine a complete Riemannian metric on $E$.  
Through an application of the inverse function theorem, choose a smooth map $E \xra{\epsilon} \RR_{>0}$ such that, for each $p\in E$, the open ball $\sB_{\epsilon(p)}(p) \subset E$ is geodesically convex.  
For $B \xra{\delta} \RR_{>0}$ a smooth map, denote the $\delta$-neighborhood of $B \subset E$ as 
\[
U_{<\delta} 
~:=~ 
\{ e \in E \mid \exists b\in B \text{ for which } d(b,e) < \delta(b)\} \subset E
~.
\]
Denote by $U_{\leq \delta}$ similarly, replacing the relation $<$ by $\leq$.

For any $\delta<\epsilon$ sufficiently small, $U_{<\delta}$ is a tubular neighborhood of $B \subset E$.
Using that $f$ is transverse to $B$, there exists a sufficiently small such $\delta$ such that, for each $0<\delta'\leq \delta$, the preimage $f^{-1}(U_{<\delta'})$ is a tubular neighborhood of $f^{-1}(B) \subset C$.
Choose such a sufficiently small $\delta$.

Using that $f$ is compactly-supported, the preimage $f^{-1}(B) \subset C$ is compact.  
It follows that any closed tubular neighborhood of it is compact.
In particular, for each $0< \delta' < \delta$, the preimage $f^{-1}(U_{\leq \delta'}) \subset C$ is compact.  

Consider the following subset of $\cMap^\pitchfork\bigl(C,\Th(\xi)\bigr)$:
\[
V'
~:=~
\left\{
g \mid 
C^+ \smallsetminus f^{-1}(U_{< 2\delta/5} )
\subset g^{-1}\Bigl( \Th(\xi) \smallsetminus U_{\leq \delta/5} \Bigr)
\text{ and }
f^{-1}(U_{\leq 3\delta/5})
\subset 
g^{-1} (U_{<4\delta/5} )
\right\}
~.
\]
Note that $f\in V'$.
Recall that $f^{-1}(U_{\leq 2\delta/5})$ and $f^{-1}(U_{\leq 4\delta/5})$ are compact. 
Because the topology on $\cMap^\pitchfork\bigl(C,\Th(\xi)\bigr)$ is finer than the compact-open topology, this subset $V' \subset \cMap^\pitchfork\bigl(C,\Th(\xi)\bigr)$ is open.

We record two important observations about an element $g\in V'$.
First, the assumed containment $C^+ \smallsetminus f^{-1}(U_{< 2\delta/5} )
\subset g^{-1}\Bigl( \Th(\xi) \smallsetminus U_{\leq \delta/5} \Bigr)$ implies the containment $g^{-1}(U_{<\delta/5}) \subset f^{-1}(U_{<2\delta/5})$.  
Therefore, $g(x) \in B$ implies $x\in f^{-1}(U_{<2\delta/5})$.
Consequently, the containment $(g_{|f^{-1}(U_{\leq 2\delta/5})})^{-1}(B)
\subset
g^{-1}(B)$ 
is entire:
\begin{equation}
    \label{e22}
(g_{|f^{-1}(U_{\leq 2\delta/5})})^{-1}(B)
~=~
g^{-1}(B)
~.
\end{equation}
Second, the assumed containment 
$
f^{-1}(U_{\leq 3\delta/5})
\subset 
g^{-1} (U_{<4\delta/5} )
$
implies restriction along the subspace $f^{-1}(U_{\leq 2\delta/5}) \subset C$ defines a map
\begin{equation}
\label{e21}
V'
\xra{~(-)_{|f^{-1}(U_{\leq 2\delta/5})}}
\Map^{\sm}\bigl( f^{-1}(U_{\leq 2\delta/5}) , E \bigr)
~.
\end{equation}
as in~(\ref{e5}). Setting $\ov{U}:= U_{\leq 2\delta/5}$, observe that this completes Step~(\ref{step.one}).

{\bf Step~(\ref{step.two})}: Henceforth, we denote
\[
N
~:=~
f^{-1}(U_{\leq \delta/5})
~,
\]
which is a compact manifold with boundary, and a closed tubular neighborhood of $f^{-1}(B)$.

Consider the exponential map 
\[
\sT E
\xra{~{\sf exp}~}
E
~.
\]
Consider the open subset of the tangent bundle of $E$:
\[
\sT^\epsilon E
~:=~
\left\{ (e,v) \mid \lVert v \rVert < \epsilon(e) \right\}
~\subset~
\sT E
~.
\]
Denote the composite map
\[
{\sf exp}^\epsilon
\colon
\sT^\epsilon E
\hookrightarrow
\sT E
\xra{~{\sf exp}~}
E
~.
\]
Observe the commutative diagram among topological spaces
\[
\xymatrix{
(f_{|N})^\ast \sT^\epsilon E
\ar[rr]^-{\pr}
&&
\sT^\epsilon E
\ar[d]^-{{\sf exp}^\epsilon}
\\
N
\ar[u]^-{\rm zero}
\ar[rr]^-{f_{|N}}
&&
E
.
}
\]
Now, the choice of $\epsilon$ is such that ${\sf exp}^\epsilon$ is a surjective submersion.  
In particular, ${\sf exp}^\epsilon$ is transverse to $B \subset E$, and the preimage $({\sf exp}^\epsilon)^{-1}(B) \subset \sT^\epsilon E$ is a smooth submanifold.  
By definition of $f$, the bottom horizontal map $f_{|N}$ is transverse to $B$.
Therefore, the composite map $N \xra{\pr \circ {\rm zero}} \sT^\epsilon E$ is transverse to $({\sf exp}^\epsilon)^{-1}(B) \subset \sT^\epsilon E$.
In other words, the map $\pr$ is transverse to $({\sf exp}^\epsilon)^{-1}(B)$ along the zero-section.
Consequently, the map $\pr$ is transverse to $({\sf exp}^\epsilon)^{-1}(B)$ along a tubular neighborhood of the zero-section.
Let us make the previous sentence more specific.
\begin{itemize}
    \item[] For each smooth map $N \xra{\rho} \RR_{>0}$ satisfying $\rho(x)<\epsilon( f(x))$ for all $x\in N$, consider the open subspace
\[
\Bigl( (f_{|N})^\ast \sT^\epsilon E \Bigr)^\rho
~:=~
\left\{
(x,v) \mid \lVert v \rVert < \rho(x)
\right\}
~\subset~
(f_{|N})^\ast \sT^\epsilon E
~.
\]
There exists such a $\rho$ such that the composite map
\[
\pr^\rho
\colon
\Bigl( (f_{|N})^\ast \sT^\epsilon E \Bigr)^\rho
~\hookrightarrow~
(f_{|N})^\ast \sT^\epsilon E
\xra{~\pr~}
\sT^\epsilon E
\]
is transverse to $({\sf exp}^\epsilon)^{-1}(B) \subset \sT^\epsilon B$.
\end{itemize}
Choose such a $\rho$.
In particular, the preimage 
\[
( {\sf exp}^\epsilon \circ \pr^\rho)^{-1}(B)
~=~
(\pr^\rho)^{-1}\Bigl( \Bigl( {\sf exp}^\epsilon \Bigr)^{-1}(B) \Bigr) \subset \Bigl( (f_{|N})^\ast \sT^\epsilon E \Bigr)^\rho
\]
is a smooth submanifold.

Value-wise vector addition and value-wise scaling defines a topological vector space structure on the space of smooth sections,
\[
\Gamma^{\sm}\left ( \left(f_{|N}\right)^{\ast} \sT E \ra N \right)
~,
\]
endowed with the compact-open $\sC^\infty$ topology.
This topological vector space is locally convex.
Post-composition with the open embedding $\Bigl( (f_{|N})^\ast \sT^\epsilon E \Bigr)^\rho
\hookrightarrow
(f_{|N})^\ast \sT^\epsilon E$ defines an open embedding
\[
\Gamma^{\sm}\left( \left( \left(f_{|N}\right)^{\ast} \sT^\epsilon E\right)^\rho  \ra N \right)
~\hookrightarrow~
\Gamma^{\sm}\left ( \left(f_{|N}\right)^{\ast} \sT E \ra N \right)
~.
\]
This open embedding contains the zero-section.
Furthermore, post-composition with $\pr^\rho$ and ${\sf exp}^\epsilon$ define a map
\[
\Gamma^{\sm}\left( \left( \left(f_{|N}\right)^{\ast} \sT^\epsilon E\right)^\rho  \ra N \right)
\xra{~{\sf exp}^\epsilon \circ \pr^\rho \circ -~}
\Map^{\sm}\left( N , E \right)
~.
\]
The choice of $\epsilon$ ensures this map is an open embedding.

Consider the subset of $\Gamma^{\sm}\left( \left( \left(f_{|N}\right)^{\ast} \sT^\epsilon E\right)^\rho  \ra N \right)$,
\begin{equation}
\label{e20}
\Gamma^{\pitchfork}\left( \left( \left(f_{|N}\right)^{\ast} \sT^\epsilon E\right)^\rho  \ra N \right)
~:=~
\Bigl\{
\sigma \mid \sigma \pitchfork ( {\sf exp}^\epsilon \circ \pr^\rho)^{-1}(B) 
\Bigr\}
~,
\end{equation}
consisting of those smooth sections that are transverse to the submanifold $( {\sf exp}^\epsilon \circ \pr^\rho)^{-1}(B) \subset \Bigl( (f_{|N})^\ast \sT^\epsilon E \Bigr)^\rho$.
Using that $N$ is compact, 
Theorem~2.1 of \cite{hirsch} applies to conclude that this subset~(\ref{e20}) is open. 
In particular, because this subset~(\ref{e20}) contains the zero-section, and using that these spaces of smooth sections are locally convex, there is a convex open neighborhood of the zero-section,
\[
\Gamma^\tau 
~\subset~ 
\Gamma^{\sm}\left( \left( \left(f_{|N}\right)^{\ast} \sT^\epsilon E\right)^\rho  \ra N \right)
~,
\]
such that each $\sigma \in \Gamma^\tau$ is transverse to 
$( {\sf exp}^\epsilon \circ \pr^\rho)^{-1}(B) \subset \Bigl( (f_{|N})^\ast \sT^\epsilon E \Bigr)^\rho$.
Consider the pullback among topological spaces:
\begin{equation}
    \label{e6}
\xymatrix{
V
\ar[rr]
\ar[d]
&&
V'
\ar[d]^-{(\ref{e21})}
\\
\Gamma^\tau
\ar[rr]^-{{\sf exp}^\epsilon \circ \pr^\rho \circ -}
&&
\Map^{\sm}\bigl( N , E \bigr)
~.
}
\end{equation}
Continuity of the map~(\ref{e21}) implies $V \subset V'$ is open.
Openness of $V' \subset \cMap^\pitchfork\bigl(C,\Th(\xi) \bigr)$ implies $V \subset \cMap^\pitchfork\bigl(C,\Th(\xi) \bigr)$ is open.  
Because $\Gamma^\tau$ contains the zero-section, there is membership $f \in V$. This completes Step~(\ref{step.two}).

{\bf Step~(\ref{step.three})}: Consider the pullback
\[\begin{tikzcd}
	{\w{\cMap}^\pitchfork(C,\Th(\xi))_{|V}} & {\w{\cMap}^\pitchfork(C,\Th(\xi))} \\
	V & {\cMap^\pitchfork\bigl(C,\Th(\xi) \bigr)}
	\arrow[from=1-1, to=1-2]
	\arrow[from=1-1, to=2-1]
	\arrow["\lrcorner"{anchor=center, pos=0.125}, draw=none, from=1-1, to=2-2]
	\arrow["\pi", from=1-2, to=2-2]
	\arrow[hook, from=2-1, to=2-2]
\end{tikzcd}
~.
\]
Consider the subspace
\[
\w{\Gamma}^\tau
~:=~
\left\{
(\sigma,x) \mid {\sf exp}^\epsilon \circ \pr^\rho \circ \sigma(x) \in B
\right\}
~\subset~
\Gamma^\tau
\times
N
~.
\]
The identity~(\ref{e22}) identifies $\w{\cMap}^\pitchfork(C,\Th(\xi))_{|V}$ as a base-change:
\begin{equation}
\label{e23}
\begin{tikzcd}
	{\w{\cMap}^\pitchfork(C,\Th(\xi))_{|V}} & {\w{\Gamma}^\tau} \\
	V & {\Gamma^\tau}
	\arrow[from=1-1, to=1-2]
	\arrow[from=1-1, to=2-1]
	\arrow["\lrcorner"{anchor=center, pos=0.125}, draw=none, from=1-1, to=2-2]
	\arrow[from=1-2, to=2-2]
	\arrow[from=2-1, to=2-2]
\end{tikzcd}
~.
\end{equation}
Using that $\Gamma^\tau$ is convex, straight-line homotopy to the zero-section defines a map
\[
I \times \Gamma^\tau
\xra{~\rm scale~}
\Gamma^\tau
~,\qquad
(t,\sigma)
\longmapsto 
t \sigma
~,
\]
where $I := [0,1]\subset \RR$ is the closed interval.
Denote the base-change
\[\begin{tikzcd}
	{\w{\Gamma}^\tau_{I}} & {\w{\Gamma}^\tau} \\
	{I \times \Gamma^\tau} & {\Gamma^\tau}
	\arrow[from=1-1, to=1-2]
	\arrow[from=1-1, to=2-1]
	\arrow["\lrcorner"{anchor=center, pos=0.125}, draw=none, from=1-1, to=2-2]
	\arrow[from=1-2, to=2-2]
	\arrow["{{\rm scale}}", from=2-1, to=2-2]
\end{tikzcd}
~.
\]
Observe the pullbacks
\begin{equation}
\label{e27}
\begin{tikzcd}
	{\w{\Gamma}^\tau} & {\w{\Gamma}^\tau_{I}} & {\Gamma^\tau \times f^{-1}(B)} \\
	{\Gamma^\tau} & {I \times \Gamma^\tau} & {\Gamma^\tau}
	\arrow[from=1-1, to=1-2]
	\arrow[from=1-1, to=2-1]
	\arrow["\lrcorner"{anchor=center, pos=0.125}, draw=none, from=1-1, to=2-2]
	\arrow[from=1-2, to=2-2]
	\arrow[from=1-3, to=1-2]
	\arrow["\lrcorner"{anchor=center, pos=0.125, rotate=-90}, draw=none, from=1-3, to=2-2]
	\arrow[from=1-3, to=2-3]
	\arrow["{\{1\}}"', from=2-1, to=2-2]
	\arrow["{\{0\}}", from=2-3, to=2-2]
\end{tikzcd}
~.
\end{equation}

Once and for all, choose a complete Riemannian metric on $C$.
With the standard Riemannian metrics on $I$, this determines a complete Riemannian metric on $I \times C$.  
Consider the vector bundle $\Gamma^\tau \times \sT( I \times C)$ over $\Gamma^\tau \times I \times C$.
Using that for each $\sigma \in \Gamma^\tau$ the subspace $(\w{\Gamma}^\tau_I)_{|\{\sigma\}} \subset \{\sigma\} \times I \times C$ is a compact submanifold with corners, 
consider the subspace of its base-change over $\w{\Gamma}^\tau_{I} \subset \Gamma^\tau \times I \times C$,
\[
\sT^{\sf fib} \w{\Gamma}^\tau_I
~:=~
\left\{
(\sigma , (x,v) ) \mid v \in \sT_x (\w{\Gamma}^\tau_I)_{|\{\sigma\}} 
\right\}
~\subset~
\Bigl( 
\Gamma^\tau \times \sT( I \times C)
\Bigr)_{|\w{\Gamma}^\tau_I}
~.
\]
Orthogonal projection with respect to the given Riemannian metric on $I \times C$ defines a surjection over $\w{\Gamma}^\tau_I$:
\begin{equation}
\label{e26}
{\sf pr}_{\w{\Gamma}^\tau_I}
\colon
\Bigl( 
\Gamma^\tau \times \sT( I \times C)
\Bigr)_{|\w{\Gamma}^\tau_I}
\longrightarrow
\sT^{\sf fib} \w{\Gamma}^\tau_I
~.
\end{equation}
The standard coordinate vector field on the closed interval $I$ defines a continuous section $T$ of the vector bundle $\Gamma^\tau \times \sT( I \times C)$ over $\Gamma^\tau \times I \times C$, which, for each $\sigma \in \Gamma^\tau$, restricts as a smooth section of the vector bundle $\{\sigma\} \times \sT(I \times C) \to \{\sigma\} \times I \times C$.
Flow by this section implements translation in the $I$-coordinate of $\Gamma^\tau \times I \times C$.
The restriction of $T$ over $\w{\Gamma}^\tau_I$ projects via~(\ref{e26}) to a section ${\sf pr}_{\w{\Gamma}^\tau_I}(T)$ of the map $\sT^{\sf fib} \w{\Gamma}^\tau_I \to \w{\Gamma}^\tau_I$.  
Now, for each $\sigma\in \Gamma^\tau$, consider the map:
\[
(\w{\Gamma}^\tau_I)_{|\{\sigma\}}
\longrightarrow
I \times \{\sigma\} = I
~.
\]
By construction of $\Gamma^\tau$, this map is a proper smooth map with no critical points.  
Consequently, the section ${\sf pr}_{\w{\Gamma}^\tau_I}(T)$ is nowhere vanishing and gradient-like.
For each $\sigma \in \Gamma^\tau$, flow by this section defines a diffeomorphism between manifolds with corners:
\[
\psi_\sigma
\colon 
\w{\Gamma}^\tau_{|\{\sigma\}}
\underset{(\ref{e27})}{=}
(\w{\Gamma}^\tau_I)_{|\{\sigma\}\times \{1\}}
\xra{~\cong~}
(\w{\Gamma}^\tau_I)_{|\{\sigma\}\times \{0\}}
\underset{(\ref{e27})}{=}
f^{-1}(B)
~.
\]
These diffeomorphisms assemble as a function between sets over $\Gamma^\tau$:
\begin{equation}
\label{e28}
\psi\colon
\Gamma^\tau \times f^{-1}(B)
\xra{~\cong~}
\w{\Gamma}^\tau
~,\qquad
(\sigma,x)
\longmapsto 
\psi_\sigma(x)
~,
\end{equation}
which restricts to each fiber as a diffeomorphism between compact manifolds with corners.  
Using smooth dependence on initial conditions of solutions to ODE (e.g., Chapter~XIV of~\cite{lang}), this map~(\ref{e28}) is a homeomorphism.  
Through the base-change identity~(\ref{e23}), this homeomorphism base-changes as a homeomorphism
\begin{equation}
\label{e32}
\varphi
\colon
V \times f^{-1}(B)
\xra{~\cong~}
\w{\cMap}^\pitchfork(C,\Th(\xi))_{|V}
~,\qquad
(g,x)
\longmapsto 
\phi_g(x)
~,
\end{equation}
over $V$, which restricts to each fiber as a diffeomorphism between compact submanifolds of $C$. This completes Step~(\ref{step.three}).

{\bf Step~(\ref{step.four})}: This homeomorphism $\varphi$ is adjoint to a function 
\begin{equation}
\label{e24}
\Phi\colon
V
\longrightarrow
\Emb\bigl(f^{-1}(B) , C \bigr)
~,\qquad
g
\longmapsto
\Bigl(
x
\longmapsto
\phi_g(x)
\Bigr)
~.
\end{equation}
Again by smooth dependence on initial conditions of solutions to ODE, this function is continuous.
Observe that the diagram of functions between topological spaces
\[
\xymatrix{
V
\ar[rr]^-{\Phi}
\ar[d]_-{\rm inclusion}
&&
\Emb \bigl(f^{-1}(B) , C \bigr)
\ar[d]^-{\sf Image}
\\
\cMap^\pitchfork\bigl(C,\Th(\xi) \bigr)
\ar[rr]^-{(-)^{-1}(B)}
&&
\Sub^k(C)
}
\]
commutes.
Continuity of $\Phi$ and ${\sf Image}$ implies the restriction of the function $(-)^{-1}(B)$ along the open subset $V \subset \cMap^\pitchfork\bigl(C,\Th(\xi) \bigr)$ is continuous.
In particular, the function $(-)^{-1}(B)$ is continuous at $f \in \cMap^\pitchfork\bigl(C,\Th(\xi) \bigr)$, which completes Step~(\ref{step.four}) and the proof as desired.

\subsection{Parametrized collapse maps}
Fix a $(p+n+k)$-manifold $C$ with corners.
In this section, we use parametrized collapse maps to construct a homotopy inverse to the map 
\[
\cMap^\pitchfork(C,\Th(\xi)) \xra{~(-)^{-1}(B)~} \Sub^\xi(C)
\]
of Lemma~\ref{t11}, at least on compact families.

\begin{remark}
We outline a construction of an inverse,
\[
\Sub^\xi(C)
\longrightarrow
\cMap^\pitchfork(C,\Th(\xi))
~,
\]
supposing the space $\Sub^\xi(C)$ admits a natural structure of an infinite-dimensional smooth manifold.
Under this supposition, the subspace 
$
\w{\Sub}^\xi(C)
\subset
\Sub^\xi(C) \times C
$ 
projects to $\Sub^\xi(C)$ as a smooth fiber bundle of compact submanifolds of $C$.
Also, it is equipped with a map
\[
\w{\Sub}^\xi(C)
\xra{~g~}
B
~.
\]
For $\w{\nu}^{\sf fib}(C)$ the fiberwise (over $\Sub^\xi(C)$) normal bundle, there is a canonical isomorphism between vector bundles over $\w{\Sub}^\xi(C)$:
\[
\alpha
\colon 
\w{\nu}^{\sf fib}(C)
~\cong~
g^\ast \xi
~.
\]
Invoking a tubular neighborhood theorem for finite-codimensional submanifolds of infinite-dimensional manifolds,
choose a closed tubular neighborhood
\[
\w{\Sub}^\xi(C)
~\subset~
\ov{N}
~\subset~
\Sub^\xi(C) \times C
\]
and denote its interior as $N$;
choose, a homeomorphism
\[
N
~\overset{\varphi}\cong~
E(\nu^{\sf fib}(C))
\]
with the total space of the fiberwise normal bundle. 
Such a choice of $\varphi$ would determine a composite map
\[
\Sub^\xi(C) \times C
\xra{~\rm collapse~}
\frac{\ov{N}}{\partial \ov{N}}
\xra{~\varphi^+~}
\Th(\nu^{\sf fib}(C))
\xra{~\alpha~}
\Th(h^\ast \xi)
\longrightarrow
\Th(\xi)
\]
which is adjoint to the desired homotopy inverse.
It should then be straightforward to check this map is a section to the map $(-)^{-1}(B)$.
To see that $(-)^{-1}(B)$ is a homotopy section would lastly follow using a family of maps $(0,\infty]\times \Th(\xi) \xra{\Phi} \Th(\xi)$ such that the support of $\Phi_t$ is a closed $t$-neighborhood of the zero-section. Our proof of Theorem~\ref{t10} follows this outline over compact finite-dimensional smooth families (in fact, over disks), where smooth manifold theory is well-established.

\end{remark}

\begin{theorem}
\label{t10}
The map~(\ref{e25}) is an equivalence between spaces:
\[
    (-)^{-1}(B)
    \colon
    \cMap^\pitchfork(C,\Th(\xi))
    \xra{~\simeq~}
    \Sub^\xi(C)
    ~.
\]

\end{theorem}

\begin{proof}
We prove that the relative homotopy groups of the map vanish.
Fix $q \geq 0$, and consider a solid commutative diagram among spaces in which the vertical maps are as named:
\begin{equation}
    \label{e3}
\xymatrix{
\SS^{q-1}
\ar[d]_-{\rm inclusion}
\ar[rr]^-{\w{f}_0}
&&
\cMap^{\pitchfork}(C,\Th(\xi))
\ar[d]^-{(-)^{-1}(B)}
\\
\DD^q
\ar[rr]_-f
\ar@{-->}[urr]^-{\w{f}}
&&
\Sub^\xi(C)
.
}
\end{equation}
We wish to construct a filler $\w{f}$ among spaces.

Note that we may choose a convenient representative of the homotopy classes of the pair of maps $(f,\w{f}_0)$.  
\begin{itemize}
    \item The map $\w{f}_0$ is adjoint to a compactly-supported map $\SS^{q-1} \times C \xra{\w{f}_0} \Th(\xi)$ that restricts over $E$ to each fiber over $\SS^{q-1}$ as a smooth map transverse to the zero-section $B \subset E$.
    By smooth approximation, and genericity of transversality, we may homotope $\w{f}_0$ to assume it is adjoint to a map that is smooth over $E$ and transverse to $B \subset E$.

    \item 
    Denote the composite map 
\[
W_f
\colon 
\DD^q
\xra{~f~}
\Sub^\xi(C)
\longrightarrow
\Sub^k(C)
~.
\]
Consider the topological subspace
\[
\w{\DD}^q
~:=~
W_f^\ast \w{\Sub}^k(C)
~=~
\left\{
(x,c)
\mid 
c\in W_f(x)
\right\}
~\subset~
\DD^q \times C
~,
\]
which is the base-change of $\w{\Sub}^k(C) \to \Sub^k(C)$ along $W_f$.
By smooth approximation, we may homotope $f$ relative to $\SS^{q-1} \subset \DD^q$ to ensure $\w{\DD}^q \subset \DD^q \times C$ is a smooth submanifold.
Observation~\ref{t19}(2) ensures the projection 
\[
\w{\DD}^q
\xra{~\pi~}
\DD^q
\]
is a smooth fiber bundle. 
Recall that the map $f$ supplies a map $g\colon \w{\DD}^q \ra B$ as well as an isomomorphism between vector bundles
\[
\alpha
\colon 
\nu^\pi
~\cong~
g^\ast \xi
\]
from the vertical normal bundle of $\w{\DD}^q \xra{\pi} \DD^q$ to the pullback of $\xi$.
By smooth approximation, 
we may homotope $g$ and $\alpha$ to assume $g$ is a smooth map and $\alpha$ is an isomorphism between smooth vector bundles.

\end{itemize}

Once and for all, choose a smooth inner-product on the smooth vector bundle $\xi$.  
For each $\delta>0$, denote the subspaces $U_{<\delta} \subset U_{\leq \delta} \subset E$ consisting of those elements whose norm is, respectively, less than $\delta$ and no more than $\delta$.
Once and for all, choose a smooth map
\[
\sigma\colon
(0,\infty]
\times
\RR
\longrightarrow
\RR
\]
with the following properties.
\begin{itemize}

    \item There is an open neighborhood $(0,\infty] \times \{0\} \subset O \subset (0,\infty]\times \RR$ such that $\sigma(t,x)=(t,x)$ for each $(t,x)\in O$. 
    
    \item For each $t\in (0,\infty]$, the smooth map $\sigma_t := \sigma_{|\{t\} \times \RR} \colon \RR \to \RR$ is an open embedding with image $\BB_t(0) \subset \RR$, the open ball of radius $t$ about the origin.

    \item The map $\sigma_\infty = \id_{\RR}$ is the identity.

\end{itemize}
Consider the map
\[
(0,\infty] \times \Th(\xi)
\xra{~\Phi~}
\Th(\xi)
~,\qquad
(t,e)
\longmapsto
\begin{cases}
    e
    &
    ,~
    \lVert e \rVert = 0
    \\
    \sigma_t^{-1}( \lVert e \rVert) \cdot \frac{e}{\lVert e \rVert}
    &
    ,~
    0<\lVert e \rVert < t
    \\
    +
    &
    ,~
    t \leq \lVert e \rVert
\end{cases}
~.
\]
Observe that $\Phi$ is continuous, and is smooth over $E$.

Using that the compactly-supported map $\SS^{q-1} \times C \xra{\w{f}_0} \Th(\xi)$ is smooth over $E$ and transverse to $B \subset E$, transversality ensures the existence of a $\delta>0$ such that each $0<\delta'< \delta$ has the following properties.
\begin{itemize}
    \item $\w{f}_0^{-1}(U_{<\delta'}) \subset \SS^{q-1} \times C$ is a tubular neighborhood of $\w{f}_0^{-1}(B) \subset \SS^{q-1} \times C$. 

    \item The subspace $\w{f}_0^{-1}(U_{<\delta'}) \subset \SS^{q-1} \times C$ is a smooth fiber subbundle with respect to the projection to $\SS^{q-1}$.
    
\end{itemize}
Denote the map
\[
\w{f}_{0|}
\colon
\w{f}_0^{-1}(B)
\longrightarrow
B
~,\qquad
(u,c)
\longmapsto
\w{f}_0(u,c)
~.
\]
Denote the open subset
\[
N_0
~:=~
\w{f}_0^{-1}(U_{<\delta/2})
~\subset~
\SS^{q-1} \times C
~.
\]
Observe the pointed continuous factorization, which is necessarily unique, through the collapse map:
\begin{equation}\label{e41}
\begin{tikzcd}
	{\SS^{q-1} \times C} && {\Th(\xi)} \\
	{N_0^{+}} && {\Th(\xi)}
	\arrow["{\w{f}_0}", from=1-1, to=1-3]
	\arrow["{{\rm collapse}}"', from=1-1, to=2-1]
	\arrow["{\Phi_{\delta/2}}", from=1-3, to=2-3]
	\arrow["{\ov{f}_0}", dashed, from=2-1, to=2-3]
\end{tikzcd}
~.
\end{equation}

Observe the canonical identification between vector bundles $\nu_{B \subset E} \cong \xi$ from the normal bundle of the zero-section $B \subset E$.
Consider the normal bundle $\nu_{\w{f}_0}$ of the submanifold $\w{f}_0^{-1}(B) \subset \SS^{q-1} \times C$.
Using that normal bundles pullback along transverse maps (Fact~\ref{trans.fact}(1)), we have a canonical isomorphism between smooth vector bundles over $\w{f}_0^{-1}(B)$:
\[
\nu_{\w{f}_0}
~\cong~
\w{f}_{0|}^\ast \nu_{B \subset E}
~\cong~
\w{f}_{0|}^\ast \xi
~.
\]
This isomorphism determines a pointed map
\begin{equation}
\label{e35}
\Th(\nu_{\w{f}_0})
\cong
\Th(\w{f}_{0|}^\ast \xi)
\longrightarrow
\Th(\xi)
~.
\end{equation}
Using the tubular neighborhood theorem, choose a diffeomorphism between the total space of the normal bundle $\nu_{\w{f}_0}$ and the tubular neighborhood $N_0$ of $\w{f}_0^{-1}(B) \subset \SS^{q-1} \times C$,
\[
E(\nu_{\w{f}_0})
\overset{\varphi_0}{~\cong~}
N_0
~,
\]
such that the diagram among pointed topological spaces
\[\begin{tikzcd}
	{N_0^+} && {\Th(\xi)} \\
	& {\Th(\nu_{\w{f}_0})}
	\arrow["{\ov{f}_0}", from=1-1, to=1-3]
	\arrow["{\varphi_0^+}"', from=1-1, to=2-2]
	\arrow["(\ref{e35})"', from=2-2, to=1-3]
\end{tikzcd}
\]
commutes.

Now, using the tubular neighborhood theorem, choose the following data:
\begin{itemize}
    \item A tubular neighborhood $\w{\DD}^q \subset N \subset \DD^q \times M$ such that the projection $N \to \DD^q$ is a smooth fiber bundle and there is an identity $N_0 = N \cap \Bigl( \SS^{q-1} \times C\Bigr)$ between subsets of $\SS^{q-1} \times C$.
    
    \item A diffeomorphism $E(\nu^\pi) \overset{\varphi}\cong N$ under $\w{\DD}^q$ extending $\varphi_0$.
\end{itemize}  
Using that $\w{\DD}^q$ is compact, we can shrink $N$ as needed to assume its closure $\ov{N} \subset \DD^q \times C$ is compact.
These data determine the composite map
\[
\w{f}
\colon
\DD^q \times C
\xra{\rm collapse}
N^+
\overset{\varphi^+}\cong
\Th(\nu^\pi)
\overset{\Th(\alpha)}{\cong}
\Th( g^\ast \xi)
\to
\Th(\xi)
~.
\]
Let us denote the adjoint of this map again as $\DD^q \xra{\w{f}} \cMap(C,\Th(\xi))$.
Observe that $\w{f}^{-1}(E) = N \subset \DD^q \times M$.
Because the closure of $N$ is compact, this map $\w{f}$ is compactly-supported.  
For $E(\chi)$ to denote the total space of a smooth vector bundle $\chi$, the definition of $\w{f}$ is such that the restriction of $\w{f}$ to $N$ is a composition:
\[
\w{f}_{|N} \colon N \xra{\varphi} E(\nu^\pi) \xra{E(\alpha)} E(g^\ast \xi) \xra{E(g)} E(\xi) = E
~.
\]
Because each factor in this composition is smooth, then $\w{f}_{|N}$ is smooth.  
Because $E(g)$ is a fiberwise isomorphism between smooth vector bundles, it is transverse to the zero-section and the preimage of the zero-section is precisely the zero-section.  
Because $\varphi$ and $\alpha$ are diffeomorphisms under $\w{\DD}^q$, we conclude that $\w{f}_{|N}$ is transverse to the zero-section $B \subset E$, and the preimage $\w{f}^{-1}(B) = \w{\DD}^q$.
We conclude that $\w{f}$ defines a map as in~(\ref{e3}) making the lower right triangle commute.

By construction of $\DD^q \times C \xra{ \w{f} } \Th(\xi)$, its restriction is identical with the map in~(\ref{e41}):
\[
\w{f}_{|\SS^{q-1} \times C}
~=~
\ov{f}_0 \circ {\rm collapse}
\underset{(\ref{e41})}{~=~}
\Phi_{\delta/2} \circ \w{f}_0
~.
\]
Using that $\sigma_\infty = \id$,
the map $[\delta/2,\infty] \times \Th(\xi) \xra{\Phi} \Th(\xi)$ then supplies a path between $\SS^{q-1} \xra{\w{f}_0} \cMap^\pitchfork(C,\Th(\xi))$
and $\SS^{q-1} \xra{\w{f}_{|\SS^{q-1}}} \cMap^\pitchfork(C,\Th(\xi))$:
\[
\w{f}_{|\SS^{q-1}}
~\simeq~
\w{f}_0
~.
\]
In other words, the upper left triangle in~(\ref{e3}) homotopy commutes, thereby completing the proof.

\end{proof}

\subsection{Simplicially compatible parametrized transverse intersections}

Here, we prove Theorem~\ref{t26}.

\begin{lemma}
    \label{t27}
    The maps 
    \[
    \cMap^\pitchfork_p(M,\Th(\xi)) := \cMap^\pitchfork(\Delta^p\times M , \Th(\xi)) \xra{(-)^{-1}(B)} \Sub^k(\Delta^p\times M) = \Bord^k_p(M)
    \]
    of Lemma~\ref{t11} respect the simplicial structure morphisms:
    \[
    \cMap^\pitchfork_\bullet(M,\Th(\xi)) 
    \xra{~(-)^{-1}(B)~}
    \Bord^k_\bullet(M)
    ~.
    \]
    
\end{lemma}

\begin{proof}
    Let $[p] \xra{\sigma} [q]$ be a morphism in $\bDelta$.
    We must show the diagram among topological spaces
    \[
\begin{tikzcd}
	{\cMap^\pitchfork_q(M,\Th(\xi))} && {\Bord^k_q(M)} \\
	{\cMap^\pitchfork_p(M,\Th(\xi))} && {\Bord^k_p(M)}
	\arrow["{(-)^{-1}(B)}", from=1-1, to=1-3]
	\arrow["{\sigma^\ast}"', from=1-1, to=2-1]
	\arrow["{\sigma^\ast}", from=1-3, to=2-3]
	\arrow["{(-)^{-1}(B)}", from=2-1, to=2-3]
\end{tikzcd}
    \]
    commutes.
    Let $(\Delta^q \times M \xra{f} \Th(\xi)) \in \cMap^\pitchfork_q(M,\Th(\xi))$.
    Commutativity of the diagram follows upon observing the sequence of identities among subsets of $\Delta^p \times M$:
    \begin{eqnarray*}
    (\sigma^\ast(f))^{-1}(B)
    &
    :=
    &
    (f \circ (\sigma_\ast \times M))^{-1}(B)
    \\
    &
    =
    &
    (\sigma_\ast \times M)^{-1}\Bigl( f^{-1}(B) \Bigr)
    \\
    &
    =:
    &
    \sigma^\ast\Bigl( f^{-1}(B) \Bigr)
    ~.    
    \end{eqnarray*}
    The first identity follows from the definition of the left vertical map $\sigma^\ast$.
    The second identity is a consequence of the general fact that the preimage of a composition is an iterated preimage.  
    The last identity is the definition of the right vertical map $\sigma^\ast$, applied to $(f^{-1}(B) \subset \Delta^q \times M) \in \Bord^k_q(M)$.
    
\end{proof}

We now construct the morphism
\begin{equation}
\label{e47}
\cMap^\pitchfork_\bullet(M,\Th(\xi))
\longrightarrow
\Bord^\xi_\bullet(M)
~,
\end{equation}
which Theorem~\ref{t26} states is an equivalence.
In what follows, we reference the material of~\S\ref{sec.normal}.
Consider the category-object in $\Top$ over $\bDelta$,
\[
\Un^\pitchfork(M)
\longrightarrow
\bDelta
~,
\]
that is the Grothendieck construction of the simplicial topological space $\cMap^\pitchfork_\bullet(M,\Th(\xi))$.
Its target morphism is a covering space, and therefore it is a topological category.
On underlying $\infty$-categories, it is the unstraightening of the simplicial space underlying $\cMap^\pitchfork_\bullet(M,\Th(\xi))$.
Through the functoriality of the Grothendieck construction, the morphism $\cMap^\pitchfork_\bullet(M,\Th(\xi)) \to \Bord^k_\bullet(M)$ between simplicial topological spaces of Lemma~\ref{t27} supplies a morphism over $\bDelta$ between topological categories in $\Top$:
\[
\Un^\pitchfork(M)
\longrightarrow
\Un^k(M)
~.
\]
Denote the pullback among category-objects in $\Top$:
\begin{equation}\label{e53}
\begin{tikzcd}
	{\w{\Un}^\pitchfork(M)} & {\w{\Un}^k(M)} \\
	{\Un^\pitchfork(M)} & {\Un^k(M)}
	\arrow[from=1-1, to=1-2]
	\arrow[from=1-1, to=2-1]
	\arrow["\lrcorner"{anchor=center, pos=0.125}, draw=none, from=1-1, to=2-2]
	\arrow[from=1-2, to=2-2]
	\arrow[from=2-1, to=2-2]
\end{tikzcd}
~.
\end{equation}
Because the right vertical morphism is a left fibration between topological categories that is a fiberwise fiber bundle, the same is true for the left vertical morphism.
In particular, $\w{\Un}^\pitchfork(M)$ is a topological category.
Observation~\ref{t43} implies the fiber of the morphism $\w{\Un}^\pitchfork(M) \xra{\pi} \Un^\pitchfork(M)$ over $[p]\in \bDelta$ is the fiber bundle $\w{\cMap}^\pitchfork(\Delta^p\times M , \Th(\xi)) \to \cMap^\pitchfork(\Delta^p \times M , \Th(\xi))$ of Lemma~\ref{t3}.
Consequently, the evaluation map~(\ref{e42}) (from the proof of Lemma~\ref{t9}) defines a map to $B$ from the topological space of objects of $\w{\Un}^\pitchfork(M)$:
\begin{equation}\label{e44}
\ev
\colon
\Obj(\w{\Un}^\pitchfork(M))
\xra{~(\ref{e42})~}
B
~.
\end{equation}
For each morphism $[p] \xra{\sigma}[q]$ in $\bDelta$, the diagram among topological spaces
{\small
\[
\begin{tikzcd}
	{\w{\cMap}^\pitchfork(\Delta^p\times M , \Th(\xi))} & {\w{\cMap}^\pitchfork(\Delta^p\times M , \Th(\xi)) \underset{\cMap^\pitchfork_p(M)} \times \cMap^\pitchfork_q(M)} & {\w{\cMap}^\pitchfork(\Delta^q\times M , \Th(\xi))} \\
	& B
	\arrow["\ev"', from=1-1, to=2-2]
	\arrow["\pr"', from=1-2, to=1-1]
	\arrow["{\sigma_\ast \times M}", from=1-2, to=1-3]
	\arrow["\ev", from=1-3, to=2-2]
\end{tikzcd}
\]
}
commutes.  
Indeed, an element in the upper middle topological space is of the form 
\[
\Bigl(~
\Delta^q \times M \xra{f} \Th(\xi)
~,~ 
x\in (\sigma_\ast \times M)^{-1}\bigl( f^{-1}(B) \bigr)
~\Bigr)
~;
\]
the value of both of the composite maps in this diagram on this element is $f \bigl( (\sigma_\ast \times M)(x) \bigr) \in B$.
Consequently,~(\ref{e44}) extends as a morphism between topological categories
\begin{equation}\label{e48}
\ev
\colon 
\w{\Un}^\pitchfork(M)
\longrightarrow
B
~,
\end{equation}
in which the codomain is the topological category whose topological space of objects is $B$ and all of whose morphisms are identity morphisms.

As observed above, the morphism $\w{\Un}^\pitchfork(M) \to \Un^\pitchfork(M)$ between topological categories is a left fibration.
It is therefore classified by a functor between $\infty$-categories:
\[
\Un^\pitchfork(M)
\longrightarrow
\Spaces
~,
\]
which, by construction, factors through $\Un^k(M)$.
The morphism~(\ref{e48}) defines a lift of this functor:
\[
\begin{tikzcd}
	&& {\Spaces_{/B}} \\
	{\Un^\pitchfork(M)} && \Spaces
	\arrow["{{\rm forget}}", from=1-3, to=2-3]
	\arrow["\ev", from=2-1, to=1-3]
	\arrow[from=2-1, to=2-3]
\end{tikzcd}
~.
\]
We next argue that the diagram among $\infty$-categories
\begin{equation}\label{e49}
\begin{tikzcd}
	{\Un^\pitchfork(M)} & {\Spaces_{/B}} \\
	{\Un^k(M)} & {\Spaces_{/\BO(k)}}
	\arrow[from=1-1, to=1-2]
	\arrow[from=1-1, to=2-1]
	\arrow[from=1-2, to=2-2]
	\arrow[from=2-1, to=2-2]
\end{tikzcd}
\end{equation}
commutes.

We now make use of~\S\ref{sec.vbdl}.
Regard the rank-$k$ vector bundle $\xi = (E \to B)$ as a constant category-object in ${\sf VBun}^k$ (i.e., all morphisms in it are identity morphisms).  
Base-change along the morphism~(\ref{e44}) defines a rank-$k$ vector bundle $\ev^\ast \xi$ over the topological category $\w{\Un}^\pitchfork(M)$. 
Recall the rank-$k$ vector bundle $\nu^{\sf fib}$ over the topological category $\w{\Un}^k(M)$.  
Base-change of $\nu^{\sf fib}$ along the defining morphism $\w{\Un}^\pitchfork(M) \xra{(\ref{e53})} \w{\Un}^k(M)$ defines a rank-$k$ vector bundle $\nu^\pi$ over the topological category $\w{\Un}^\pitchfork(M)$.

The canonical identifications~(\ref{e30}), applied to $C=\Delta^p\times M$ for each $p\geq 0$, is an identification between vector bundles $\nu^\pi \cong \ev^\ast \xi$ over the topological space $\Obj(\w{\Un}^\pitchfork(M))$ of objects of $\w{\Un}^\pitchfork(M)$.  
Furthermore, by way of Fact~\ref{trans.fact}, these canonical identifications~(\ref{e30}) are such that the diagram among vector bundles over $\Mor(\w{\Un}^\pitchfork(M))$
\[\begin{tikzcd}
	{{\sf s}^\ast \nu^\pi_0} & {{\sf t}^\ast \nu^\pi_0} \\
	{{\sf s}^\ast \ev^\ast \xi} & {{\sf t}^\ast \ev^\ast \xi}
	\arrow["{\nu^\pi_1}", from=1-1, to=1-2]
	\arrow["{(\ref{e30})}"', from=1-1, to=2-1]
	\arrow["{(\ref{e30})}", from=1-2, to=2-2]
	\arrow["\cong"{description}, draw=none, from=2-1, to=2-2]
\end{tikzcd}\]
commutes.
We conclude an isomorphism
\[
\nu^\pi
~\cong~
\ev^\ast \xi
\]
between vector bundles over $\w{\Un}^\pitchfork(M)$.
Using that $\w{\Un}^\pitchfork(M)$ is a topological category, 
this isomorphism represents an identification
\[
\nu^\pi
\simeq
\ev^\ast \xi
~\colon~
\w{\Un}^\pitchfork(M)
\longrightarrow
\BO(k)
\]
between $\infty$-categories.
We conclude that the diagram among $\infty$-categories~(\ref{e49}) indeed commutes.

By definition of $\Un^\xi(M)$, the commutative diagram~(\ref{e49}) is a functor between $\infty$-categories
\[
\Un^\pitchfork(M)
\longrightarrow
\Un^\xi(M)
~,
\]
both of which are right fibrations over $\bDelta$.
The straightening of this functor is the sought morphism~(\ref{e47}) between simplicial spaces:
$
\cMap^\pitchfork_\bullet(M,\Th(\xi))
\xra{(\ref{e47})}
\Bord^\xi_\bullet(M)
$.

\begin{proof}[Proof of Theorem~\ref{t26}]
By construction, for each $p \geq 0$, the value of the morphism~(\ref{e47}) on spaces of $p$-simplices is the map between spaces
\[
\cMap^\pitchfork(\Delta^p\times M , \Th(\xi))
\xra{~(-)^{-1}(B)~}
\Sub^\xi(\Delta^p \times M)
\]
of Lemma~\ref{t9}.
Theorem~\ref{t10} states that this map between spaces of $p$-simplices is an equivalence.
Therefore, the morphism~(\ref{e47}) between simplicial spaces is an equivalence, as desired. 

Functoriality with respect to open embeddings $M\hookrightarrow M'$ follows from commutativity of the diagram of inverse images
\[
\xymatrix{
\cMap^\pitchfork(\Delta^p\times M , \Th(\xi))
\ar[rr]^-{(-)^{-1}(B)}\ar[d]&&\Sub^\xi(\Delta^p \times M)\ar[d]\\
\cMap^\pitchfork(\Delta^p\times M' , \Th(\xi))
\ar[rr]^-{(-)^{-1}(B)}&&\Sub^\xi(\Delta^p \times M')
}
\]
in which the right vertical map is the evident inclusion, and the left vertical map is extension by zero, i.e., it extends a compactly-supported map $\Delta^p\times M \ra \Th(\xi)$ to $\Delta^p\times M'$ so as to be constant at the basepoint outside of $\Delta^p \times M$.
\end{proof}

\section{The simplicial set $\dBord_\bullet^{\xi}(M)$}\label{sec.Kan.condition}
Throughout this section, we fix $n,k \geq 0$; an $(n+k)$-manifold $M$ with boundary; and a rank-$k$ vector bundle $\xi = (E \to B)$.
The main object in this section is a simplicial set $\dBord^\xi_\bullet(M)$. 
The main results are Corollary~\ref{cor.dBord.Bord}, which states that its geometric realization agrees with that of the simplicial space $\Bord^\xi_\bullet(M)$; and Corollary~\ref{cor.dBord.kan}, which states that the simplicial set $\dBord^\xi_\bullet(M)$ is a Kan complex.

\subsection{Constructing $\dBord^\xi_\bullet(M)$}
Throughout this subsection, we fix a manifold with corners $C$.

\begin{notation}
The set
\[
\dSub^k(C)
~:=~
\left\{
~W \subset C~
\right\}
\]
consists of compact codimension-$k$ submanifolds of $C$.
\end{notation}

\begin{definition}\label{a1}
The set
\[
\dSub^\xi(C)
~:=~
\left\{
\left(
~
W\subset C 
~,~ 
W \xra{g} B
~,~
\nu_{W \subset C} \overset{\alpha} \cong g^\ast \xi 
~
\right)
\right\}
\]
consists of triples in which: $(W \subset C)\in \Sub^k(C)$ is a compact codimension-$k$ submanifold; $g\in \Map(W,B)$ is a map; and $\alpha \in {\sf Iso}(\nu_{W \subset C} , g^\ast \xi)$ is an isomorphism between vector bundles over $W$.

\end{definition}

Let $D \xra{\sigma} C$ be a proper face-submersion from a manifold with corners in the sense of Definition~\ref{d6}.
Consider the map between sets
\begin{equation}\label{x31}
\dSub^\xi(C)
\xra{~\sigma^{-1}(-)~}
\dSub^\xi(D)
~,
\end{equation}
whose value on an element 
$\left(
~
W\subset C
~,~
W \xra{g} B
~,~
\nu_{W \subset C} \overset{\alpha} \cong g^\ast \xi
~
\right)$
in the domain is the element
\[
\left(
~\sigma^{-1}(W) \subset D
~,~
\sigma^{-1}(W) \xra{g \circ \sigma_{|\sigma^{-1}(W)}} B
~,~
\nu_{\sigma^{-1}(W) \subset D}
\overset{\sigma^{-1}(\alpha)}\cong
(g \circ \sigma_{|\sigma^{-1}(W)})^\ast \xi
\right)
\]
in the codomain, in which the third coordinate $\sigma^{-1}(\alpha)$ is the canonical composite isomorphism between vector bundles over $\sigma^{-1}(W)$:
\begin{eqnarray*}
\sigma^{-1}(\alpha) 
\colon
\nu_{\sigma^{-1}(W) \subset D}
&
\overset{\sD \sigma_{|\sigma^{-1}(W)}}\cong 
&
(\sigma_{|\sigma^{-1}(W)})^\ast \nu_{W \subset C}
\\
&
\overset{(\sigma_{|\sigma^{-1}(W)})^\ast \alpha}
\cong
&
(\sigma_{|\sigma^{-1}(W)})^\ast g^\ast \xi
\\
&
\cong
&
(g \circ \sigma_{|\sigma^{-1}(W)})^\ast \xi
~,
\end{eqnarray*}
in which the first isomorphism is Fact~\ref{trans.fact}(1).
Lemma~\ref{t57} ensures the map~(\ref{x31}) is indeed well-defined.

\begin{observation}\label{a7}
Let $E \xra{\tau} D \xra{\sigma} C$ be a composable pair of proper face-submersions.
The diagram among sets
\[\begin{tikzcd}
	{\dSub^\xi(C)} && {\dSub^\xi(E)} \\
	& {\dSub^\xi(D)}
	\arrow["{(\sigma \circ \tau)^{-1}(-)}", from=1-1, to=1-3]
	\arrow["{\sigma^{-1}(-)}"', from=1-1, to=2-2]
	\arrow["{\tau^{-1}(-)}"', from=2-2, to=1-3]
\end{tikzcd}\]
commutes.
Indeed, this is obviously so on the first two coordinates, and Fact~\ref{trans.fact}(2) implies it is so on the third coordinates.  

\end{observation}

Observation~\ref{a7} validates the following.
\begin{definition}
    We denote the simplicial sets
    \[
    \dBord^k_\bullet(M)
    ~:=~
    \dSub^k(\Delta^\bullet \times M)
    \qquad\text{ and }\qquad
    \dBord^\xi_\bullet(M)
    ~:=~
    \dSub^\xi(\Delta^\bullet \times M)
    ~.
    \]
\end{definition}

\begin{observation}\label{a9}
The set $\dSub^k(C)$ is that underlying the topological space $\Sub^k(C)$.
Therefore, the simplicial set $\dBord^k_\bullet(M)$ is that underlying the simplicial topological space $\Bord^k_\bullet(M)$.
\end{observation}

Observe the existence of a natural forgetful map
    \begin{equation}\label{a2}
    \dSub^\xi(C)
    \longrightarrow
    \dSub^k(C)
    ~,\qquad
    (W,f,\alpha)
    \longmapsto 
    W
    ~.
    \end{equation}

\begin{observation}\label{a5}
Let $D \xra{\sigma} C$ be a proper face-submersion.
The diagram among sets
\[
\begin{tikzcd}
	{\dSub^\xi(C)} & {\dSub^\xi(D)} \\
	{\dSub^k(C)} & {\dSub^k(D)}
	\arrow["\sigma^{-1}(-)", from=1-1, to=1-2]
	\arrow[from=1-1, to=2-1]
	\arrow[from=1-2, to=2-2]
	\arrow["\sigma^{-1}(-)", from=2-1, to=2-2]
\end{tikzcd}
\]
commutes.
\end{observation}

Observation~\ref{a5} yields a forgetful morphism between simplicial sets
\begin{equation}\label{a21}
\dBord^\xi_\bullet(M)
\longrightarrow
\dBord^k_\bullet(M)
~.
\end{equation}

\subsection{$\Bord^\xi_\bullet(M)$ as a simplicial topological space}
Here, we construct a topology on the simplicial set $\dBord^\xi_\bullet(M)$ whose underlying simplicial space is $\Bord^\xi_\bullet(M)$ of Definition~\ref{def.Bord.s.space}.

\begin{prop}\label{t88}
    There is a topological space $\tSub^\xi(C)$ with the following properties.
    \begin{enumerate}
        \item Its underlying set is $\dSub^\xi(C)$.
        
        \item The resulting function $\tSub^\xi(C) \xra{(\ref{a2})} \Sub^k(C)$ has the following properties.
        \begin{enumerate}
            \item It is continuous.

            \item It is a Serre fibration.
            
            \item The underlying space of $\tSub^\xi(C)$ is identical with the space $\Sub^\xi(C)$ of Definition~\ref{d5}.
            
        \end{enumerate}
    \end{enumerate}
\end{prop}

\begin{proof}
Let $W$ be a compact $(p+n)$-manifold with corners.
Consider the diagram among topological spaces
\[
\begin{tikzcd}
	W \\
	{\Emb(W,C) \times W} & {\Emb(W,C) \times \Map(W,B) \times W} & {\Map(W,B) \times W} & B \\
	C
	\arrow["\pr"', from=2-1, to=1-1]
	\arrow["\ev", from=2-1, to=3-1]
	\arrow["{\pi_1}"', from=2-2, to=2-1]
	\arrow["{\pi_2}", from=2-2, to=2-3]
	\arrow["\ev", from=2-3, to=2-4]
\end{tikzcd}
~,
\]
in which the maps $\pi_1$ and $\pi_2$ are the evident projections.
Taking derivatives defines an injection between vector bundles over $\Emb(W,C) \times W$, which defines a short exact sequence of vector bundles:
\[
0
\longrightarrow
\pr^\ast \tau_W
\xra{~\sD~}
\ev^\ast \tau_C
\xra{~\rm quotient~}
\nu
\longrightarrow
0
~.
\]
Consider the two vector bundles $\pi_1^\ast \nu$ and $(\ev \circ \pi_2)^\ast \xi$ over $\Emb(W,C) \times \Map(W,B) \times W$.
Consider the resulting principal ${\sf GL}_n(\RR)$-bundle of isomorphisms between these vector bundles:
\begin{equation}\label{x27}
X
~:=~
{\sf Iso}
~:=~
{\sf Iso}\left(
~
\pi_1^\ast \nu
~,~(\ev \circ \pi_2)^\ast \xi
~
\right)
\longrightarrow
\Emb(W,C) \times \Map(W,B) \times W
~=:~
Y
~.
\end{equation}
Note that each of the maps 
\[
X
\longrightarrow
Y
\longrightarrow
\Emb(W,C) \times \Map(W,B)
~=:~
Z
\]
is a fiber bundle.  
Consider the fiber bundles
\[
\Map^{\sf rel}_{Z}\left(
Y
,
Y
\right)
\longrightarrow
Z
\qquad\text{ and }\qquad
\Map^{\sf rel}_{Z}\left(
Y
,
X
\right)
\longrightarrow
Z
~,
\]
of relative sections of 
$
Y
\to 
Z$ 
along $Y \to Z$
and
of relative sections of 
$
X
\to 
Z$ 
along $Y \to Z$, respectively.\footnote{For $(W \xra{e} C , W \xra{g} B) \in Z$, the fiber of the first fiber bundle over this point is the topological space of maps $\Map(W,W)$, while the fiber of the second fiber bundle over this point is the topological space of maps $\Map\left( W ,{\sf Iso}(\nu_e , g^\ast \xi) \right)$ to the ${\sf GL}_n(\RR)$-torsor of isomorphisms between vector bundles over $W$.}
Notice the canonical continuous forgetful map
\begin{equation}\label{x28}
\Map^{\sf rel}_{Z}\left(
Y
,
X
\right)
\longrightarrow
\Map^{\sf rel}_{Z}\left(
Y
,
Y
\right)
\end{equation}
between fiber bundles over $Z$.
Because $X \to Y$ is a fiber bundle, so is the map~(\ref{x28}).
The identity map on $Y$ is a continuous section of the codomain of~(\ref{x28}).
Define the pullback topological space
\[\begin{tikzcd}
	{\Emb^\xi(W,C)} & {\Map^{\sf rel}_Z(Y,X)} \\
	Z & {\Map^{\sf rel}_Z(Y,Y)}
	\arrow[from=1-1, to=1-2]
	\arrow[from=1-1, to=2-1]
	\arrow["\lrcorner"{anchor=center, pos=0.125}, draw=none, from=1-1, to=2-2]
	\arrow[from=1-2, to=2-2]
	\arrow["{\lag \id\rag}", from=2-1, to=2-2]
\end{tikzcd}
~.
\]
Because fiber bundles are stable under base-change, the maps
\begin{equation}\label{x29}
\Emb^\xi(W,C)
\longrightarrow
Z
=
\Emb(W,C) \times \Map(W,B)
\xra{~\rm projection~}
\Emb(W,C)
\end{equation}
are each fiber bundles.
The first of these fiber bundles is a principal $\Map(W,{\sf GL}_n(\RR))$-bundle whose fiber over $(W \xra{e} C , W \xra{g} B)$ is the $\Map(W,{\sf GL}_n(\RR))$-torsor ${\sf Iso}(\nu_e,g^\ast \xi)$ of isomorphisms of vector bundles over $W$.
Specifically, this first fiber bundle is a base-change among topological spaces:
\begin{equation}\label{x40}\begin{tikzcd}
	{\Emb^\xi(W,C)} && {\Map(W,\BO(k))} \\
	{\Emb(W,C) \times \Map(W,B)} && {\Map(W,\BO(k)) \times \Map(W,\BO(k)) = \Map(W,\BO(k) \times \BO(k))} \\
	\\
	\\
	& {}
	\arrow[from=1-1, to=1-3]
	\arrow["{(\ref{x29})}"', from=1-1, to=2-1]
	\arrow["{{\rm diagonal}}"', from=1-3, to=2-3]
	\arrow["{\nu \times (f\circ -)}", from=2-1, to=2-3]
\end{tikzcd}
~.
\end{equation}

Observe how the continuous diagonal action $\Diff(W) \lacts \Emb(W,C) \times \Map(W,B) \times W$ canonically determines a continuous action $\Diff(W) \lacts \Emb^\xi(W,C)$ with respect to which the fiber bundle~(\ref{x29}) is $\Diff(W)$-equivariant.
We have the resulting commutative diagram among topological spaces:
\begin{equation}\label{a23}
\begin{tikzcd}
	{\Emb^\xi(W,C)} && {\Emb(W,C)} \\
	{\Emb^\xi(W,C)_{/\Diff(W)}} && {\Emb(W,C)_{/\Diff(W)}} \\
	\\
	\\
	& {}
	\arrow["{(\ref{x29})}", from=1-1, to=1-3]
	\arrow["{{\rm quotient}}"', from=1-1, to=2-1]
	\arrow["{{\rm quotient}}", from=1-3, to=2-3]
	\arrow["{(\ref{x29})_{/\Diff(W)}}", from=2-1, to=2-3]
\end{tikzcd}
~.
\end{equation}
Necessarily, the action $\Diff(W) \lacts \Emb^\xi(W,C)$ is free, and by slices, because the action $\Diff(W) \lacts \Emb(W,C)$ is.  
Therefore the vertical maps in this diagram are principal $\Diff(W)$-bundles, and this diagram is a pullback.
Using that~(\ref{x29}) is a fiber bundle as well, Lemma~\ref{t62} implies the lower horizontal map is a Serre fibration.
Observe the canonical bijection between sets,
\[
\underset{[W]} \coprod 
\Emb^\xi(W,C)_{/\Diff(W)}
\xra{~\cong~}
\dSub^\xi(C)
~,
\]
in which the coproduct is indexed by isomorphism classes of $(n+p)$-manifolds with corners.
This bijection endows the set $\dSub^\xi(C)$ with a topology, thereby proving statement~(1) of the proposition.  
Denote this topological space $\tSub^\xi(C)$.
Because the lower horizontal map in~(\ref{a23}) is a Serre fibration, the canonical map
\[
\tSub^\xi(C)
\xra{~(\ref{a2})~}
\Sub^k(C)
\]
is a Serre fibration, thereby proving statements~(2)(a) and~(2)(b).
By Definition~\ref{d5} of the map of space $\Sub^\xi(C) \to \Sub^k(C)$ as a base-change, the pullback diagram~(\ref{x40}) for each $W$ then implies
Statement~(2)(c) by taking $\Diff(W)$-quotients for each $W$.
\end{proof}

\begin{convention}
    In what follows, we regard $\tSub^\xi(C)$ as a topological space with the topology of Proposition~\ref{t88}.
\end{convention}

Proposition~\ref{t88}(2)(b) has the following strengthening.
\begin{lemma}\label{t79}
The map $\tSub^\xi(C) \xra{(\ref{a2})} \Sub^k(C)$ is a fiber bundle.

\end{lemma}

\begin{proof}
Let $(W \subset C) \in \Sub^k(C)$.
Using the tubular neighborhood theorem, choose an open embedding $E(\nu_{W \subset C}) \xra{\psi} C$ under $W$.
Consider the resulting open embedding ${\sf Image} \circ \psi_\ast \colon \Gamma^{\sf sm}(\nu_{W\subset C}) \xra{{\sf Image}(\psi \circ-)} \Sub^k(C)$ as in Lemma~\ref{t52}.
Consider the base-change
\[
\begin{tikzcd}
	{\Sub^\xi(C)_{|\Gamma^{\sf sm}(\nu_{W\subset C})}} & {\Sub^\xi(C)} \\
	{\Gamma^{\sf sm}(\nu_{W \subset C})} & {\Sub^k(C)}
	\arrow[hook, from=1-1, to=1-2]
	\arrow[from=1-1, to=2-1]
	\arrow["\lrcorner"{anchor=center, pos=0.125}, draw=none, from=1-1, to=2-2]
	\arrow[from=1-2, to=2-2]
	\arrow["{\sf Image} \circ \psi_\ast", hook, from=2-1, to=2-2]
\end{tikzcd}
~.
\]
Consider the map
\[
\Sub^\xi(C)_{|\Gamma^{\sf sm}(\nu_{W\subset C})}
\xra{~F~}
\Gamma^{\sf sm}(\nu_{W\subset C}) \times \tSub^\xi(C)_{|W}
~,
\]
whose value on an element 
$
\bigl(
~
W \xra{s} E(\nu_{W\subset C})
~,~
\psi(s(W))
\xra{g}
B
~,~
\nu_{\psi(s(W)) \subset C}
\overset{\alpha} \cong
g^\ast \xi
~
\bigr)
$
in the domain is the element
$
\bigl(
~
W \xra{s} E(\nu_{W\subset C})
~,~
W
\xra{g\circ \psi \circ s}
B
~,~
\nu_{W \subset C}
\overset{\alpha \circ \sD (\psi \circ  {\sf Add}_s \circ \psi^{-1})}\cong
g^\ast \xi
~
\bigr)
$
in the codomain, where the last coordinate is the composite isomorphism between vector bundles over $W$:
\begin{eqnarray*}
\nu_{W \subset C}
&
\overset{\sD \psi^{-1}}\cong
&
\nu_{W \subset E(\nu_{W \subset C})}
\\
&
\overset{\sD{\sf Add}_s}\cong
&
\nu_{s(W) \subset E(\nu_{W \subset C})}
\\
&
\overset{\sD \psi}\cong
&
\nu_{\psi(s(W)) \subset C}
\\
&
\overset{\alpha} \cong
&
g^\ast \xi
~.
\end{eqnarray*}
Here, the second isomorphism is that induced by the morphism between affine bundles $\nu_{W \subset C} \xra{{\sf Add}_s} \nu_{W \subset C}$ over $W$ given by $(w,v)\mapsto (w,v+s(w))$, which carries the (image of the) zero-section $W \subset E(\nu_{W \subset C})$ to the image of the section $s$. 
The map $F$ is evidently a bijection, with inverse given by 
\[
(s,g,\alpha)
\longmapsto
\left(
~
s
~,~
\psi(s(W)) \xra{g \circ s^{-1} \circ \psi^{-1}} B
~,~ 
\alpha \circ \sD\psi \circ \sD {\sf Add}_s^{-1} \circ \sD \psi^{-1} 
~
\right)
~.
\]
Unpacking the topology of $\tSub^\xi(C)$ from the proof of Proposition~\ref{t88}, and thereafter of its subspaces $\tSub^\xi(C)_{|\Gamma^{\sf sm}(\nu_{W\subset C})}$ and $\Sub^\xi(C)_{|W}$, reveals that both $F$ and $F^{-1}$ are continuous.  
Therefore, $\Sub^\xi(C) \to \Sub^k(C)$ is locally trivializable.

\end{proof}

Lemma~\ref{t79} draws attention to the fibers of the forgetful map $\tSub^\xi(C) \xra{(\ref{a2})} \Sub^k(C)$, which we examine now.
Consider the fiber of the map~(\ref{a2}) over $(W\subset C) \in \Sub^k(C)$,
\[
\tSub^\xi(C)_{|W}
~:=~
\left\{
\left(~
W \xra{g} B
~,~
\nu_{W \subset C} \overset{\alpha} \cong g^\ast \xi
~
\right)
\right\}
~,
\] 
whose underlying set consists of pairs in which $g \in \Map(W,B)$ is a map and $\alpha \in {\sf Iso}(\nu_{W \subset C} , g^\ast \xi)$ is an isomorphism between vector bundles over $W$.
Consider the function between underlying sets:
\begin{equation}\label{a3}
\tSub^\xi(C)_{|W}
\longrightarrow
\Map(W,B)
~,\qquad
(g,\alpha)
\longmapsto
g
~.
\end{equation}

Inspecting the topology on $\tSub^\xi(C)$ reveals the following.
\begin{observation}\label{a6}
Let $(W \subset C) \in \Sub^k(C)$.
The function~(\ref{a3}) is continuous.
Furthermore, the evident action $\Aut(\nu_{W\subset C}) \lacts \tSub^\xi(C)_{|W}$ is continuous and free, and the map~(\ref{a3}) witnesses the quotient by $\Aut(\nu_{W\subset C})$.
Lastly, this action is by slices, and therefore the quotient map~(\ref{a3}) is a principal $\Aut(\nu_{W\subset C})$-bundle.
\end{observation}

For $D \xra{\sigma} C$ a proper face-submersion, 
the commutative diagram of Observation~\ref{a5} supplies, for each $(W \subset C) \in \Sub^k(C)$, a function between underlying sets
\begin{equation}\label{a20}
\tSub^\xi(C)_{|W}
\xra{~\sigma^{-1}(-)_|~}
\tSub^\xi(D)_{|\sigma^{-1}(W)}
~.
\end{equation}

\begin{lemma}\label{t82}
Let $D \xra{\sigma} C$ be a proper face-submersion.
The function~(\ref{a20}) is continuous. 
Furthermore, if $\sigma$ is a proper face-embedding, then~(\ref{a20}) is a Serre fibration.
\end{lemma}

\begin{proof}

Now, notice that the resulting map between fibers fits into a commutative diagram among sets:
\begin{equation}\label{x34}
\begin{tikzcd}
	{\Sub^\xi(C)_{|W}} && {\Sub^\xi(D)_{|\sigma^{-1}(W)}} \\
	{\Map(W,B)} && {\Map(\sigma^{-1}(W) ,B)}
	\arrow["{\sigma^{-1}(-)_{|}}", from=1-1, to=1-3]
	\arrow["{(\ref{a3})}"', from=1-1, to=2-1]
	\arrow["{(\ref{a3})}", from=1-3, to=2-3]
	\arrow["{- \circ \sigma_{|\sigma^{-1}(W)}}", from=2-1, to=2-3]
\end{tikzcd}
~.
\end{equation}
The lower horizontal map in this diagram is continuous because restriction maps are.  
By Observation~\ref{a6}, the vertical maps~(\ref{a3}) in this diagram are continuous.
Consequently, the map $\Sub^\xi(C)_{|W} \to \Map(\sigma^{-1}(W),B)$ in the diagram~(\ref{x34}) is continuous.
This map is adjoint to the diagonal map in the diagram among topological spaces:
\begin{equation}\label{a10}
\begin{tikzcd}
	{\sigma^{-1}(W)} & {\sigma^{-1}(W) \times \Sub^\xi(C)_{|W}} \\
	W & {W \times \Sub^\xi(C)_{|W}} & {W \times \Map(W,B)} & B
	\arrow["{\sigma_{|}}"', from=1-1, to=2-1]
	\arrow["\pr"', from=1-2, to=1-1]
	\arrow["{\sigma_{|}}"', from=1-2, to=2-2]
	\arrow["G", from=1-2, to=2-4]
	\arrow["\pr"', from=2-2, to=2-1]
	\arrow["{W \times (\ref{a3})}"', from=2-2, to=2-3]
	\arrow["\ev"', from=2-3, to=2-4]
\end{tikzcd}
~.
\end{equation}
Denote the composite horizontal map $W \times \Sub^\xi(C)_{|W} \xra{g} B$. 
By definition of $\Sub^\xi(C)_{|W}$, there is a canonical isomorphism $\pr^\ast \nu_{W\subset C} \overset{\alpha} \cong g^\ast \xi$ between vector bundles over $W \times \Sub^\xi(C)_{|W}$.
This isomorphism pulls back as an isomorphism between vector bundles over $\sigma^{-1}(W) \times \Sub^\xi(C)_{|W}$,
\begin{equation}\label{a11}
\pr^\ast \nu_{\sigma^{-1}(W) \subset D}
\underset{\rm Fact~\ref{trans.fact}(1)}\cong
\pr^\ast {\sigma_|}^\ast \nu_{W \subset C}
=
{\sigma_|}^\ast \pr^\ast \nu_{W \subset C}
\overset{{\sigma_|}^\ast \alpha}\cong
{\sigma_|}^\ast g^\ast \xi
=
G^\ast \xi
~,
\end{equation}
in which the equalities follow from commutativity of the diagram~(\ref{a10}).
Unpacking the definition of $\Sub^\xi(D)_{|\sigma^{-1}(W)}$ reveals how the upper horizontal map in~(\ref{x34}), as it lifts the map $\Sub^\xi(C)_{|W} \to \Map(\sigma^{-1}(W),B)$, is adjoint to the isomorphism~(\ref{a11}).
The fact that the isomorphism~(\ref{a11}) is, in particular, continuous implies the upper horizontal map in~(\ref{x34}) is continuous, as desired.

Generally, for $\chi = \left( E(\chi) \to B(\chi) \right)$ a vector bundle, consider the fiber subbundle $\Aut^{\sf rel}(\chi) \to B(\chi)$ of $\Map^{\sf rel}_{B(\chi)}(E(\chi),E(\chi)) \to B(\chi)$ consisting of the fiberwise linear isomorphisms.
This fiber bundle over $B(\chi)$ is a group-object among such.
The topological group $\Aut(\chi)$ of automorphisms of $\chi$ is the topological space of sections of $\Aut^{\sf rel}(\chi) \to B(\chi)$.
Observe the canonical isomorphisms between fiber bundles over $\sigma^{-1}(W)$,
\[
(\sigma_{|\sigma^{-1}(W)})^\ast \Aut^{\sf rel}(\nu_{W \subset C})
~\cong~
\Aut^{\sf rel}\left( (\sigma_{|\sigma^{-1}(W)})^\ast \nu_{W \subset C} \right)
~\cong~
\Aut^{\sf rel}( \nu_{\sigma^{-1}(W) \subset D} )
~,
\]
the first of which is direct from universal properties, and the second of which follows from Fact~\ref{trans.fact}.
Taking sections determines a continuous homomorphism:
{\small
\begin{equation}\label{x35}
\Aut(\nu_{W \subset C})
=
\Gamma\left(
\Aut^{\sf rel}(\nu_{W \subset C})
\right)
\xra{\sigma_{|\sigma^{-1}(W)}}
\Gamma\left(\sigma_{|\sigma^{-1}(W)})^\ast \Aut^{\sf rel}(\nu_{W \subset C}\right)
\cong
\Gamma\left( \Aut^{\sf rel}( \nu_{\sigma^{-1}(W) \subset D} ) \right)
=
\Aut(\nu_{\sigma^{-1}(W) \subset D})
~.
\end{equation}
}
Notice that the map $\tSub^\xi(C)_{|W} \xra{\sigma^{-1}(-)} \tSub^\xi(D)_{|\sigma^{-1}(W)}$ between fibers is equivariant with respect to the continuous homomorphism $\Aut(\nu_{W \subset C}) \xra{(\ref{x35})} \Aut(\nu_{\sigma^{-1}(W) \subset D})$.

Now assume $\sigma$ is a proper face-embedding.
The diagram~(\ref{x34}) determines a commutative diagram among topological spaces,
\[
\begin{tikzcd}
	{\Sub^\xi(C)_{|W}} & \square && {\Sub^\xi(D)_{|\sigma^{-1}(W)}} \\
	{\Map(W,B)} & {\Map(W,B)} && {\Map(\sigma^{-1}(W),B)}
	\arrow[from=1-1, to=1-2]
	\arrow[from=1-1, to=2-1]
	\arrow[from=1-2, to=1-4]
	\arrow[from=1-2, to=2-2]
	\arrow[from=1-4, to=2-4]
	\arrow["{=}", from=2-1, to=2-2]
	\arrow["{-\circ \sigma_{|\sigma^{-1}(W)}}", from=2-2, to=2-4]
\end{tikzcd}
~,
\]
in which $\square$ is the pullback of the right square, and the map to it is the canonical one as such.  
The existence of collar neighborhoods for manifolds with corners ensures the map $D \xra{\sigma} C$ is a cofibration.
Consequently, the map $\Map(W,B) \xra{- \circ \sigma_{|\sigma^{-1}(W)}} \Map(\sigma^{-1}(W),B)$ is a Serre fibration.
Because Serre fibrations are stable under base-change, it follows that the map $\square \to \tSub^\xi(D)_{|\sigma^{-1}(W)}$ is a Serre fibration.
Similarly, the map~(\ref{x35}) is a Serre fibration.  
Together with Observation~\ref{a6}, it follows that the map $\tSub^\xi(C)_{|W} \to \square$ is locally a Serre fibration.
Using paracompactness and Hausdorffness of mapping spaces and of $\Sub^k(D)$ (Lemma~\ref{t83}) and thereafter of $\Sub^k(D)_{|\sigma^{-1}(W)}$, 
Theorem~13 of Section 7 of Chapter 2 of~\cite{spanier} then implies the map $\tSub^\xi(C)_{|W} \to \square$ is a Serre fibration.
Using that the composition of Serre fibrations is a Serre fibration, we conclude that the map $\tSub^\xi(C)_{|W} \to \tSub^\xi(D)_{|\sigma^{-1}(W)}$ is a Serre fibration, as desired.

\end{proof}

\begin{lemma}\label{t77}
Let $D \xra{\sigma} C$ be a proper face-submersion.  
The map $\tSub^\xi(C) \xra{\sigma^{-1}(-)} \tSub^\xi(D)$ is continuous.
Furthermore, if $\sigma$ is a proper face-embedding, then this map is a Serre fibration.
\end{lemma}

\begin{proof}

We prove this map, which is~(\ref{x31}), is continuous at each element $(W,f,\alpha) \in \tSub^\xi(C)$.
This proof will reference the following diagram of sets, which is explained below:
{\Small
\[
\begin{tikzcd}
	{\Gamma^{\sf sm}(\nu_{W\subset C}) \times \tSub^\xi(C)_{|W}} & {\tSub^\xi(C)_{|\Gamma^{\sf sm}(\nu_{W \subset C})}} & {\tSub^\xi(D)_{|\Gamma^{\sf sm}( \nu_{\sigma^{-1}(W) \subset D} )}} & {\Gamma^{\sf sm}(\nu_{\sigma^{-1}(W)\subset D}) \times \tSub^\xi(D)_{|\sigma^{-1}(W)}} \\
	& {\tSub^\xi(C)} & {\tSub^\xi(D)}
	\arrow["\cong"{description}, draw=none, from=1-1, to=1-2]
	\arrow["{\tSub^\xi_{|(\ref{e58})}}", from=1-2, to=1-3]
	\arrow[from=1-2, to=2-2]
	\arrow["\cong"{description}, draw=none, from=1-3, to=1-4]
	\arrow[from=1-3, to=2-3]
	\arrow["{(\ref{x31})}", from=2-2, to=2-3]
\end{tikzcd}
~.
\]
}
Choose compatible tubular neighborhoods of $W \subset C$ and of $\sigma^{-1}(W) \subset D$ as in the proof of Lemma~\ref{t58}.
The square is then the base-change (in the category $\Ar(\Sets)$) of the square~(\ref{x33}) along $\left(\tSub^\xi(C) \xra{\sigma^{-1}(-)} \tSub^\xi(D)\right) \to \left(\Sub^k(C) \xra{\sigma^{-1}(-)} \Sub^k(D)\right)$.
The two bijections, which are homeomorphisms, are those from the proof of Lemma~\ref{t79}.
Now, inspecting the definition of each map in the top horizontal composite 
{\small
\[
\Gamma^{\sf sm}(\nu_{W\subset C})
\times
\tSub^\xi(C)_{|W}
~\cong~
\tSub^\xi(C)_{|\Gamma^{\sf sm}(\nu_{W\subset C})}
\xra{(\ref{x31})}
\tSub^\xi(D)_{|\Gamma^{\sf sm}(\nu_{\sigma^{-1}(W)\subset D})}
~\cong~
\Gamma^{\sf sm}(\nu_{\sigma^{-1}(W)\subset D})
\times
\tSub^\xi(D)_{|\sigma^{-1}(W)}
\]
}
reveals that it
is a product of the map~(\ref{e58}) between spaces of smooth sections and the map~(\ref{x34}) between fibers.  
Lemma~\ref{t52}(1) implies the vertical maps are open embeddings.
Continuity of~(\ref{x31}) at $(W,f,\alpha) \in \tSub^\xi(C)$ then follows from the continuity of the map~(\ref{e58}) and Lemma~\ref{t82}.

Now suppose $D \xra{\sigma} C$ is a proper face-embedding.
We may factor the map $\tSub^\xi(C) \xra{\sigma^{-1}(-)} \tSub^\xi(D)$ as the top horizontal map in the commutative diagram among topological spaces:
\[
\begin{tikzcd}
	{\tSub^\xi(C)} & \square & {\tSub^\xi(D)} \\
	{\Sub^k(C)} & {\Sub^k(C)} & {\Sub^k(D)}
	\arrow[from=1-1, to=1-2]
	\arrow[from=1-1, to=2-1]
	\arrow[from=1-2, to=1-3]
	\arrow[from=1-2, to=2-2]
	\arrow["\lrcorner"{anchor=center, pos=0.125}, draw=none, from=1-2, to=2-3]
	\arrow[from=1-3, to=2-3]
	\arrow["{=}", from=2-1, to=2-2]
	\arrow["{\sigma^{-1}(-)}", from=2-2, to=2-3]
\end{tikzcd}
\]
where $\square$ makes the right square pullback the map to it is the canonical one as such.
Because Serre fibrations are stable under base-change, Lemma~\ref{t61} implies the map $\square \to \tSub^\xi(D)$ is a Serre fibration.
Because Serre fibrations are stable under composition, the map $\tSub^\xi(C) \xra{\sigma^{-1}(-)}\tSub^\xi(D)$ is a Serre fibration provided the map $\tSub^\xi(C) \to \square$ is a Serre fibration.
Because fiber bundles are stable under base-change, Lemma~\ref{t79} implies the middle vertical map is a fiber bundle.  
As argued above, the base-change of $\tSub^\xi(C) \to \square$ along the open embedding $\Gamma^{\sf sm}(\nu_{W \subset C}) \hookrightarrow \Sub^k(C)$ is homeomorphic with a product of maps
\[
\Gamma^{\sf sm}(\nu_{W \subset C})
\times 
\tSub^\xi(C)_{|W}
\xra{~\id \times (\ref{x34})~}
\Gamma^{\sf sm}(\nu_{W \subset C}) 
\times
\tSub^\xi(D)_{|\sigma^{-1}(W)}
~.
\]
Lemma~\ref{t82} states that~(\ref{x34}) is a Serre fibration.
Therefore, the map $\tSub^\xi(C) \to \square$ is, locally over $\square$, a Serre fibration.
Lemma~\ref{t83} implies the topological space $\square$ is metrizable, and therefore paracompact and Hausdorff. A local Serre fibration is Serre fibration if the base is paracompact and Hausdorff (e.g., see Theorem~13, \S2.7 of~\cite{spanier}), so the map $\tSub^\xi(C) \to \square$ is therefore a Serre fibration, as desired.

\end{proof}

Lemma~\ref{t77} justifies the following.
\begin{definition}\label{a22}
    The simplicial topological space $\tBord^\xi_\bullet(M)$ is
    \[
    \tBord^\xi_\bullet(M)
    ~:=~
    \tSub^\xi(\Delta^\bullet \times M)
    ~.
    \]
\end{definition}

Lemma~\ref{t88}(2)(c) immediately implies the following.
\begin{cor}\label{a25}
The underlying simplicial space of the simplicial topological space of Definition~\ref{a22} is identical with that of Definition~\ref{def.Bord.s.space}
\end{cor}

\subsection{The map $\dBord_\bullet^{\xi}(M)\ra \Bord_\bullet^{\xi}(M)$ is a trivial Kan fibration}
Here, we prove that the simplicial set $\dBord^\xi_\bullet(M)$ is a Kan complex, and that its geometric realization is identical with that of $\Bord^\xi_\bullet(M)$.

Let $C$ be a manifold with corners satisfying the property that the closure of each face is a manifold with corners in its own right.  
Let $C_0 \subset C$ be a closed union of faces of a manifold with corners.  
Consider the poset $\cP(C_0)$ in which an element is a face of $C$ that is contained in $C_0$, and whose partial order is given by $E \leq F$ if there is containment in the closure: $E \subset \ov{F}$. 
Note that, if $C_0 \subset C$ is the closure of a face, then $\cP(C_0)$ has a final object.  
The assumption on $C$ ensures, for $E \leq F$ in $\cP(C_0)$, that the inclusion between closures $\ov{E} \hookrightarrow \ov{F}$ is a proper face-embedding between manifolds with corners.  
Consider the functor
\[
\cP(C_0)^{\op}
\xra{~\Sub^k(-)~}
\Top^{\Sub^k(C)/}
\]
whose value on $F \in \cP(C_0)$ is the map $\Sub^k(C) \xra{(\ref{e56})}\Sub^k(\ov{F})$ which Lemma~\ref{t58} ensures is continuous. 
Functoriality of this assignment follows from the standard identity in set theory, $(\tau \circ \sigma)^{-1}(W) = \sigma^{-1}( \tau^{-1}(W))$.
This functor determines a canonical map to the limit topological space
\[
\Sub^k(C)
\longrightarrow
\underset{F \in \cP(C_0)^{\op}}
\limit
\Sub^k(\ov{F})
~=:~
\Sub^k(C_0)
~.
\]
Denote the underlying set of this limit as $\dSub^k(C_0)$.
Consider the commutative diagram among topological spaces,
\[
\begin{tikzcd}
	{\dSub^k(C)} & {\Sub^k(C)} \\
	{\dSub^k(C_0)} & {\Sub^k(C_0)}
	\arrow[from=1-1, to=1-2]
	\arrow[from=1-1, to=2-1]
	\arrow[from=1-2, to=2-2]
	\arrow[from=2-1, to=2-2]
\end{tikzcd}
~,
\]
in which the downward maps are given by intersection with $C_0$ and the rightward maps are the identity maps on underlying sets.  
This commutative diagram determines a map to the homotopy pullback:
\begin{equation}\label{e79}
\dSub^k(C)
\longrightarrow
\dSub^k(C_0) \underset{\Sub^k(C_0)} {\times^{\sf h}} \Sub^k(C)
~.
\end{equation}

\begin{lemma}\label{t75}
Let $C$ be a manifold with corners satisfying the property that the closure of each face is a manifold with corners in its own right.  
Let $C_0 \subset C$ be a closed union of faces.
The map~(\ref{e79}) is surjective on path-components.

\end{lemma}

\begin{proof}
Recall that the homotopy pullback of the diagram $\dSub^k(C_0)\ra \Sub^k(C_0) \la \Sub^k(C)$ among topological spaces may be defined as the limit of the diagram among topological spaces
\[\begin{tikzcd}
	\dSub^k(C_0) && {\Sub^k(C_0)^I} && \Sub^k(C) \\
	& \Sub^k(C_0) && \Sub^k(C_0)
	\arrow[""', from=1-1, to=2-2]
	\arrow["{\ev_0}", from=1-3, to=2-2]
	\arrow["{\ev_1}"', from=1-3, to=2-4]
	\arrow["", from=1-5, to=2-4]
\end{tikzcd}\]
where $\Sub^k(C_0)^I := \Map(I,\Sub^k(C_0))$ is the set of maps from the closed interval $I := [0,1]$ to $\Sub^k(C_0)$, endowed with the compact-open topology, and where $\ev_i$ is evaluation at $i\in \{0,1\}$.
Unpacking this description of the homotopy pullback in the case at hand lends to the following description of a point in the codomain of~(\ref{e79}):
\begin{itemize}
    \item[+] a compact codimension-$k$ submanifold
    \[
    \w{W} \subset C
    ~,
    \]
    
    \item[+] a fiber bundle over $I$ of compact codimension-$k$ submanifolds 
    \[
    W_I \subset I \times C_0
    ~;
    \]

\end{itemize}
such that the two subsets of $C_0 = \{0\} \times C_0$ agree:
\begin{itemize}
    \item[-] $\w{W} \cap C_0 = W_I \cap (\{0\} \times C_0)$~.
\end{itemize}
By smooth approximation, we may homotope the continuous map $I \to \Sub^k(C_0)$ classifying $W_I \subset I \times C_0$ to assume $W_I \subset I \times C_0$ is a submanifold.  

Through this description of a point in the codomain of the map~(\ref{e79}), this map is given by 
\[
(\w{W} \subset C)
\longmapsto 
\Bigl( ~(\w{W} \subset C)~,~(I \times (\w{W} \cap C_0))~ \Bigr)
~.
\]
Evidently, this map~(\ref{e79}) is injective, and its image consists of those $(\w{W},W_I)$ for which $W_I = I \times (\w{W} \cap C_0)$.

Now, let $(\w{W} \subset C , W_I \subset I \times C_0)$ be a point in the codomain of~(\ref{e79}).
In what follows, we will construct a compact subspace $\w{W}_I \subset I \times C$ that satisfies the following properties:
\begin{enumerate}

    \item The subspace
    \[
    \w{W}_I
    ~\subset~
    I \times C
    \]
    is a fiber bundle over $I$ of compact codimension-$k$ submanifolds of $I \times C$.

    \item The two subsets of $I \times C_0$ agree:
    \[
    W_I
    ~=~
    \w{W}_I \cap (I \times C_0) 
    ~.
    \]

    \item The two subsets of $C = \{0\} \times C$ agree:
    \[
    \w{W}
    ~=~
    \w{W}_I \cap (\{0\} \times C) 
    ~.
    \]
    
\end{enumerate}
Given such a $\w{W}_I \subset I \times C$, then $(\w{W}_I \cap (\{1\} \times C) \subset \{1\} \times C = C) \in \dSub^k(C)$, and $\w{W}_I$ witnesses a path in the codomain of~(\ref{e79}) from its image to the given element $(\w{W} \subset C , W_I \subset I \times C_0)$.
Therefore, the lemma is proved upon constructing such a $\w{W}_I \subset I \times C$.

Consider the normal vector bundle $\nu_{\w{W}}$ of $\w{W} \subset C$.
Through the tubular neighborhood theorem, choose an open neighborhood $\w{W} \subset \w{N} \subset C$ together with a diffeomorphism under $\w{W}$ from the total space $E(\nu_{\w{W}}) \cong \w{N}$.
Similarly, consider the normal vector bundle $\nu_{W_I}$ of $W_I \subset I \times C_0$.
Through the tubular neighborhood theorem, choose an open neighborhood $W_I \subset N_I \subset I \times C_0$ together with a diffeomorphism under $W_I$ and over $I$ from the total space $E(\nu_{W_I}) \cong N_I$.
Denote the compact codimension-$k$ submanifold $W_0 = \w{W} \cap C_0 = W_I \cap (\{0\} \times C_0) \subset C_0$.
Upon shrinking and reparametrizing as needed, we may assume $N_0 := \w{N} \cap C_0 = N_I \cap (\{0\} \times C_0) \subset C_0$ and the composite isomorphism 
\[
E(\nu_{W_0}) 
~\cong~
E(\nu_{\w{W}})_{|W_0}
~\cong~
\w{N} \cap C_0 = N_I \cap (\{0\} \times C_0) 
~\cong~ 
E(\nu_{W_I})_{|W_0} 
~\cong~ 
E(\nu_{W_0})
\]
is the identity.
Consider the compact codimension-$k$ submanifold $I \times \w{W} \subset I \times C$.
These data determine a tubular neighborhood $I \times \w{W} \subset I \times \w{N} \subset I \times C$ together with a diffeomorphism $E(\nu_{I\times \w{W}}) = I \times E(\nu_{\w{W}}) \cong I \times \w{N}$, which is 
under $I \times \w{W}$ and over $I$ from the total space $E(\nu_{W_I}) \cong N_I$.
These data, in turn, determine smooth retractions
\[
\pi_I
\colon 
N_I
\cong
E(\nu_{W_I})
\to
W_I
~,\qquad
\w{\pi}
\colon 
\w{N}
\cong
E(\nu_{\w{W}})
\to
\w{W}
~,\text{ and }\qquad
\pi_0
\colon 
N_0
\cong
E(\nu_{W_0})
\to
W_0
~,
\]
the last of which is a common restriction of each of the first two.

{\bf Containment Case.}
We now prove the lemma in the case that there is containment between subsets of $I \times C_0$,
\[
W_I 
~\subset~ 
I \times N_0 
~,
\]
and the resulting map over $I$,
\[
\pr \colon W_I 
~\hookrightarrow~ 
I \times N_0 
\xra{~I \times \pi_0~}
I \times W_0
~,
\]
is a diffeomorphism.
In this case, the composite map
\begin{equation}\label{x22}
I \times W_0
\xra{~\pr^{-1}~}
W_I
\hookrightarrow
I \times N_0
~\cong~
I \times E(\nu_{W_0})
~=~
E(\nu_{I \times W_0})
\end{equation}
is a smooth section of the normal vector bundle $\nu_{I \times W_0}$ over $I \times W_0 \subset I \times C_0$.
Consider the zero-section
\begin{equation}\label{x23}
\w{W}
\xra{~\rm zero~}
E(\nu_{\w{W}})
~,
\end{equation}
which is a smooth section of the normal vector bundle $\nu_{\w{W}}$ over $\w{W} \subset C$.
Observe that~(\ref{x22}) restricts along $W_0 \hookrightarrow I \times W_0$ as the zero-section.  
So the sections~(\ref{x22}) and~(\ref{x23}), together, define a smooth section 
of $\nu_{I \times \w{W}}$ over the closed union of faces $(I \times W_0) \cup \w{W} \subset I \times \w{W}$:
\begin{equation}\label{x24}
\w{W} \cup W_I
\xra{~(\ref{x22})\cup (\ref{x23})~}
E( \nu_{I \times \w{W}})
~.
\end{equation}
Now, using compactness of $I \times \w{W}$, choose a smooth vector bundle injection $\nu_{I \times \w{W}} \hookrightarrow \epsilon^D_{I \times \w{W}}$ into a high-rank trivial vector bundle over $I\times \w{W}$; choose, additionally, a smooth vector bundle retraction $\epsilon^D_{I \times \w{W}} \to \nu_{I \times \w{W}}$ to this injection.  
Lemma~\ref{lemma.sm.sing} then implies the section~(\ref{x24}) extends as a smooth section:
\begin{equation}\label{x25}
I \times \w{W}
\longrightarrow
E(\nu_{I \times \w{W}})
~.
\end{equation}
We then have a composite smooth embedding:
\begin{equation}\label{x26}
I \times \w{W}
\xra{~(\ref{x25})~}
E(\nu_{I \times \w{W}})
~\cong~
I \times \w{N}
~\subset~
I \times C
~.
\end{equation}
The image of~(\ref{x26}),
\[
\w{W}_I
~:=~
{\sf Image}(\ref{x26})
~\subset~
I \times C
~,
\]
is a compact codimension-$k$ submanifold of $I \times C$.
By construction, $\w{W}_I \cap (I \times C_0) = W_I \subset I \times C_0$ and $\w{W}_I \cap (\{0\} \times C) = \w{W} \subset C$.
Additionally, because $\w{W}_I$ is the image of the section of the normal vector bundle of the (trivial) fiber bundle $I \times \w{W} \subset I \times C$, then $\w{W}_I \subset I \times C$ is a fiber bundle over $I$ of compact codimension-$k$ submanifolds of $C$.
This completes the proof of the lemma in this Containment Case.

{\bf General Case.}
We now prove the lemma in the general case, which will reduce to the Containment Case.
Next, using that $W_I \subset I \times C_0$ is a fiber bundle over $I$ of compact codimension-$k$ submanifolds of $C_0$, choose $0=t_0<t_1<\cdots<t_R=1$ such that, for each $0<k\leq R$, the following conditions are satisfied.
\begin{enumerate}
    \item There containment between subsets of $[t_{k-1},t_k] \times C_0$:
    \[
    W_{[t_{k-1},t_k]}
    ~:=~
    W_I \cap \Bigl( \left[t_{k-1},t_k \right]\Bigr) \times C_0
    ~\subset~
    \left[t_{k-1},t_k\right] \times \Bigl(N_I \cap \Bigl(\left\{t_{k-1} \right\} \times C_0\Bigr) \Bigr)
    ~.
    \]

    \item The resulting composite map over $[t_{k-1},t_k]$,
    \[
    \pr_k
    \colon 
    W_{[t_{k-1},t_k]}
    \hookrightarrow
    [t_{k-1},t_k] \times ( N_I \cap ( \{t_{k-1}\} \times C_0 ) )
    \xra{~[t_{k-1},t_k] \times (\pi_I)_{|}~}
    [t_{k-1},t_k] \times ( W_I \cap ( \{t_{k-1}\} \times C_0 ) )
    \]
    is a diffeomorphism.
\end{enumerate}
We now inductively construct, for each $0\leq k\leq R$, a compact subspace $\w{W}_{[0,t_k]} \subset [0,t_k] \times C$ 
that satisfies the following properties:
\begin{enumerate}

    \item The subspace
    \[
    \w{W}_{[0,t_k]}
    ~\subset~
    [0,t_k] \times C
    \]
    is a fiber bundle over $[0,t_k]$ of compact codimension-$k$ submanifolds of $[0,t_k] \times C$.
    
    \item The two subsets of $[0,t_k] \times C_0$ agree:
    \[
    W_I \cap ([0,t_k] \times C_0)
    ~=~
    \w{W}_{[0,t_k]} \cap ([0,t_k] \times C_0) 
    ~.
    \]

    \item The two subsets of $C = \{0\} \times C$ agree:
    \[
    \w{W}
    ~=~
    \w{W}_{[0,t_k]} \cap (\{0\} \times C) 
    ~.
    \]
    
\end{enumerate}
Note that the case $k=R$ is, then, the sought data.
So completing this inductive construction completes the proof of the lemma.

Necessarily, $\w{W}_{[0,t_0]} = \w{W}$, which supplies the base case of $k=0$ for the inductive construction.
Now assume $k>0$, and such a $\w{W}_{[0,t_{k-1}]} \subset [0,t_{k-1}] \times C$ has been constructed.  
The Containment Case supplies a compact codimension-$k$ submanifold $\w{W}_{[t_{k-1},t_k]} \subset [t_{k-1},t_k] \times C$ that satisfies the following properties:
\begin{enumerate}

    \item The subspace
    \[
    \w{W}_{[t_{k-1},t_k]}
    ~\subset~
    [t_{k-1},t_k] \times C
    \]
    is a fiber bundle over $[t_{k-1},t_k]$ of compact codimension-$k$ submanifolds of $[t_{k-1},t_k] \times C$.

    \item The two subsets of $[t_{k-1},t_k] \times C_0$ agree:
    \[
    W_I \cap ([t_{k-1},t_k] \times C_0)
    ~=~
    \w{W}_{[t_{k-1},t_k]} \cap ([t_{k-1},t_k] \times C_0) 
    ~.
    \]

    \item The two subsets of $C = \{t_{k-1}\} \times C$ agree:
    \[
    \w{W}_{[0,t_{k-1}]} \cap ( \{t_{k-1}\} \times C)
    ~=~
    \w{W}_{[t_{k-1},t_k]} \cap (\{t_{k-1}\} \times C) 
    ~.
    \]
    
\end{enumerate}
Denote the union
\[
\w{W}_{[0,t_k]}
~:=~
\w{W}_{[0,t_{k-1}]} \cup \w{W}_{[t_{k-1},t_k]}
~\subset~
[0,t_k] \times C
~.
\]
This subspace $\w{W}_{[0,t_k]} \subset [0,t_k] \times C$ is a compact codimension-$k$ subspace that, by construction, satisfies the three named conditions.

\end{proof}

\begin{cor}\label{cor.dBord.Bord.kan.noxi}
    The map of simplicial spaces 
    \[
    \dBord_\bullet^k(M) \longrightarrow \Bord_\bullet^k(M)
    \]
    is a trivial Kan fibration.
\end{cor}
\begin{proof}
Surjectivity of path-components is exactly the statement of Lemma~\ref{t75} applied to $(C_0 \subset C) = (\partial \Delta^p \times M \subset \Delta^p \times M)$.

\end{proof}

We have following $\xi$-structured version of Lemma~\ref{t75}.

\begin{lemma}\label{t75xi}
Let $C$ be a manifold with corners satisfying the property that the closure of each face is a manifold with corners in its own right.  
Let $C_0 \subset C$ be a closed union of faces.
The map
\[
\dSub^\xi(C)
\longrightarrow
\dSub^\xi(C_0) \underset{\tSub^\xi(C_0)} {\times^{\sf h}} \tSub^\xi(C)
\]
is surjective on path-components.    

\end{lemma}
\begin{proof}
We have a commutative diagram
\[
\xymatrix{
\dSub^\xi(C)\ar[r]^-{\ov{s}}\ar[d]&\ar[d]\dSub^\xi(C_0) \underset{\Sub^\xi(C_0)} {\times^{\sf h}} \Sub^\xi(C)\\
\dSub^k(C)\ar[r]^-s&\dSub^k(C_0) \underset{\Sub^k(C_0)} {\times^{\sf h}} \Sub^k(C)
}
\]
Given the $\pi_0$-surjectivity the bottom horizontal arrow, to show $\pi_0$-surjectivity of the top arrow it suffices to check $\pi_0$-surjectivity of the fiber for each $(W\subset C)\in \dSub^k(C)$. An element

That is, it suffices to show $\pi_0$-surjectivity
\[
\Bigl\{\bigl(W\xra{g} B, g^\ast \xi \overset{\alpha}\cong \nu_{W\subset C}\bigr)\Bigr\}
\longrightarrow
\Bigl(\dSub^k(C_0) \underset{\Sub^k(C_0)} {\times^{\sf h}} \Sub^k(C)\Bigr)_{|s(W\subset C)}
\]
from the set of $\xi$-structures on $   (W\subset C)$ to the fiber of topright homotopy-pullback over the image $s(W\subset C)$. Choose an element in this fiber, which is indexed by the set of $\xi$-structures on $\{0\}\times W_0\subset \{0\}\subset C_0$, together with an extension to a $\xi$-structure on $I\times W_0\subset I\times C_0$ and a $\xi$-structure on $W\subset C$, agreeing on $\{1\}\times W_0\subset\{1\}\times C$. Since the right vertical map is a fibration by by Lemma~\ref{t88}(2)(b), this fiber is the space of sections (fixed on $\{0\}\times W_0$) of a fibration over the space
\[
I\times W_0\underset{\{1\}\times W_0}\bigcup \{1\}\times W~.
\]
The inclusion
\[
I\times W_0\underset{\{1\}\times W_0}\bigcup \{1\}\times W
\hookrightarrow
I\times W
\]
is an acyclic cofibration relative to $\{0\}\times W_0$, therefore there exists an extension of this $\xi$-structure to $I\times W \subset I\times C$. Restricting this to $\{0\}\times W\subset \{0\}\times C$ gives the desired lift in $\dSub^\xi(C)$.

\end{proof}

Applying Lemma~\ref{t75xi} to $(C_0 \subset C) = (\partial \Delta^p \times M \subset \Delta^p \times M)$ gives the following.
\begin{cor}\label{cor.dBord.Bord.kan}
The map of simplicial spaces 
    \[
    \dBord_\bullet^\xi(M) \longrightarrow \Bord_\bullet^\xi(M)
    \]
    is a trivial Kan fibration.
\end{cor}

\begin{remark}\label{rem.Reedy}
Reedy fibrancy has the following surprising property:
    If a simplicial topological space $X_\bullet$ is Reedy fibrant, then the natural map from the underlying simplicial set, endowed with the discrete topology,
    \[
    \delta X_\bullet \longrightarrow X_\bullet
    \]
    is a trivial Kan fibration of simplicial spaces. It thereby induces an equivalence of realizations: $|\delta X_\bullet| \simeq |X_\bullet|$. We expect that the simplicial topological space $\Bord_\bullet^{\xi}(M)$ is Reedy fibrant, which would provide a slightly different proof of Corollary~\ref{cor.dBord.Bord.kan}, but we have found it more convenient to prove the trivial Kan fibration property directly.
\end{remark}

\begin{cor}\label{cor.dBord.kan}
    The simplicial sets $\dBord^k_\bullet(M)$ and $\dBord_\bullet^\xi(M)$ are both Kan complexes.
\end{cor}

\begin{proof}
Consider the canonical morphisms between simplicial spaces:
\[
\dBord^k_\bullet(M)
\to 
\Bord^k_\bullet(M)
\qquad\text{ and }\qquad
\dBord^\xi_\bullet(M)
\to 
\Bord^\xi_\bullet(M)
~.
\]
Corollaries~\ref{cor.dBord.Bord.kan.noxi} and~\ref{cor.dBord.Bord.kan} imply these morphisms are Kan fibrations.
Corollary~\ref{cor.bord.kan} states that the codomains satisfy the Kan condition.
Now, observe that the composition of Kan fibrations is a Kan fibration;
observe that a simplicial space satisfies the Kan condition precisely if its unique map to the terminal simplicial space is a Kan fibration.
Consequently, the simplicial space $\dBord^\xi_\bullet(M)$ satisfies the Kan condition.
As the simplicial spaces $\dBord^k_\bullet(M)$ and $\dBord^\xi_\bullet(M)$ are simplicial sets, we conclude that they are both Kan complexes.

\end{proof}

Corollaries~\ref{cor.dBord.Bord.kan.noxi} and~\ref{cor.dBord.Bord.kan} enable Lemma~\ref{lemma.claim} to imply the following.
\begin{cor}\label{cor.dBord.Bord}
The forgetful morphism between simplicial spaces
    \[
    \dBord_\bullet^\xi(M) \longrightarrow \Bord_\bullet^\xi(M)
    \]
    induces an equivalence between geometric realizations
    \[
    \bigl|\dBord_\bullet^\xi(M)\bigr| 
    ~\simeq~ 
    \bigl|\Bord_\bullet^\xi(M)\bigr| 
    \underset{\rm Thm~\ref{t26}}{~\simeq~}
    \cMap(M,\Th(\xi)).
    \]
\end{cor}

\section{Appendix: vector bundles over topological categories}\label{sec.vbdl}

The function of this section is to give one solution to the following problem.
\begin{itemize}
    \item[]
    Let $\sC$ be a topological category.  
    \begin{itemize}
        \item Find explicit data representing a functor $\sC \to \BO(k)$ from its underlying $\infty$-category.

        \item Given such data $\chi$ and $\eta$, find explicit data representing an equivalence between their associated functors $\sC \to \BO(k)$.
    \end{itemize}
    
\end{itemize}
Here and throughout, for $\cX$ an $\infty$-category (such as $\spaces$ or the ordinary category $\Top$), the $\infty$-category of \bit{category-objects in $\cX$} is the full $\infty$-subcategory $\Cat[\cX] \subset \Fun(\bDelta^{\op} , \cX)$ consisting of those simplicial objects that satisfy the Segal conditions.  
For $\sC$ a category-object in $\cX$, we denote the two face maps
\[
{\sf s} , {\sf t} \colon 
\Mor(\sC) := \sC([1])
\longrightarrow
\sC([0]) =: \Obj(\sC)
~,
\]
and refer to them, respectively, as the source and the target morphisms of $\sC$.
A \bit{topological category} is a category-object $\sC \in \Cat[\Top]$ in $\Top$ with the property that the composite functor,
\[
\bDelta^{\op}
\xra{~\Hom([\bullet],\sC)~}
\Top
\xra{~\rm forget~}
\Spaces
~,
\]
is a Segal space: in other words, for each $r>0$, the canonical downward map
\[\begin{tikzcd}
	{\Hom([r],\sC)} & {\Hom(\{0<1\},\sC) \underset{\Hom(\{1\},\sC)} \times \cdots \underset{\Hom(\{r-1\},\sC)} \times \Hom(\{r-1<r\},\sC)} \\
	{\Hom^{\sf h}([r],\sC)} & {\Hom(\{0<1\},\sC) \underset{\Hom(\{1\},\sC)} {\times^{\sf h}} \cdots \underset{\Hom(\{r-1\},\sC)} {\times^{\sf h}} \Hom(\{r-1<r\},\sC)}
	\arrow["\cong"{description}, draw=none, from=1-1, to=1-2]
	\arrow[from=1-1, to=2-1]
	\arrow[from=1-2, to=2-2]
	\arrow["{:=}"{description}, draw=none, from=2-1, to=2-2]
\end{tikzcd}
~,
\]
is a weak homotopy equivalence.
Here, the superscript ${\sf h}$ denotes the homotopy fiber product.
Note that $\sC$ is a topological category if either the source or target map $\Mor(\sC) \to \Obj(\sC)$ is a Serre fibration.
If $\sC$ is a topological category, we denote the $\infty$-category underlying this Segal space again by $\sC$.

Consider the category
\[
{\sf VBun}^k
\]
in which an object is a rank-$k$ vector bundle, and a morphism is a Cartesian morphism between vector bundles (i.e., a morphism between vector bundles that is a fiberwise isomorphism).  
Reporting the base of a vector bundle defines a functor
\[
{\sf VBun}^k
\xra{~\rm base~}
\Top
~,
\]
which is a right fibration.
This functor preserves and creates finite limits.  
In particular, it carries category-objects in ${\sf VBun}^k$ to category-objects in $\Top$:
\begin{equation}
    \label{e50}
\Cat[{\sf VBun}^k]
\xra{~\Cat[{\rm base}]~}
\Cat[\Top]
~.
\end{equation}
Furthermore,~(\ref{e50}) is a right fibration.
For $\sC$ a category-object in $\Top$,     the groupoid 
    \[
    {\sf VBun}^k_\sC
    \]
    of \bit{rank-$k$ vector bundles over $\sC$} is the fiber of~(\ref{e50}) over $\sC$.

Let $\sC$ be a category-object in $\Top$.
The source and target maps of $\sC$ is a span of topological spaces:
\[
\Obj(\sC)
\xla{~{\sf s} = {\sf d}_1~}
\Mor(\sC)
\xra{~{\sf t} = {\sf d}_0~}
\Obj(\sC)
~.
\]
More generally, for $[p] \xra{\delta_i} [p+1]$ the inclusion of the complement of the $i \in [p+1]$, we denote the restriction map
\[
{\sf d}_i
\colon 
\sC([p+1])
\longrightarrow
\sC([p])
~.
\]
Note the identity ${\sf d}_i \circ {\sf d}_j = {\sf d}_{j=1} \circ {\sf d}_i$ for each $0\leq i < j \leq p$.

Now, explicitly, 
\begin{itemize}
    \item a rank-$k$ vector bundle over $\sC$ is
    \begin{itemize}
        \item a rank-$k$ vector bundle $\chi_0$ over the topological space $\Obj(\sC)$, 

        \item an isomorphism between vector bundles ${\sf s}^\ast \chi_0 \overset{\chi_1} \cong {\sf t}^\ast \chi_0$ over $\Mor(\sC)$;
    
    \end{itemize}
    such that the diagram among vector bundles over $\sC([2])$,
\begin{equation}\label{e51}
\begin{tikzcd}
	{{\sf d}_2^\ast {\sf d}_1^\ast\chi_0} & {{\sf d}_1^\ast {\sf d}_0^\ast\chi_0} && {{\sf d}_0^\ast {\sf d}_1^\ast\chi_0} & {{\sf d}_0^\ast {\sf d}_0^\ast\chi_0} \\
	& {{\sf d}_0^\ast {\sf d}_0^\ast\chi_0} && {{\sf d}_0^\ast {\sf d}_0^\ast\chi_0}
	\arrow["\cong"{description}, draw=none, from=1-1, to=1-2]
	\arrow["{{\sf d}_2^\ast \chi_1}"', from=1-1, to=2-2]
	\arrow["{{\sf d}_1^\ast \chi_1}", from=1-2, to=1-4]
	\arrow["\cong"{description}, draw=none, from=1-4, to=1-5]
	\arrow["\cong"{description}, draw=none, from=2-2, to=2-4]
	\arrow["{{\sf d}_0^\ast \chi_1}"', from=2-4, to=1-5]
\end{tikzcd}
~,
\end{equation}
commutes;

\item 
for $\chi = (\chi_0 , \chi_1)$ and $\eta = (\eta_0 , \eta_1)$ two such vector bundles over $\sC$, an isomorphism $\chi \cong \eta$ is 
   \begin{itemize}
        \item an isomorphism $\chi_0 \overset{\alpha}\cong \eta_0$ between vector bundles over $\Obj(\sC)$;
    
    \end{itemize}
    such that the diagram among vector bundles over $\Mor(\sC)$,
\begin{equation}\label{e52}
\begin{tikzcd}
	{{\sf s}^\ast\chi_0} && {{\sf t}^\ast\chi_0} \\
	{{\sf s}^\ast\eta_0} && {{\sf t}^\ast\eta_0}
	\arrow["{\chi_1}", from=1-1, to=1-3]
	\arrow["{{\sf s}^\ast\alpha}"', from=1-1, to=2-1]
	\arrow["{{\sf t}^\ast\alpha}", from=1-3, to=2-3]
	\arrow["{\eta_1}", from=2-1, to=2-3]
\end{tikzcd}
~,
\end{equation}
    commutes.
\end{itemize}

Provided $\sC$ is a topological category, forgetting to underlying $\infty$-categories defines a functor from the groupoid of vector bundles over $\sC$ to the $\infty$-groupoid of functors to $\BO(k)$:
\[
{\sf VBun}^k_\sC
\xra{~\rm forget~}
\Fun(\sC,\BO(k))
~.
\]


\begin{thebibliography}{99}


\bibitem[Ay]{phd} Ayala, David.
Geometric cobordism categories.
Thesis (Ph.D.)–Stanford University
ProQuest LLC, Ann Arbor, MI, 2009. 166 pp.

\bibitem[AF]{bord} Ayala, David; Francis, John. The cobordism hypothesis.
Preprint, 2017.
Available at https://arxiv.org/abs/1705.02240

\bibitem[AFR]{striation}
Ayala, David; Francis, John; Rozenblyum, Nick.
A stratified homotopy hypothesis.
J. Eur. Math. Soc. (JEMS) 21 (2019), no. 4, 1071–1178.


\bibitem[BK]{bousfield.kan}
Bousfield, A. K.; Kan, D. M.
Homotopy limits, completions and localizations
Lecture Notes in Math., Vol. 304
Springer-Verlag, Berlin-New York, 1972, v+348 pp.

\bibitem[BRS]{brs}
Buoncristiano, Sandro; Rourke, Colin; Sanderson, Brian.
A geometric approach to homology theory.
London Math. Soc. Lecture Note Ser., No. 18
Cambridge University Press, Cambridge-New York-Melbourne, 1976. iii+149 pp.


\bibitem[BLR]{blr}
Burghelea, Dan; Lashof, Richard; Rothenberg, Melvin.
Groups of automorphisms of manifolds.
With an appendix (``The topological category'') by E. Pedersen
Lecture Notes in Math., Vol. 473
Springer-Verlag, Berlin-New York, 1975. vii+156 pp.

\bibitem[CW]{CW}
Christensen, J. Daniel; Wu, Enxin.
The homotopy theory of diffeological spaces.
New York J. Math. 20 (2014), 1269–1303.

\bibitem[GMTW]{GMTW} Galatius, Søren; Tillmann, Ulrike; Madsen, Ib; Weiss, Michael.
The homotopy type of the cobordism category.
Acta Math. 202 (2009), no. 2, 195–239.


\bibitem[GH]{gepner.heine} Gepner, David; Heine, Hadrian. Oriented category theory. Preprint, 2025.

\bibitem[GIMM]{gimm}
Goodwillie, Thomas; Igusa, Kiyoshi; Malkiewich, Cary; Merling, Mona.
On the functoriality of the space of equivariant smooth h-cobordisms. Preprint available at https://arxiv.org/abs/2303.14892

\bibitem[HLLR]{HLLR}
Hebestreit, Fabian; Land, Markus; Lück, Wolfgang; Randal-Williams, Oscar.
A vanishing theorem for tautological classes of aspherical manifolds. Geom. Topol. 25 (2021), no. 1, 47–110.

\bibitem[Hi]{hirsch}
Hirsch, Morris.
Differential topology.
Corrected reprint of the 1976 original
Grad. Texts in Math., 33
Springer-Verlag, New York, 1994.

\bibitem[HM]{hirsch.mazur}
Hirsch, Morris; Mazur, Barry.
Smoothings of piecewise linear manifolds.
Ann. of Math. Stud., No. 80
Princeton University Press, Princeton, NJ; University of Tokyo Press, Tokyo, 1974.



\bibitem[Jo]{joyce}
Joyce, Dominic.
On manifolds with corners. Advances in geometric analysis, 225–258.
Adv. Lect. Math. (ALM), 21.
International Press, Somerville, MA, 2012.

\bibitem[KM]{KM.smoothapprox} Khokhliuk, Oleksandra; Maksymenko, Sergiy.
Smooth approximations and their applications to homotopy types.
Proc. Int. Geom. Cent. 13 (2020), no. 2, 68–108.

\bibitem[Ki]{kihara}
Kihara, Hiroshi.
Smooth homotopy of infinite-dimensional $C^\infty$-manifolds. Mem. Amer. Math. Soc. 289 (2023), no. 1436, vii+129 pp.

\bibitem[Kr]{krannich}
Krannich, Manuel.
A homological approach to pseudoisotopy theory. I.
Invent. Math. 227 (2022), no. 3, 1093–1167.

\bibitem[La]{lang}
Lang, Serge.
Real and functional analysis.
Third edition
Grad. Texts in Math., 142
Springer-Verlag, New York, 1993. 


\bibitem[LM]{laures.mcclure}
Laures, Gerd; McClure, James.
Multiplicative properties of Quinn spectra.(English summary)
Forum Math. 26 (2014), no. 4, 1117–1185.

\bibitem[Lu1]{HTT} Lurie, Jacob. Higher topos theory.
Ann. of Math. Stud., 170
Princeton University Press, Princeton, NJ, 2009. xviii+925 pp.

\bibitem[Lu2]{lurie.kan.fibration} Lurie, Jacob. Simplicial Spaces (Lecture 7). Notes from Algebraic L-theory and Surgery (287x), 2011.

\bibitem[Mc1]{mcduff1}
McDuff, Dusa. Configuration spaces of positive and negative particles. Topology 14 (1975), 91–107.

\bibitem[Mc2]{mcduff2}
McDuff, Dusa.
Configuration spaces.K-theory and operator algebras (Proc. Conf., Univ. Georgia, Athens, Ga., 1975), pp. 88–95.
Lecture Notes in Math., Vol. 575.
Springer-Verlag, Berlin-New York, 1977.

\bibitem[OT]{oh.hiro}
Oh, Yong-Geun; Tanaka, Hiro Lee.
Smooth constructions of homotopy-coherent actions.
Algebr. Geom. Topol. 22 (2022), no. 3, 1177–1216.


\bibitem[Pa]{palais} Palais, Richard.
Local triviality of the restriction map for embeddings.
Comment. Math. Helv. 34 (1960), 305–312.


\bibitem[Qu1]{quinn.thesis}
Quinn, Frank.
A geometric formulation of surgery.
Thesis (Ph.D.)–Princeton University.
ProQuest LLC, Ann Arbor, MI, 1970. 80 pp.

\bibitem[Qu2]{quinn} Quinn, Frank. Assembly maps in bordism-type theories. Novikov conjectures, index theorems and rigidity, Vol. 1 (Oberwolfach, 1993), 201–271.
London Math. Soc. Lecture Note Ser., 226.
Cambridge University Press, Cambridge, 1995.

\bibitem[RW1]{oscar} Randal-Williams, Oscar.
Embedded cobordism categories and spaces of submanifolds.
Int. Math. Res. Not. IMRN 2011, no. 3, 572–608.

\bibitem[RW2]{oscar2} Randal-Williams, Oscar.
The family signature theorem.
Proc. Roy. Soc. Edinburgh Sect. A 154 (2024), no. 6, 2024–2067.

\bibitem[RS]{raptis.steimle} Raptis, George; Steimle, Wolfgang.
Parametrized cobordism categories and the Dwyer-Weiss-Williams index theorem.
J. Topol. 10 (2017), no. 3, 700–719.


\bibitem[SP]{chrisSP} Schommer-Pries, Christopher.
Invertible topological field theories.
J. Topol. 17 (2024), no. 2.

\bibitem[Sp]{spanier}
Spanier, Edwin.
Algebraic topology.
Corrected reprint of the 1966 original
Springer-Verlag, New York, [1995]. xvi+528 pp.

\bibitem[Se]{segal}
Segal, Graeme.
The topology of spaces of rational functions.
Acta Math. 143 (1979), no. 1-2, 39–72.



\bibitem[Wh]{whitney}
Whitney, Hassler.
Analytic extensions of differentiable functions defined in closed sets.
Trans. Amer. Math. Soc. 36 (1934), no. 1, 63–89.

\bibitem[Wi]{williams} Williams, Lucas.
The Pontryagin--Thom theorem for families of framed manifolds.
Preprint, 2025:
https://arxiv.org/pdf/2503.20818





\end{thebibliography}
\end{document}